\documentclass[12pt]{amsart}

\usepackage[utf8]{inputenc}

\usepackage{amscd}

\usepackage{enumitem}

\usepackage{float}

\usepackage{verbatim}
\usepackage{hyperref}
\usepackage{blkarray}
\usepackage{amssymb}
\usepackage{amsmath}
\usepackage{amsthm}
\usepackage{graphicx}
\usepackage{tikz-cd}
\usepackage{pdfsync}

\usepackage[bordercolor=white,backgroundcolor=gray!30,linecolor=black,colorinlistoftodos]{todonotes}

\newtheorem{fakt}{fakt}[]

\numberwithin{fakt}{section}

\newtheorem{Corollary}[fakt]{Corollary}

\newtheorem{Lemma}[fakt]{Lemma}

\newtheorem{Theorem}[fakt]{Theorem}

\theoremstyle{definition}

\newtheorem{example}[fakt]{Example}

\newtheorem{Example}[fakt]{Example}

\newtheorem{Exampleintro}{Example}

\newtheorem{Remark}[fakt]{Remark}

\newtheorem{Question}[fakt]{Question}

\newtheorem{definition}[fakt]{Definition}

\newtheorem{Definition}[fakt]{Definition}

\newcommand{\N} {{\mathbb N}} \newcommand{\Z} {{\mathbb Z}}   \newcommand{\C} {{\mathbb C}}

\newcommand{\Spec}{\operatorname{Spec}}

\newcommand{\Stab}{\operatorname{Stab}}
\newcommand{\Sym}{\operatorname{Sym}}
\newcommand{\Syz}{\operatorname{Syz}}

\newcommand{\idealp}{ {\mathfrak p} }
\newcommand{\idealq}{ {\mathfrak q} }
\newcommand{\idealm}{ {\mathfrak m} }

\newcommand{\tensor}{\otimes}

\newcommand{\vandermonde}{D}

\newcommand{\plusdots}{+ \ldots +}

\newcommand{\SLG}{\operatorname{Sl} }

\newcommand{\GLG}{\operatorname{Gl} }
\newcommand{\dual}{*}

\newcommand{\coker}{\operatorname{coker}}

\begin{document}

\title{Module schemes in invariant theory}

\author[Holger Brenner]{Holger Brenner}
\address{Holger Brenner\\
Institut für Mathematik\\
Universität Osnabrück\\
Albrechtstr. 28a\\
49076 Osnabrück
}
\curraddr{}
\email{hbrenner@uni-osnabrueck.de}
\thanks{}

\begin{abstract}
Let $G$ be a finite group acting linearly on the polynomial ring with invariant ring $R$. If the action is small, then a classical result of Auslander gives in dimension two a correspondence between linear representations of $G$ and maximal Cohen-Macaulay $R$-modules. We establish a correspondence  for all linear actions between representations and objects over the invariant ring by looking at quotient module schemes (up to modification) instead of the modules of covariants.
 \end{abstract}

\keywords{Linear representation, invariant theory, quotient scheme, module scheme}

\subjclass[2020]{13A50, 13C13, 14L15, 14L24, 14J60, 16T05, 20C99}

\maketitle

\section*{Introduction}

We fix a field $K$. Let $G$ be a finite group and fix a linear representation $\beta:G \rightarrow \GLG_d (K)$. This corresponds to a linear action of $G$ on the polynomial ring $S=K[X_1 , \ldots, X_d]$, respectively, on the scheme ${\mathbb A}^d$, and its invariant ring $R=S^G$ (rather $S^\beta$) and the quotient scheme  $X={\mathbb A}^d/ \beta = \Spec R $. Suppose moreover that $\beta$ is faithful. We consider $\beta$ as the basic representation and want to understand what objects on $X$ are defined by other linear representations of $G$.

Let $\rho$ be another linear representation of $G$ in $ \GLG_m (K)$ with the corresponding linear action on a vector space $V \cong K^m$ or on the affine space ${\mathbb A}^m$. This defines the module of covariants $( S \tensor_K V)^G $ ($G$ acting on both tensor components) which is a maximal Cohen-Macaulay module over $R$. If $G$ is a \emph{small} group, meaning that $\beta(G) \subseteq \GLG_d(K)$ does not contain any (pseudo-)reflection (a linear mapping $\neq \operatorname{Id}$ that fixes a hyperplane) and \emph{nonmodular}, meaning that its order is not a multiple of the characteristic of the field, then we have the following correspondence (see \cite[Corollary 5.20]{leuschkewiegand}).

\begin{enumerate}
\item linear irreducible representations $(V,\rho)$ of $G$.

\item indecomposable $R$-modules of the form  $( S \tensor_K V)^G $, which are a direct summand of $S$ as an $R$-module.
\end{enumerate}

This correspondence is particularly nice in dimension two, where every maximal Cohen-Macaulay module of $R$ is of the form described in (2). This correspondence has been studied intensively from various perspectives and under various conditions, studying finite Cohen-Macaulay type, looking at the resolution of the singularities of $\Spec R$, using the Mc Kay quiver, applying noncommutative algebra and non commutative resolutions, working with matrix factorizations and derived categories, etc., see \cite{drozdroiter} \cite{herzogfinitely}, \cite{esnaultreflexive}, \cite{artinverdier}, \cite{auslanderrationalsplit}, \cite{auslanderreitencohenmacaulay}, \cite{mckay}, \cite{bridgelandkingreid}, \cite{vandenberghcrepant}, \cite{burbandrozd} \cite{buchweitzfaberingalls}.

However, if $G$ is a reflection group (generated by reflections), then the invariant ring $S^G$ is itself a polynomial ring by the theorem of Chevalley, Shephard, Todd and Serre, and therefore the only maximal Cohen-Macaulay modules are the free ones. Hence for all these groups, such a correspondence cannot work, as every representation $\rho$ of dimension $m$ yields always the free module $R^m$.

In this paper, we develop another perspective on the question of which object on $\Spec R$ should correspond to $\rho$. The two representations $\beta$ and $\rho $ define the product representation
\[ \beta \times \rho: G \longrightarrow \GLG_d(K) \times \GLG_m(K) \]
and the corresponding product action of $G$ on $ {\mathbb A}^d \times {\mathbb A}^m $. We denote the quotient scheme
\[ Z_\rho := ( {\mathbb A}^d \times {\mathbb A}^m)  / \beta \times \rho   \]
considered as a scheme over $X$. This is quite a natural approach, as it deals on both sides, the basic representation $\beta$ and the varying representation $\rho$, with the same construction. On the ring level this means that we look at the invariant algebra $K[X_1 , \ldots, X_d ,W_1 , \ldots, W_m]^{\beta \times \rho}$ as an algebra  over the invariant ring  $K[X_1 , \ldots, X_d ]^\beta$.

The gain of this perspective can be seen in the simplest examples.

\begin{Exampleintro}
Let $G = \Z/(2)$ act on $K[X]$ with the nontrivial element acting by negation. The invariant ring is $K[X^2]=K[A]$, which is again a polynomial ring. The trivial action of $G$ on $K[W]$ yields the invariant algebra $K[A,W]$, just the polynomial algebra over the invariant ring. However, the nontrivial linear action of $G$ on $K[W]$ yields the invariant algebra $ K[X^2,W^2,XW]=K[A,B,C]/(AB-C^2) $ over $K[A]$. We see that this algebra has an isolated singularity. Moreover, the fiber of its spectrum over $ \Spec K[A] $ is for a point where $ A \neq 0 $ just an affine line, represented by $\Spec K[C]$, but over the point where $A=0$, the fiber is $ K[B,C]/(C^2) $, a nonreduced affine line. Hence we see a big schemetheoretic (though settheoretic neglectable) difference between the two quotient schemes corresponding to the two representations. Note also that, for $K= \C$, the fundamental group of the punctured second quotient scheme is the group itself.
\end{Exampleintro}

We mention the following features of our approach which are not available when only looking at the module of covariants. The difference is in particular strong in the case of a reflection group.

\begin{enumerate}
	
\item
The methods of invariant theory for rings apply directly to invariant algebras/quotient schemes.
	
\item 
The properties of $Z_\rho$ as a $K$-scheme are relevant as well as the properties of the morphism $Z_\rho \rightarrow X$ ($K$-algebra versus algebra over the invariant ring). See Example \ref{cyclicone5} and Example \ref{cyclicone5fiber}.

\item 
The fibers of $Z_\rho \rightarrow X$ have geometric properties that reflect properties of the group representation (Theorem \ref{fiberdescription}).

\item
The quotient schemes exhibit ramification behavior that mirrors properties of the group representation (Theorem \ref{nonreducedfiber}).
 
\item 
 The dimension of the fibers of $Z_\rho$ is constant, which helps to distinguish them from other Cohen-Macaulay modules (Section \ref{fiberflatsection}, Section \ref{regularfreesection}).
 
\item
The reflection properties of the basic representation and of the product representation have parallel but not identical impacts (Section \ref{reflectionssection}).

\item
The singularities of the quotient schemes reflect properties of the group representation (Theorem \ref{fibersingularonepoint}, Theorem \ref{fibersingular}).

\item
Ring-theoretic constructions (in particular, the normalization) allow the reconstruction of the group representation from the quotient schemes (Theorem \ref{pullback}, Corollary \ref{correspondence}).

\item
The fundamental group (for $K=\C$) of the regular loci of the quotient schemes helps to reconstruct the group and its action (Corollary \ref{faithfulfundamental}, Theorem \ref{moduleschemefundamentalgroup}, Theorem \ref{reconstruction}).
\end{enumerate}

Other properties are quite parallel to the situation where one works with modules of covariants.

\begin{enumerate}

\item
The product of representations corresponds to the (normalization of the) product of quotient schemes (Corollary \ref{productcompatibility}).

\item 
The regular representation, which contains all irreducible representions is reflected in the decomposition of the regular invariant algebra (Theorem \ref{regularrepresentation}).

\item
The module of covariants $(S \tensor V)^G$ and the invariant algebra $(S[V])^G$ should be considered more generally for any action of $G$ on a $K$-algebra $S$ by $K$-algebra automorphisms (Section \ref{quotientsection}, Section \ref{normalsubgroupsection}).

\item
In the small case, the symmetric algebra of the module of covariants and the invariant algebra coincide over the regular locus of $X$ (Section \ref{reflexivesection}).

\item 
In the small case, irreducible representations correspond to irreducible fiberflat bundles (Section \ref{irreduciblesection}).

\end{enumerate}

A recurring feature is the comparison of the symmetric algebra of the module of covariants and the invariant algebra.  There is a natural ring homomorphism $ \Sym_R (( S \tensor V)^G) \rightarrow  ( S \tensor K[V])^G  $, see Lemma \ref{modulealgebra}, Corollary \ref{modulealgebraisomorphism}, Lemma \ref{modulealgebrascheme}. The comparison between  $\Spec ( \Sym_R (( S \tensor V)^G) )$ and the quotient scheme   $Z_\rho= \Spec ( S \tensor K[V])^G $ as schemes over $X$ is enhanced by the following concepts: the spectrum of the symmetric algebra is a  module scheme (Section \ref{moduleschemesection}), which is basically a scheme-theoretic version of a module, considered in passing by Grothendieck in \cite[1.7.13]{ega2}. A module scheme $M$ is a commutative group scheme with an addition $\alpha :M \times M\rightarrow M$ and also a scalar multiplication $ {\mathbb A}^1 \times M \rightarrow M$ fulfilling the usual axioms. Typical examples are geometric vector bundles. In contrast, for invariant algebras, the coaddition on the polynomial ring $ S \tensor K[V]  \rightarrow   S \tensor K[V \oplus V ] $ induces a natural ring homomorphism $ ( S \tensor K[V])^G   \rightarrow   ( S \tensor K[V \oplus V ])^G   $. The latter algebra contains the tensor product $  ( S \tensor K[V])^G \tensor_{S^G}  ( S \tensor K[V])^G  $, but the coaddition does not land in the tensor product. In geometric language, there are morphisms
\[ \begin{matrix}  Z_{\rho \times \rho}    & \longrightarrow & Z_ \rho \\ \downarrow & ... \nearrow & \\ { Z_\rho \times_X Z_\rho} \, ,  &      &   \end{matrix} \]
but the dotted arrow which is needed to talk about an addition on $Z_\rho$ is not a morphism. However, when $Y=\Spec S$ is normal, then the vertical arrow is the normalization (see Theorem \ref{normalproductcompatibility}), hence the addition exists as a birational morphism, and it exists as a morphism on the normalization of the tensor product. Thus we introduce module schemes up to normalization (Section \ref{modulenormalizationsection}) which is a very natural concept within invariant theory.

We provide an overview of the paper and some indications of how one may read it. In Section \ref{settingsection}, we describe the situation in the generality we will be working in, and we fix notation and conventions. Section \ref{lemmatasection} collects some lemmata of invariant theory for which we could not find an adequate reference. In Section \ref{toricsection}, we study the toric case with many examples (many relevant phenomena already occur in this case), and in Section \ref{examplessection} nontoric examples, to which we will come back as the theory evolves. Section \ref{lemmatasection} to Section \ref{examplessection} might be skipped first and consulted when needed.

Section \ref{modulealgebrasection} deals with the direct relation between the symmetric algebra of the module of covariants and the invariant algebra. There is a ring homomorphism from the first to the second (Lemma \ref{modulealgebra}), which is an isomorphism over the open subset $U \subseteq X$ coming from the fixed point free locus and where they both define the same vector bundle (Corollary \ref{modulealgebraisomorphism}). Outside of $U$, we have almost never an isomorphism, and the quotient scheme contains much more subtle information. In Section \ref{fiberssection}, we describe the fibers of the quotient scheme $Z_\rho \rightarrow X$ and in Section \ref{ramificationsection}, we determine which fibers are reduced. To do this, the restriction of the representation $\rho$ to the stabilizers of the basic action $\beta$ is important (Theorem \ref{fiberdescription}, Theorem	\ref{nonreducedfiber}). In Section \ref{productssection}, we look at the product of two representations $\rho_1$ and $\rho_2$ (with basic representation $\beta$ fixed) and study the relation between $Z_{\rho_1 \times \rho_2}$ and $Z_{\rho_1} \times_X Z_{\rho_2}$. Theorem \ref{normalproductcompatibility} shows that, under the condition that $Y$ is normal, the quotient scheme of the product action is the normalization of the product of the quotient schemes. In Section \ref{regularrepresentationsection}, we apply this approach to understand the structure of the quotient scheme of the regular representation (the reader more interested in the geometric properties of the quotient schemes might jump directly to Section \ref{singularitiessection}).

In Section \ref{moduleschemesection} to Section \ref{similarsection}, we draw our attention on the question of what kind of object the quotient schemes $ Z_\rho \rightarrow X $ are. In Section \ref{moduleschemesection}, we have a look at the notion of a module scheme, which is a commutative group scheme together with an action of ${\mathbb A}^1$ on it fulfilling natural conditions. The spectrum of the symmetric algebra of a module is a module scheme. Section \ref{moduleschemehopfgradedsection} deals with the Hopf-side of a module scheme. It turns out that the algebras describing a module scheme are standard-graded and that we can use a theorem of Milnor and Moore to deduce a structure result for module schemes over a ring containing a field of characteristic zero (Theorem \ref{moduleschememilnormoore}). In Section \ref{modulenormalizationsection}, we weaken the condition of a module scheme to a module scheme up to normalization, which fits the quotient schemes $Z \rightarrow X$ that we encounter in invariant theory. There is an action of  ${\mathbb A}^1$ on $Z$ as before, but, as in the diagram above, the addition is not defined on $Z \times_XZ$, but rather on a finite birational extension of it. In Section \ref{linearsection}, we define linear mappings between module schemes up to normalization and linear actions of a finite group $G$ on a module scheme up to normalization, in Section \ref{quotientsection}, we address the corresponding quotients and in Section \ref{normalsubgroupsection}, we show how these quotients behave in the presence of a normal subgroup of $G$. In Section \ref{similarsection} we collect some observations on objects where similar module schemes up to modification occur (Rees algebra, reflexive symmetric algebra, tangent bundle on a resolution).

Section \ref{fiberflatsection} to Section \ref{regularfreesection} describe further properties of the quotient schemes. In Section \ref{fiberflatsection}, we introduce fiberflat bundles, a property which quotient schemes share, in contrast to the spectra of symmetric algebras. Section \ref{reflexivesection} deals with the case where the locus where the bundle is a vector bundle contains all points of codimension one, which holds for quotient schemes for a small action. In this case, one can translate back and forth between invariant modules and invariant algebras, though the objects are still quite different. It turns out that the  so-called factorial closure (see \cite{vasconcelosarithmetic}) is the ring of global sections of the total space of the vector bundle. In Section \ref{regularfreesection} we give criteria to assert that a reflexive bundle does not have a fiberflat realization. 

Section \ref{reflectionssection} deals with relations between reflections of the basic representation and reflections of a product representations. This is helpful in Section \ref{singularitiessection} to understand where the singularities of a quotient scheme $Z$ are located and how they are related to the singularities of $X$. In Section \ref{fundamentalgroupsection}, we compute for $K=\C$ the fundamental group of the regular locus of $Z_\rho$ in terms of the acting group $G$ and the reflection subgroup of the product representation (Theorem \ref{moduleschemefundamentalgroup}).

In Section \ref{pullbacksection} to Section \ref{trivializationsection}, we deal with the question of in what sense we can reconstruct the original representation $\rho$ and the linear action on $ {\mathbb A}^d \times {\mathbb A}^m$ from the quotient scheme $Z_\rho \rightarrow X$. In Section \ref{pullbacksection}, we show that for a given action of $G$ on a normal scheme $Y$ with quotient $X$, we can reconstruct the linear action of $G$ on ${\mathbb A}^m$ from the quotient scheme $Z_\rho$ by pulling it back along $Y \rightarrow X$ and normalizing the pull-back (Theorem \ref{pullback}). Hence, we do not lose information when going from the linear action to the quotient schemes. In Section \ref{correspondencesection}, we extend this to obtain in Corollary \ref{correspondence}, for a basic action on a normal affine scheme $Y$ with a unique fixed point, a correspondence between linear representations, linear actions on $Y \times {\mathbb A}^m$ and quotient schemes on $X$.
Section \ref{irreduciblesection} proves that in the small case, this correspondence translates irreducible representations into irreducible reflexive fiberflat bundles. In Section \ref{trivializationsection}, we consider fiberflat bundles $Z$ on $X$ (without knowing $Y$) and try to find conditions that help to reconstruct a group $G$, a scheme $Y$ and a linear action on $Y \times {\mathbb A}^m $. For this, we use the fundamental group of the regular loci of $X$ and $Z$. We close in Section \ref{questionssection} with some questions.

I thank Janusz Adamus, Severin Barmeier, Marc Chardin, Hailong Dao, David Eisenbud, Eleonore Faber, Sudarshan Gurjar, Robin Hartshorne, Mitsuyasu Hashimoto, Craig Huneke, Najib Idrissi-Kaitouni, Colin Ingalls, Christian Liedtke, Ines Melzer, Claudia Miller, Rosa Mir\'o-Roig, Mandira Mondal, Julia Pevtsova, Peter Schenzel, Stefan Schr\"oer, Anurag Singh, Ilya Smirnov, Bernd Sturmfels, Peter Symonds, Shunsuke Takagi, Vijaylaxmi Trivedi, Bernd Ulrich and Keiichi Watanabe for useful discussions, comments and references.

This material is based upon work supported by the National Science Foundation under Grant No. DMS-1928930 and by the Alfred P. Sloan Foundation under grant G-2021-16778, while the author was in residence at the Simons Laufer Mathematical Sciences Institute (formerly MSRI) in Berkeley, California, during the Spring 2024 semester.

\section{The setting}
\label{settingsection}

Let $K$ be a field and $S$ be a $K$-algebra. Let $G$ be a finite group acting faithfully on $S$ by $K$-algebra automorphisms. The main case for us is when $S=K[X_1, \ldots, X_d]$ is a polynomial algebra and where the action is $K$-linear. However, to allow natural base changes and intermediate quotients, it is better to allow any $K$-algebra $S$. We write the geometric action on $\Spec S$ from the left, hence $(gh)Q)=g(h(Q))$ for $g,h \in G$ (where we consider $g$ as the corresponding automorphism), and the action of $G$ on the ring from the left, so that $fg$ corresponds to the composition $\Spec S \stackrel{g}{\rightarrow} \Spec S \stackrel{f}{ \rightarrow} {\mathbb A}^1$. The identity $(fg)h=f(gh)$ is clear, as both sides are $f \circ g \circ h$. When we want to stress the action instead of the group, we sometimes write $g^\beta (Q)$, in the sense that the group element $g$ has to be interpreted as the automorphism $g^\beta$ determined by $\beta$ (and then applied to the point $Q$).

Let $\rho$ be an $m$-dimensional $K$-linear \emph{representation} of $G$, i.e., a group homomorphism $G \rightarrow \GLG_m(K)$, or a linear action of $G$ on $K^m$. This gives rise to a linear action of $G$ on $K[W_1, \ldots, W_m]$, on $S[W_1, \ldots, W_m]$ and on $Y \times {\mathbb A}^d$ (action on both components). In particular, if $S$ is itself a polynomial ring, we denote the basic representation by $\beta$ and then $\beta \times \rho$ acts on $K[X_1, \ldots, X_d,W_1, \ldots, W_m]$.

Let $Y=\Spec S$. The \emph{invariant ring} $R=S^G=S^\beta$ is the subring of $S$ consisting of all the elements of $S$ that are fixed by the action. They are the basic objects of invariant theory (for finite groups), for background on invariant theory, see \cite{benson}, \cite{derksenkemper}. $X=\Spec R$ is the \emph{quotient} of $Y$ under the group action, i.e., the morphism $\Spec S \rightarrow \Spec R$ is invariant, it is a finite morphism, and the points of $\Spec R$ correspond to the $G$-orbits of $\Spec S$. If $Y$ is affine space, then we call the image point in $X$ of the origin the \emph{vertex point} of $X$.

The action of $G$ on $Y \times {\mathbb A}^m$ via $\beta \times \rho$ gives in the same way rise to a \emph{quotient scheme}
\[ Z_\rho = (Y \times {\mathbb A}^m  )/\beta \times \rho \, ,\]
which is the spectrum of the \emph{invariant algebra} (we talk about invariant algebras in contrast to the basic invariant ring) $ B^\rho = (S[W_1, \ldots, W_m])^{\beta \times \rho}$. Since we consider the action $\beta$ to be fixed, we usually only write $\rho$. As the morphism $ Y \times  {\mathbb A}^m  \rightarrow Y  $
is $G$-equivariant, we get a commutative diagram
\[ \begin{matrix} Y & \longleftarrow & Y \times  {\mathbb A}^m \\
\downarrow & & \downarrow \\ X & \longleftarrow &  Z_\rho \, . \end{matrix} \]
The vertical morphisms are quotient morphisms for group actions, and the horizontal morphisms are bundle-like morphisms (which has to be made precise for the lower morphism, see Section \ref{modulenormalizationsection}). These objects $Z_\rho \rightarrow X$ are the main objects of this paper. More precisely, for $G$, $Y$ and $X$ fixed, we want to understand how the $Z_\rho$ reflect properties of the (varying) representations $\rho$. This is also the reason why, in the second component, we stick to linear actions on a polynomial ring. The schemes $Z_\rho$ are always normal, as this is true for all quotient schemes of a normal scheme under a finite group, and the invariant algebra is always an $\N$-graded (not standard-graded) Cohen-Macaulay (in the \emph{nonmodular} case, i.e., if the group order is not $0$ in $K$, see  \cite[Proposition 13]{hochstereagon}) and of finite type over $K$ if $Y$ is. If the representation $\rho $ is trivial, then $Z_\rho =X \times {\mathbb A}^m $ is the trivial vector bundle over $X$ ($m=0$ yields the zero-bundle $X \rightarrow X$).

The invariant algebra $ B^\rho = K[X_1, \ldots, X_d,W_1, \ldots, W_m]^{\beta \times \rho } $ (or, more generally, $S[W_1, \ldots, W_m]^{\beta \times \rho } $)  is not so easy to compute, as is the case for invariant rings in general. The knowledge of the invariant rings $  K[X_1, \ldots, X_d ]^{\beta  }  $ and $ K [ W_1, \ldots ,W_m]^{ \rho }  $ does not suffice for this question. Of course,  $  K[X_1, \ldots, X_d ]^{\beta  } \tensor K [ W_1, \ldots, W_m]^{ \rho }  $ is a $K$-subalgebra of $B^\rho$, consisting of the $X$-invariants and the $W$-invariants together. However, there are also \emph{mixed invariants}, some of which come from characters. A \emph{character} for $G$ is a group homomorphism $\chi: G \rightarrow K^\times$, i.e., a one dimensional representation. For an action of the group $G$ on a $K$-algebra $S$ and a given character $\chi$, the \emph{module of semi-invariants} is
\[ S^{G, \chi} =\{   f \in S| fg = \chi(g) \cdot f \text{ for all } g \in G   \} \, .\]
We denote by $\chi^{-1}$ the character given by a character $\chi$ followed by the inversion in the field. A tensor without a subscript always refers to the base field.

\begin{Lemma}
\label{charactercontribution}
Let $G$ be a finite group acting on $K$-algebras $S$ and $T$ via $\beta$ and $\rho$ through $K$-algebra automorphisms. We then have the inclusion
\[ \bigoplus_\chi  S^{\beta, \chi} \tensor T^{\rho, \chi^{-1} } \subseteq (S \tensor T)^{ \beta \times \rho}     \, ,\]
where the sum runs over the characters $\chi$ of $G$.
\end{Lemma}
\begin{proof}
The inclusions are clear since for $f_1 \in  S^{\beta, \chi}  $, $ f_2 \in T^{\rho, \chi^{-1} }  $ and $g \in G$ we have
\[ ( f_1 \tensor f_2)  g^{\beta \times \rho}  =  f_1 g^\beta \tensor f_2 g^\rho  = \chi (g) \cdot  f_1 \tensor ( \chi (g))^{-1} \cdot  f_2  =f_1 \tensor f_2 \, .   \]
\end{proof}

In the toric situation, we will see in Lemma \ref{Toriccharacterdecomposition} that we have an equality in the previous statement.

\begin{Example}
In general, there are also elements in the invariant algebra that do not stem from characters. Let $G$ be the binary icosahedral group with the action $\beta$ on $K[U,V]$ yielding the $E_8$-singularity. It is known that this group has no nontrivial character, which is the reason why this singularity is factorial. We look at the product of $\beta$ with itself. The only contribution from Lemma \ref{charactercontribution} is the tensor product $K[U,V]^G \tensor K[W,Z]^G $. However, the invariant algebras are always over a non empty open subset of $X=K[U,V]^G $ a polynomial algebra (by Corollary \ref{modulealgebraisomorphism}) and hence generically smooth.	
\end{Example}

\section{Some lemmata}
\label{lemmatasection}

We provide some well-known lemmata that we need in the following and where we could not find a source of appropriate generality. We use the language of group schemes and consider a finite (usual) group $G$ as the group scheme $\Spec K^{|G|}$, which consists of $|G|$ closed $K$-points, endowed with the group scheme structure coming from the group.

The following is basically proven in \cite[Theorem 1.1]{GIT} for reductive groups in char $0$ (see \cite[Amplification 1.3]{GIT} for the nonmodular case in positive characteristic) to give a universal geometric quotient, but we need base change on the quotient itself (for finite groups).

\begin{Lemma}
\label{actionuniversal}
Let $G$ be a finite nonmodular group acting on $S$ by $K$-algebra automorphisms with invariant ring $R=S^G$. Let $A$ be an $R$-algebra. Then
\[  (A \tensor_R S)^G \cong A \, ,  \]
where the action on the left is on $S$.
\end{Lemma}
\begin{proof}
Let $\tau:S \rightarrow R, s \mapsto \frac{1}{|G|} \sum_{g \in G} s g$, denote the Rayleigh map, which is $R$-linear, and exists due to the nonmodular assumption. There is a natural ring homomorphism
\[ \varphi :A \longrightarrow (A \tensor_RS)^G,\, a \longmapsto a \tensor 1 \, .\]
Tensorization gives the $A$-linear map $\operatorname{Id}_A \tensor \tau: A \tensor_R S \rightarrow A$, which also shows that $A \subseteq  A \tensor_R S $ is a subring and so $\varphi$ is injective. The Rayleigh map $\tilde{\tau}$ for $A \tensor_R S$ equals $\tilde{\tau}   =   \varphi \circ  ( 1 \tensor \tau) $,	
since
\[ \begin{aligned}   \tilde{\tau} (a \tensor s)    =  &    \frac{1}{|G|} \sum_{g \in G}  (a \tensor s)  g  =     \frac{1}{|G|} \sum_{g \in G}  (a \tensor s g) = ( a \tensor   \frac{1}{|G|}\sum_{g \in G}  s g)       
 \\  = & a \tensor \tau(s)   =a \tau(s) \tensor 1 	=  \varphi(  a   \tau(s) ) \, .  \end{aligned} \]
Hence $\varphi$ is also surjective since $\tilde{\tau}$ is.
\end{proof}

This means that in the situation
\[ \begin{matrix}  & &Y &  \longleftarrow  &Y  \times {\mathbb A}^m \\ &  & \downarrow & & \downarrow \\ T  & \longrightarrow & X & \longleftarrow & Z_ \rho \, , \end{matrix} \]
where $T$ is a scheme with a fixed morphism to $X$, in the base change diagram to $T $, 
\[ \begin{matrix} T \times_X Y &  \longleftarrow  & T \times_X ( Y  \times {\mathbb A}^m )\\  \downarrow & & \downarrow \\ T & \longleftarrow & T \times_X Z_ \rho \end{matrix} \]
both vertical maps are quotients of the corresponding group actions by Lemma \ref{actionuniversal} (on the right, apply the lemma to the base change $T \times_X Z_\rho \rightarrow Z_\rho$). For $T=\{P\}$ a point, this is
\[ \begin{matrix} P \times_X Y &  \longleftarrow  & P \times_X ( Y \times {\mathbb A}^m )\\  \downarrow & & \downarrow \\ P & \longleftarrow & P \times_X Z_ \rho \, , \end{matrix} \]
where $P \times_X Y $ is the fiber over $P$ in $Y$, which is given by the fiber ring  $S/ {\mathfrak m}_P S$. In the case of $Y= {\mathbb A}^d$ and $P$ the image of the origin (the vertex), this is the \emph{ring of coinvariants}.

Important properties of an action $\beta$ of a finite group $G$ on $Y$ can be understood by looking at the morphism
\[ \beta \times p_2:G \times Y  \rightarrow Y \times Y , (g,y) \longmapsto (gy,y)\, ,\] and at the morphism
$G \times Y  \rightarrow Y \times_X Y$, where $X$ denotes the quotient. The ring version of this is 
\[ S \tensor S \rightarrow K^{| G|} \tensor S \cong S^G ,\,  s \tensor t \mapsto    \sum_{g \in G} (s) gt e_g \, . \]
The action  is called  \emph{free} (in the sense of fixed point free) if the morphism $\beta \times p_2:G \times Y  \rightarrow Y \times Y $ is a closed embedding.
 
\begin{Remark}
\label{freeprincipal}
In the free case according to \cite[Proposition 0.9]{GIT}, $Y \rightarrow X$ is faithfully flat and $ G \times Y \rightarrow Y \times_X Y$ is a $G$-equivariant isomorphism, where the action on the left is only on the first component. Such actions are called \emph{principal fiber bundles}. As we consider only group schemes $G$ coming from a usual group (so that the corresponding group scheme is reduced), the morphism $Y \rightarrow X$ in the free case (and if everything is essentially of finite type) is also \'etale.
\end{Remark}

In our situation, we basically never have a free action, but we have a free action on an open subset.

\begin{Lemma}
\label{genericfree}
Let $S$ be a $K$-domain and let the finite group $G$ act faithfully on $S$ via $K$-algebra automorphisms. Then there exists an invariant $0 \neq f \in S$ such that the action on $D(f)$ is free.
\end{Lemma}
\begin{proof}
Let $R=S^G$. In the quotient fields, we have a Galois field extension $Q(R) \subseteq Q(S)$, hence, we have a $Q(R)$-isomorphism
\[ Q(S) \tensor_{Q(R)} Q(S)   \cong  K^{|G|} \tensor Q(S) \cong Q(S)^{|G|}  \, \]
coming from the natural ringhomomorphism
\[    S \tensor_R S \longrightarrow  K^{|G|} \tensor S \, .  \]
This gives an isomorphism after a suitable localization.
\end{proof}

There is then also a largest open $G$-invariant subset $V \subseteq Y$ where the action is free.

\begin{Lemma}
\label{freebasechange}
If $G$ acts freely on $V$, then $G$ also acts freely on $Y \times V$, where $G$ acts somehow on $Y$.
\end{Lemma}
\begin{proof}
Let $\nu:G \times Y \rightarrow Y$ be the action on $Y$. By assumption, 
$G \times V \stackrel{\mu \times p_2}{\longrightarrow } V \times V$ is a closed embedding. We look at
\[  G \times Y \times  V \stackrel{\nu \times \mu \times p_2 \times p_3}{\longrightarrow } ( Y \times  V )   \times ( Y \times  V) \, .\]
Here, $ \mu \times p_2 \times p_3$ is already a closed embedding into $Y \times V \times V$, hence, it remains a closed embedding after adjoining the additional map $\nu$.
\end{proof}

Every reflection group yields a flat morphism to the quotient, but the action is not free and the quotient map is usually not \'etale.

\begin{Remark}
\label{flatdescent}
Let $G$ act on $Y$ via $\mu$ with quotient $X$. We look at the commutative diagram ($\nu$ denotes the operation in the group)
\[ \begin{matrix} G \times G \times Y & \stackrel{\nu \times p_3, p_1 \times p_3, p_2 \times p_3}{ \longrightarrow} &  G \times Y & \stackrel{\mu, p_2}{ \longrightarrow} &Y & \stackrel{q}{\longrightarrow} & X \\
\mu_1 \times \mu_2 \times 	p_3 \downarrow & & \downarrow  \mu \times p_2 & & \downarrow & & \downarrow \\ 	
 Y \times_X Y \times_X Y &\stackrel{ p_1 \times p_2, p_1 \times p_3, p_2 \times p_3 }{\longrightarrow} & Y \times_X Y & \stackrel{p_1, p_2}{ \longrightarrow } &  Y& \stackrel{q}{\longrightarrow} & X , \,  \end{matrix}  \]
where the lower row represents the setting of faithfully flat descent, see \cite[Section VIII]{sga1}. When the action is free, we have vertical isomorphisms everywhere, see \cite[Example 6.2.B]{boschneron}, and then the data above $Y$ (a module, an algebra, an affine scheme) with a group action are compatible with the group action if and only if they are descent data.
\end{Remark}

\begin{Example}
Let $S=K [X]/(X^k)$ with the action of $\Z/(k)$ where the generator acts by multiplication with a primitive $k$th root of unity $\zeta$, provided it exists. The invariant ring is just $K$, and the morphism $K \rightarrow S$ is flat, but not \'etale since $dX \neq 0$. The action is not free, because, under the ring homomorphism
\[ K[X][Y]/(X^k,Y^k)   \longrightarrow   K[U ]/(U^k-1) [X]/( X^k), \, X \mapsto X, Y \mapsto UX \, , \]
$U$ does not lie in the image (or, because on the left, the spectrum has one point and on the right it has $k$ points, so there is no closed embedding).
\end{Example}

See also Example \ref{cycliconenotfree} for a related example.

\section{Toric case}
\label{toricsection}

Invariant rings and invariant algebras are especially easy to describe in the toric situation and are therefore a good starting point. Let $K$ be any field. We consider normal toric rings, which are given as the degree $0$-part of $K[X_1, \ldots, X_d ]$ under a $D$-grading given by a group homomorphism $ \delta :{\mathbb Z}^d \rightarrow D$, where $D$ is a finite commutative group and the grading is determined by $ \delta_i := \delta (e_i) $, the degree of the variable $X_i$ (also called the weights of the corresponding action). We usually assume that $ \delta $ is surjective. The dual group $G=D^\dual=\textnormal{Hom} (D,K^\times)$  (here $K^\times$ is the multiplicative unit group of $K$) acts on the polynomial ring in the following way: An element $\chi \in G$, i.e., a character $\chi :D \rightarrow K^\times$, acts on the variables $X_i$ by
\[ \chi X_i = \chi( \delta_i )  X_i \, , \]
hence, by the diagonal matrix
\[ \begin{pmatrix} \chi ( \delta_1)  & 0 & \ldots & 0 \\ 0 & \chi (d_2)  & \ldots & 0 \\ \vdots  & \vdots & \ddots &  \vdots  \\ 0  & 0 & \ldots & \chi (\delta_d) \end{pmatrix} \, .\]
The invariant ring under this action is, under the assumption that $K$ contains enough roots of unity, just $K[X_1, \ldots ,X_d]_0$, the degree zero part of the $\delta$-grading. The condition that the grading is surjective corresponds to the property that the action of $G$ is faithful.

All toric normal monoid rings with a finite divisor class group arise in this way. These are the monoids given by a rational polyhedral cone that are simplicial, meaning that the number of facets of the cone equals the dimension. A two dimensional cone is always simplicial, hence, two dimensional affine monoid rings can always be realized in this way.

A (not necessarily surjective) $D$-grading $\epsilon $ of another polynomial ring $K[W_1, \ldots, W_m ]$ corresponds to another linear action $\rho$ of $G$. Its invariant algebra over the invariant ring
$K[X_1, \ldots, X_d ]_0$ is the degree zero part
\[ B^\rho =  K[X_1, \ldots, X_d, W_1, \ldots, W_m ]_0  \, .\]
This algebra can be written as
\[ B^\rho = \bigoplus_{ \nu \in \N^m } K[X_1, \ldots, X_d]_{- \epsilon (\nu)}  W^\nu \, .  \]

\begin{Lemma}
\label{Toriccharacterdecomposition}
Suppose the toric situation and that $K$ contains enough roots of unity so that $D^\dual$ is isomorphic to $D$. Then, we have
\[ B^\rho = \bigoplus_\chi  K[X_1, \ldots, X_d]^{\beta, \chi} \tensor K[ W_1, \ldots, W_m]^{\rho , \chi^{-1} } \, ,\]
where the sum runs over the characters of $G$.
\end{Lemma}
\begin{proof}
We have
\[	\begin{aligned}  B^\rho & =   \bigoplus_{ \nu \in \N^m } K[X_1, \ldots, X_d]_{- \epsilon (\nu)}  W^\nu \\ & =  \bigoplus_{ \lambda \in D } (  \bigoplus_{ \nu \in \N^m , \, \epsilon ( \nu) = \lambda } K[X_1, \ldots, X_d]_{- \lambda }  W^\nu ) 
\\ &=  \bigoplus_{ \lambda \in D }   K[X_1, \ldots, X_d]_{- \lambda }  \tensor   K[W_1, \ldots, W_m]_{ \lambda }  \end{aligned}  \]
The grading group $D$ is the same as the character group of $G=D^\dual $, and the elements of degree $\lambda$ are the semi-invariants of the corresponding character since for a character $\chi$ of $G=D^\dual$ corresponding to a degree $\lambda$ we have
\[	\begin{aligned} K[X_1, \ldots, X_d]^\chi =& \{f \in K[X_1, \ldots, X_d] | g \cdot  f  =   \chi(g) f \text{ for all } g \in G\} \\
=&  \{ f= \sum_{\delta \in D} f_\delta  | g \cdot  f  =  g(\lambda) f \text{ for all } g \in D^\dual \} \\	=&  \{  \sum_{\delta \in D} f_\delta  | \sum_{\delta \in D} g(\delta) f_\delta   =  g(\lambda) \sum_{\delta \in D} f_\delta \text{ for all } g \in D^\dual \} \\ = & K[X_1, \ldots, X_d]_\lambda \, ,
\end{aligned} \] 
because the condition on the roots of unity implies that $f_\delta=0$ for $\delta \neq \lambda$.
\end{proof}

We look at one-dimensional examples, which already exhibit quite typical behavior.

\begin{Lemma}
\label{cyclicone}
Let $k \in \N_+$, let $G= \Z /(k)$ and let $\zeta$ denote a primitive $k$th root of unity in $K$ (assuming that it exists). Let $\beta$ be the action of $G$ on $K[X]$, where the group generator acts as $X \mapsto \zeta X$, and let $\rho =\rho_\ell$, $\ell =0,1, \ldots, k-1$, be the action of $G$ on $K[W]$, where the generator acts as $W \mapsto \zeta^\ell W$. Then the following properties hold.
	
\begin{enumerate}
\item
The invariant ring is $K[X]^\beta=K[X^k]=K[A]=R$.
		
\item
The  invariant $K [X]^\beta $-algebra is
\[ A_\ell = K[X,W]^{\beta \times \rho_\ell } = K[X^iW^j| i + j \ell \in \Z k] =  \bigoplus_{j =0}^\infty K[X]_{ -j \ell } W^j \, .\]
		
\item
For $\ell=0$, we obtain $A_0 = K[X^k] [W]$.
		
\item
For $\ell \geq 1$, the monomial $X^{k-\ell} W^1 $ has the property that every monomial as above can be expressed as (for some $n$)
\[  X^iW^j =\frac{ ( X^{k-\ell } W^{1})^{j}  }{ (X^k)^n} \, . \]
In particular, after localizing at $A$, we have  
\[  K[X,W]^{\beta \times \rho_\ell }_A \cong K[A]_A [  X^{k-\ell} W^1    ]  \, . \]
\end{enumerate}
\end{Lemma}
\begin{proof}
(1) and (2) follow from the toric situation, and Lemma \ref{Toriccharacterdecomposition}, (3) and (4) are also clear.
\end{proof}

In Lemma \ref{cyclicone}, the action on $K[X,W]$ is given by the matrix $\begin{pmatrix}  \zeta & 0 \\ 0 & \zeta^\ell \end{pmatrix} $.

\begin{example}
\label{cycliconespecial}
We look at Lemma \ref{cyclicone} for $\ell=k-1=-1$. The invariant algebra is $K[X^k,W^k, XW] \cong K[A,B,C]/(AB-C^k)$, describing the two-dimensional $A_{k-1}$-singularity, considered as a $K[A]$-algebra. Over the punctured spectrum of $\Spec R$, the isomorphism of  Lemma \ref{cyclicone} (4) is  
\[  K [A,B,C]_{A}  \cong K[A]_{A} [ C ]  \, .\] The fiber ring over the vertex point (the origin in our case), when we mod out $A$, is $K[B,C]/(C^k)$, a nonreduced affine line.	
\end{example}

\begin{example}
We look at Lemma \ref{cyclicone} for $\ell=1$, the invariant algebra is the Veronese ring
\[ K[X^k,X^{k-1}W, \ldots ,W^k]=K[ A_0 ,\ldots, A_k]/( \text{binary relations}) \, ,\]
considered as a $K[X^k]$-algebra. Over the punctured spectrum of $\Spec R$, the isomorphism of Lemma \ref{cyclicone} (4) is
\[  K[X^k,X^{k-1}W, \ldots ,W^k]_{X^k}  \cong K[X^k]_{X^k} [   X^{k-1} W ]  \, .\]
In the fiber ring over the origin, when we mod out $X^k$, all the monomials with the exception of $W^k$ (which is the free variable of the fiber) become nilpotent. The exponent of nilpotence of $XW^{k-1}$ is $k$, that of  $X^2W^{k-2}$ is $ \lfloor (k+1)/2 \rfloor$, etc.
\end{example}

\begin{example}	
We continue with Lemma \ref{cyclicone} for $k=2$. There are two one-dimensional representations, accordingly we have on the quotient scheme, which is just the affine line, two quotient schemes of relative dimension one, the trivial one, and $Z_1 =\Spec [A,B,C]/( AB-C^2) $ with a nonreduced fiber over the origin.	
\end{example}

\begin{example}	
We continue with Lemma \ref{cyclicone} for $k=3$. There are three one-dimensional representations, for $\ell=0,1,2$, and accordingly we have on the quotient scheme, which is just the affine line $\Spec K[X^3] = \Spec K[A]$, three quotient schemes of relative dimension one, the trivial one, $Z_2 =\Spec K[A,B,C]/( AB-C^3) $ with the nonreduced fiber being $K[C,B]/(C^3)$ and the Veronese ring
\[	\begin{aligned} Z_1 & =   \Spec K[X^3, X^2W,XW^2,W^3] \\ & \cong  \Spec K[A,B,C,D]/(AD-CB, B^2-AC, C^2-BD,C^3-AD^2, B^3-A^2D )      \, ,   \end{aligned}  \]
with the fiber over the origin being $ \Spec K[B,C,D]/(CB,B^2,C^2-BD,C^3)    $, where $D$ is a free variable, and $1,B,C,C^2 $ is a basis over $K[D]$.
\end{example}

\begin{example}	
\label{cyclicone5}
We continue with Lemma \ref{cyclicone} for $k=5$ and look at the one-dimensional representations for $\ell= 2$ and $3$. The invariant ring is $ K[X^5] = \Spec K[A]$ and the invariant algebras are
\[ \begin{aligned} B_2 = K [X,W]^{(1,2)} &= K[X^5,X^3W,XW^2,W^5 ]  \\ & \cong  K[A,C,D,B]/(C^2-AD, CD^2-AB, D^3-BC) \, ,   \end{aligned} \]	
and 	
\[ \begin{aligned} B_3= K [X,W]^{(1,3)}  &=  K[X^5,X^2W,XW^3,W^5 ] \\  & \cong K[A,E,F,B]/(F^2-BE, FE^2-AB, E^3-AF) \, .  \end{aligned} \]	
As $K$-algebras, these two algebras are isomorphic via $A \leftrightarrow B, C \leftrightarrow F, D \leftrightarrow E$ (inherited from $X\leftrightarrow W$). However, they are not isomorphic as $K[A]$-algebras (this follows from Theorem \ref{pullback}). In fact, not even the fiber rings above the origin are isomorphic $K$-algebras, as we will see in Example \ref{cyclicone5fiber}.	
\end{example}

\begin{Example}
\label{artinian}
As in Lemma \ref{cyclicone}, we consider the action of $\Z/(k)$ but now on the Artinian algebra $S=K[X]/(X^k)$, the invariant ring is just $K$. This is the action from Lemma \ref{cyclicone} over the base change $K[A] \rightarrow K,\, A \mapsto 0$, see also Lemma \ref{actionuniversal}. The invariant algebra for the representation $\rho_\ell$ is
\[ K[ X^iW^j | \, i + \ell j \in \Z k ]/(X^k)   \, ,\]
its reduction is $K[W^k]$.
\end{Example}

\begin{Example}
\label{cycliconenotfree}
In the situation of Lemma \ref{cyclicone}, the morphism $\mu \times p_2$ from Remark \ref{flatdescent} corresponds to the ring homomorphism
\[     K[X] \tensor_{ K[X^k]} K[X] =K[X,Y]/(X^k-Y^k)  \longrightarrow K[X] \tensor K[U]/(U^k-1), Y \mapsto UX \, . \]
On the left, we have $k$ lines meeting in a point, on the right, we have $k$ disjoint lines.	
\end{Example}

\begin{example}
\label{toric11.1}
We consider the action of the group $G=\Z/(2)$ on $S=K [X,Y]$, where the nontrivial group element acts by multiplication with $-$, and the invariant ring is $R=K[X^2,Y^2, XY] \cong K[A,B,C]/(AB-C^2)$. We further consider the action of $G$ on $K[W]$ by sign. The invariant algebra is the Veronese algebra in three variables in degree $2$,
\[	\begin{aligned}  K[X,Y,W]^G & = K[X^2,Y^2,XY, W^2,XW,YW] \\
	& = K[A,B,C,D,E,F]/(AB-C^2, AD- E^2, BD-F^2,\ldots )  \end{aligned}	\]
Outside of the zero locus $V(A)$, this algebra is $K[A,C]_{A} [E] $ (similar outside $V(B)$), over $V(A,B,C)$ the fiber algebra is $K[D,E,F]/(E^2,F^2 ,EF)$.
\end{example}

\begin{example}
\label{Asingularity}
The action of the group $\Z/(k)$ on the polynomial ring $K[X,Y]$, where $1$ acts via the matrix $\begin{pmatrix} \zeta &0 \\ 0 & \zeta^{-1} \end{pmatrix}	$ ($\zeta$ a primitive $k$th root of unity), yields the invariant ring $R=K[X^k,Y^k,XY] \cong K[A,B,C]/(AB-C^k)$, the $A_{k-1}$-singularity. The action $\rho=\rho_\ell$ given on $K[W]$ by $1$ acting as $W \mapsto \zeta^\ell W$ gives the quotient schemes
\[ Z_\ell = \Spec K[X,Y,W]^{\beta, \rho_\ell}  \, .\]
The invariant algebra is given as
\[ K[X,Y,W]^{(1,-1, \ell)}  = [X^iY^jW^m:\, i -j+ \ell m \in  \Z k] \, . \]
In the first nontrivial case, $k=2$ and $\ell =1$, this is	$K[X^2,Y^2,XY, W^2,XW,YW]$, so the second Veronese ring in three variables.	Over $R$, this algebra is
\[ R[D,E,F]/(CD-EF, CE- AF, BE -C F, AD-E^2, BD-F^2) \, .\]
When we invert $A$, this algebra is isomorphic to the polynomial ring $R_A[E]$, as  $D=E^2/A$ and $F=CE/A$, when we invert $B$, this algebra is isomorphic to $R_B[F]$.	The fiber ring over $(A,B,C)$ is $K[D,E,F]/(EF,E^2,F^2)$.	
\end{example}

\begin{example}
\label{kleinaction}
Let $ G= \Z/(2) \times \Z/(2) $ act naturally by compo\-nentwise negation on $K[X,Y]$, with invariant ring $K[X^2,Y^2]=K[A,B] $.
	
(1). The action on $K[W]$, where the first component of the group acts by $-$ and the second component acts as the identity, so on $K[X,Y,W]$, $(+,+)$ acts as the identity, $(+,-)$ as $ (+,-,+)$, $(-,+)$ as $ (-,+,-)$,  $(-,-)$ as $ (-,-,-)$. This  yields the invariant algebra \[  K[X^2,Y^2, W^2, XW ] \cong K[A,B,C,D]/(D^2-AC) \, .\]
This is a nonisolated singularity, the singular locus consists of the line $V(A,C,D)$, and its image is $V(A)$. If we invert $A$, we get $K[A,B]_A[D]$, but over $A=0$, we get $ K[B] \subseteq K[B,C,D]/(D^2)$, so this is a nonreduced affine line over an affine line.
	
(2).	If we look at the action which factors through the sum $ \Z/(2) \times \Z/(2) \rightarrow \Z/(2) $, then $(+,-)$ acts as $ (+,-,-)$, $(-,+)$ as $ (-,+,-)$, and $(-,-)$ as $ (-,-,+)$. This  yields the invariant algebra 
\[ K[X^2,Y^2, W^2,XYW ] \cong K[A,B,C,D]/( D^2 -  ABC ) \, .\]
This is a nonisolated singularity, and the singular locus consists of three lines $V(A,B,D)$, $V(A,C,D)$ and $V(B,C,D)$. The images of these lines are the point $V(A,B)$ and the lines $V(A)$ and $V(B)$. If both $A$ and $B$ are inverted, then we get $K[A,B]_{AB} [D] $. Above $A=0$ (similar above $B=0$), we get $ K[B] [C,D]/(D^2) $.
\end{example}

\begin{Example}
\label{cyclic4reflection}
Let $G$ be the cyclic group of order $4$ with the basic action $\beta$ given (for the generator) by the matrix $M=\begin{pmatrix} -1 & 0 \\ 0 & \mathrm {i}   \end{pmatrix}$, the subgroup $H$ of reflections is $ \Z/(2)$, generated by $ M^2 $. The invariant ring under $H$ is $ \C[ X_1,X_2^2] $, and the invariant ring under $G$ is $\C[ X_1^2, X_2^4, X_1X_2^2 ] \cong \C[A,B,C]/(AB-C^2)$. Let $\rho$ be the one-dimensional representation of $G$ given by multiplication with $\mathrm {i} $. Then the representation $\beta \times \rho$ is small, and the invariant algebra is
\[ K[X_1,X_2,W]^{(2,1,1)} =K[X_1^2,X_1X_2^2, X_2^4, W^4,X_2W^3,X_2^2W^2,X_2^3W, X_1W^2, X_1X_2W] \, .\]	
\end{Example}

See also Example \ref{cyclic4reflectionsingularities} and Example \ref{cyclic4reflectionfundamental}.

\section{Further examples}
\label{examplessection}

We describe several further nontoric examples (the first one is toric in characteristic $\neq 2$ after a change of coordinates).

\begin{example}
We consider the symmetric action of the group $G=\Z/(2)$ on $S=K [X,Y]$, where the nontrivial group element exchanges $X$ and $Y$, and the invariant ring is $K[X+Y,XY] \cong K[A,B]$. We further consider the action of the group on $K[W]$ by sign. The invariant algebra is then
\[  \begin{aligned} K[X,Y,W]^G & = K[X+Y,XY,W^2, (X-Y)W] \\ &= K[A,B,U,V]/( V^2- U( A^2-4B)   ) \, .    \end{aligned} \]
Outside of the zero locus $V(A^2-4B)$, this algebra is $K[A,B]_{A^2-4B} [V]$, over this zero locus the algebra is $K[A,B]/(A^2-4B) \subseteq K[A,B]/(A^2-4B)[U,V]/(V^2)$.
\end{example}

\begin{Lemma}
\label{symmetricdeterminant}
For the symmetric group $S_d$ with its natural action on $K[X_1, \ldots, X_d]$ (and with the invariant ring given by $K[E_1, \ldots, E_d]$, $E_i$ being the elementary symmetric polynomials) and the nontrivial action of $S_d$ on $K[W]$, which factors through the sign homomorphism $ S_d \rightarrow \{1,-1\}$, the invariant algebra is 
\[  K[E_1, \ldots, E_d][W^2, \vandermonde W] \cong  K[E_1, \ldots, E_d]  [U,V]/( V^2 - U \vandermonde^2) \,  \] 
where $\vandermonde =\prod_{i <j} (X_i-X_j)$ denotes the Vandermonde determinant. Outside of $V(\vandermonde^2)$, the invariant algebra is $ K[E_1, \ldots, E_d]_{\vandermonde^2} [V]$, and above $V(\vandermonde^2)$, the invariant algebra is $  K[E_1, \ldots, E_d]/(\vandermonde^2)  [U,V]/(V^2)$.
\end{Lemma}
\begin{proof}
This follows from the fact that for the action of $S_d$ on the polynomial ring, the semi-invariants with respect to $-$ are given by
$ K[E_1, \ldots , E_d]  \vandermonde $.
\end{proof}

Note that the natural action of the symmetric group is generated by the reflections coming from the transpositions, the mirror hyperplanes are $V(X_i-X_j)$ for $i \neq j$. The diagonal is fixed by the action.

\begin{example}
\label{symmetricdeterminant3}
We consider the symmetric group $S_3$ with its natural action on $ K [X,Y,Z] $, and the invariant ring is $ K[X+Y+Z,XY+XZ+YZ,XYZ] \cong K[A,B,C]$. We further consider the action of the group via the sign homomorphism $S_3 \rightarrow \{1,-1\}$ on $K[W]$ by sign. The invariant algebra is 
\[ \begin{aligned} K[X,Y,Z,W]^G & =  K[X+Y+Z,XY+XZ+YZ,XYZ,W^2, \vandermonde W] \\ & = K[A,B,C,U,V]/(V^2- U \vandermonde^2   ) \, , \end{aligned} \]
where $\vandermonde=(X-Y)(X-Z)(Y-Z)$, in accordance with Lemma \ref{symmetricdeterminant}.
\end{example}

\begin{Remark}
\label{discriminant}
In general, if $G \subseteq \GLG_d(K)$ is a finite group generated by reflections, with mirror hyperplanes $H_1, \ldots, H_n$, then the image of the mirror hyperplanes in $X$ is described by the \emph{discriminant} $\triangle$. If the mirror $H_i$ is given by the linear form $L_i$ and the order of the subgroup with this mirror is $\nu_i$, then the discriminant is (up to a unit) $\prod_{i =1}^n L_i^{\nu_i} $, which is an element of the invariant ring, see \cite[2.3.1]{buchweitzfaberingalls}. In the situation of Lemma \ref{symmetricdeterminant}, we have $\triangle= \vandermonde^2$.
\end{Remark}

\begin{example}
We let the symmetric group	$S_3$ act on ${\mathbb A}^2 $ via the matrices (from the left)
\[ \begin{pmatrix}  1 & 0 \\ 0 & 1\end{pmatrix}\,, \begin{pmatrix}  0 & 1 \\ 1 &  0  \end{pmatrix}\,, \begin{pmatrix}  -1 & 0 \\ -1  & 1 \end{pmatrix}\,,   \begin{pmatrix}  1 & -1 \\ 0 & -1  \end{pmatrix}  \,,  \begin{pmatrix}  -1 & 1 \\ -1 & 0  \end{pmatrix}\,,   \begin{pmatrix}  0 & -1 \\ 1  & -1 \end{pmatrix}\,, \]
the one with determinant $-1$ are reflections (with axes $K(e_1+e_2)$, $Ke_1$, $Ke_2$) and we have a reflection group. The corresponding actions on $K[X,Y]$ of the transpositions are given by $X \leftrightarrow Y$, $X \mapsto -X$, $Y \mapsto -X+Y$ and $X \mapsto X-Y$, $Y \mapsto -Y$. The invariant ring is $ K[X^2+Y^2-XY , 2X^3-3X^2Y- 3XY^2+2Y^3 ] =K[A,B] $. The discriminant is $\triangle= 4A^3-B^2=27X^2Y^2(X-Y)^2 $. Its square root is a semi-invariant with respect to $-$, as in Lemma \ref{symmetricdeterminant}. Hence, if $\rho$ is the representation given by the determinant, then the invariant algebra is $K[A,B][U,V]/ ( V^2 - U \triangle) $.
\end{example}

\begin{example}
\label{kleinactionhyperplane}
We restrict the action from Example \ref{kleinaction} to the union of the two reflecting lines, so $G=\Z/(2) \times \Z/(2)$ is acting on the ring $K[X,Y]/(XY)$ with invariant ring $K[A,B]/(AB)$. For the first representation, the invariant algebra is $K[A,B,C,D]/(AB,D^2-AC)$, for the second representation, it is $K[A,B,C,D](AB, D^2)$, which can be computed directly or follows from Example \ref{kleinaction} by Lemma \ref{actionuniversal}.
\end{example}

\section{Invariant modules and invariant algebras}
\label{modulealgebrasection}

Suppose that a finite group $G$ acts on the $K$-algebra $S$ by $K$-algebra automorphism and that $G$ acts linearly via $\rho$ on a finite dimensional $K$-vector space $V$. Then the natural action of $G$ on the free $S$-module $N = S \tensor_K V \cong S^m$ is given for $g \in G$ as (action from the right, think of $V$ as a space of linear forms)
\[ S \tensor_K V \longrightarrow  S \tensor_K V , s \tensor v \longrightarrow  sg \tensor vg \, . \]
This action is compatible in the sense that
\[  s g \cdot w g = (s w) g\]
holds for all $g \in G$ and $s \in S$. The module $S \tensor_K V$ is an $S^G$-module by restricting the scalars in the first component, and the mappings given by $g \in G$ are $S^G$-module automorphisms. The \emph{module of covariants} (or invariant module from now on)
\[ N^\rho= \{ w \in N:\, wg =w \text{ for all } g \in G \}\]
is therefore an $S^G$-module.

\begin{Remark}
If $S$ is a Cohen-Macaulay ring and $G$ is nonmodular, then the invariant modules are maximal Cohen-Macaulay $R$-modules. They are, if they are indecomposable, direct summands of $S$ as $R$-modules, see Section \ref{regularrepresentationsection} (in particular Lemma \ref{cyclic4reflection}), and $S$ is also a Cohen-Macaulay $R$-module, as follows with local cohomology.  
\end{Remark}

\begin{Remark}
Basically every example of a group generated by reflections shows that this module assignment comes with a loss of information, because then the invariant ring is a polynomial ring and all maximal Cohen-Macaulay modules are free.
\end{Remark}

\begin{Lemma}
\label{modulealgebra}
Let	$G$ be a finite group acting faithfully via $\beta$ on a $K$-algebra $S$ by $K$-algebra automorphisms with invariant ring $R=S^G$. Let $\rho$ be a $K$-linear representation on the $K$-vector space $V$, let $G$ act naturally on $N=S \otimes_K V$ and let $M=N^\rho $ be the corresponding invariant $R$-module.	Then there is a canonical $\N$-graded $R$-algebra homomorphism
\[  \Sym_R (M) \longrightarrow (S \tensor_K K[V])^{\beta \times \rho}  \, . \]
It is an isomorphism in degrees $0$ and $1$.
\end{Lemma}
\begin{proof}
We start with the natural $S$-module homomorphism
\[S \tensor_KV \longrightarrow S \tensor_K K[V] , \]
where on the right we have an $S$-algebra. This mapping is $G$-equivariant, therefore, we obtain an $R$-module homomorphism
\[ M = (S \tensor_KV)^\rho \longrightarrow  (S \tensor_K K[V])^\rho   \, .   \]
By the universal property of the symmetric algebra, this gives an $R$-algebra homomorphism
\[  \Sym_R (M) \longrightarrow (S \tensor_K K[V])^\rho  \, . \]
On both sides, the degree $0$ part is $R$ and the degree $1$ part is $M$.
\end{proof}

This algebra homomorphism is in general not surjective. The corresponding scheme morphism is $Z_\rho \rightarrow \Spec ( \Sym_R (M) )  $, and in Lemma \ref{modulealgebrascheme}, we will understand this morphism as a morphism of module schemes up to normalization.

\begin{example}
In the situation of Lemma \ref{cyclicone}, the invariant $K[X^k]$-module is the free module of rank $1$ generated by $X^{k-\ell}W^1$, which is embedded in the invariant algebra according to Lemma \ref{modulealgebra}. The ring homomomorphism from Lemma \ref{modulealgebra} is only an isomorphism when $\ell=0$. This will follow from Theorem \ref{nonreducedfiber} or Theorem \ref{pullback}, but can also be seen directly: For $\ell \neq 0$, let $j$ be the smallest number such that $ j(k- \ell) \geq k$. Then the invariant monomial $X^{j(k-\ell) - k} W^j $  is not a power of $X^{k-\ell}W$, and hence not in the image.
\end{example}

\begin{Example}
In Example \ref{artinian}, the invariant module is just a $K$-vector space, and hence, its symmetric algebra is always a polynomial algebra. The ring homomorphism from Lemma	\ref{modulealgebra} is
\[ K  [T]  \longrightarrow K[ X^iW^j| \, i + \ell j \in \Z k ]/(X^k) , T \mapsto X^{k-\ell} W^1 \, .\]
\end{Example}

\begin{Example}
In Example \ref{kleinaction} (2), the ring homomorphism from Lemma \ref{modulealgebra} is
\[ K[A,B] [T]  \longrightarrow K[A,B] [C,D]/(D^2 - ABC ), T \mapsto D \, .\]
\end{Example}

\begin{Example}
\label{modulealgebradimension}
In Example \ref{Asingularity}, the ring homomorphism from Lemma \ref{modulealgebra} (for $k=2$ and $\ell=1$) is
\[ K[A,B,C]/( AB-C^2 ) [S,T]/(C S- AT, BS- CT )  \longrightarrow K[X^2,Y^2,XY, XW,YW,W^2] \, , \]
with $ S \mapsto XW=D$, $T \mapsto YW=E$. This is an isomorphism over the punctured spectrum (in codimension one), but not in general, as $W^2$ is not in the image. Hence, also for a small action, we cannot expect an isomorphism. The fiber of the symmetric algebra over the isolated singularity $(A,B,C)$ is a two-dimensional plane, the fiber of the invariant ring is $\Spec K[D,E,F](D^2,E^2,DE)$, and this fiber contracts to the origin under the morphism. In particular, the morphism between the spectra is not surjective.
\end{Example}

\begin{Lemma}
\label{invariantalgebradegree}
Let	$G$ be a finite group acting faithfully via $\beta$ on a $K$-algebra $S$ by $K$-algebra automorphisms with invariant ring $R=S^G$. Let $\rho$ be a $K$-linear representation on the $K$-vector space $V$.	
Then the degree $n$ component of the invariant algebra $( S \tensor_K K[V])^{\beta \times \rho}  $ is
\[ ((S \tensor_K K[V])^{\beta \times \rho} )_n  = (S \tensor_K \Sym^n_K V )^{G}  \, , \]
where $G$ acts naturally on the symmetric powers of $V$.
\end{Lemma}
\begin{proof}
Because the action is homogeneous on $K[V]$, we have
\[  (S \tensor_K K[V])^{\beta \times \rho} )_n =(  (S \tensor_K K[V])_n )^\rho = (S \tensor_K \Sym^n_K V )^{G}  \, .  \] 
\end{proof}

For a finitely generated $\N$-graded $R$-algebra $A$ with $A_n$ finite $R$-modules, we call the function given by $n \mapsto \mu_R (A_n)$ the \emph{Hilbert function} of $A$, where $\mu_r$ denotes the minimal number of $R$-module generators.

\begin{Corollary}
\label{invariantalgebrahilbert}
Let	$G$ be a finite nonmodular group acting linearly and faithfully via $\beta$  as a reflection group on  $K[X_1, \ldots, X_d]$ with invariant ring $R=S^G \cong K[T_1, \ldots, T_d]$. Let $\rho$ be a $K$-linear representation on the $K$-vector space $V$ of dimension $m$.	Then the Hilbert function of the invariant  $R$-algebra $( S \tensor_K K[V])^{\beta \times \rho}  $ is $\binom{n+m-1}{m-1}$.
\end{Corollary}
\begin{proof}
By Lemma \ref{invariantalgebradegree}, we have to determine the minimal number of $R$-generators of $ (S \tensor_K \Sym^n_K V )^{G} $. As we have a reflection group acting, this is a free $R$-module of rank $\binom{n+m-1}{m-1}$.	
\end{proof}

In the following, we want to show that over suitable open subsets $U\subseteq X$ the $K$-algebra homomorphism from Lemma \ref{modulealgebra} is an isomorphism and that $Z_\rho$ exhibits the structure of a vector bundle there.

\begin{Lemma}
\label{modulealgebrafree}
Let	$G$ be a finite group acting faithfully via $\beta$ on the $K$-algebra $S$ through $K$-algebra automorphisms with invariant ring $R=S^G$. Suppose that the action of $G$ on $\Spec S$ is free. Let $\rho$ be a $K$-linear representation on the $K$-vector space $V$, let $G$ act naturally on $N=S \otimes_K V$, and let $M=N^\rho $ be the corresponding invariant $R$-module. Then the natural $K$-algebra homomorphism from Lemma \ref{modulealgebra} is an isomorphism. Moreover, $ \Spec ( \Sym_R(M)) \cong Z_\rho$  is a vector bundle over $ \Spec R$.
\end{Lemma}
\begin{proof}	
We have $ S \tensor_R S \cong  K^{|G|} \tensor S  \cong S^{|G|}$ by \cite[Proposition 0.9]{GIT} and $ R \rightarrow S$ is faithfully flat. We look at the diagram
\[ \begin{matrix} S & \longrightarrow & S[W_1, \ldots, W_m] \\ \uparrow & & \uparrow \\  R & \longrightarrow & B^\rho \end{matrix} \]
after applying the faithfully flat base change $ R \rightarrow S$, which yields
\[ \begin{matrix}  K^{|G|} \tensor S & \longrightarrow & ( K^{|G|} \tensor S )  [W_1, \ldots, W_m] \\ \uparrow & & \uparrow \\  S & \longrightarrow &  S \tensor_R B^\rho \, . \end{matrix}  \]
The action of $G$ on $K^{|G|} $ and on the variables $W_i$ is as before, but there is no action on $S$ anymore.
	
We denote the idempotent elements of $K^{|G|} $ by $e_g$. We look at the $ K^{|G|} \tensor S$-homomorphism
\[ \psi:  ( K^{|G|} \tensor S) [T_1, \ldots, T_m] \longrightarrow (K^{|G|} \tensor S  ) [W_1, \ldots, W_m] \]
given by $T_i \mapsto \sum_{g \in G } ( W_i g) e_g$, where $G$ acts on the left-hand side only via the natural action on $K^{|G|} $.
	
For any element $F \in S[W_1, \ldots, W_m]$, the element $F e_g$ is the element $F$ on the open and closed subset $D(e_g)$ (which is isomorphic to $S[W_1, \ldots, W_m]$) of $\Spec (K^{|G|} \tensor S) [W_1, \ldots, W_m]$  corresponding to $g$ and the zero function on the other components. This mapping $\psi$ is $\N$-graded, $G$-equivariant, as for $h \in G$, we have
\[ ( \sum_{g \in G } W_i g e_g) h = \sum_{g \in G } W_i gh e_{gh} = \sum_{g \in G } W_i g e_{g} \, , \]
and it is an isomorphism since it is an isomorphism on each $D(e_g)$. This means that we can trivialize the action of $G$ and that the resulting invariant algebra is isomorphic to $ S[T_1, \ldots, T_m]$, and in particular
\[  \Sym_S  ((K^{|G|} \tensor S \tensor K[W_1, \ldots, W_m]_1 )^G ) \longrightarrow  S[T_1, \ldots, T_m] \cong ( (K^{|G|} \tensor S) [W_1, \ldots, W_m])^G   \]
is an isomorphism.
	
We now use descent properties for faithfully flat base change. By \cite[Proposition 2.7.1 (viii)]{ega4.2}, it follows that the original ring homomorphism is an isomorphism, and by \cite[Proposition 2.5.2]{ega4.2} it follows that $M$ is locally free.
\end{proof}

\begin{Remark}
\label{freedescent}
The previous result can also be deduced from faithfully flat descent for affine schemes (see \cite[Théorème VIII.2.1]{sga1}). To see this, one needs the mentioned result from GIT that $R \subseteq S$ is flat and that the compatible free group action corresponds to descent data, see \cite[Example 6.2.B]{boschneron}.
\end{Remark}

\begin{Example}
\label{galoisfree}
A special case of Lemma \ref{modulealgebrafree} is that $K \subseteq L$ is a Galois extension of fields. This means that in this case every invariant algebra is a polynomial algebra. One can also apply  Lemma \ref{modulealgebrafree} when $Y=X_L=X \times_KL$ for any $K$-scheme $X$.
\end{Example}

\begin{Corollary}
\label{modulealgebraisomorphism}
Let	$G$ be a finite group acting faithfully via $\beta$ on the $K$-algebra $S$ through $K$-algebra automorphisms with invariant ring $R=S^G$. Let $V \subseteq \Spec S$ be an invariant open subset where the action is free and let  $ U \subseteq \Spec R$ be its image. Let $\rho$ be a $K$-linear representation on the $K$-vector space $V$, let $G$ act naturally on $N=S \otimes_K V$, and let $M=N^\rho $ be the corresponding invariant $R$-module. Then the natural $K$-algebra homomorphism from Lemma \ref{modulealgebra} 
is an isomorphism over $U$. Moreover, $( \Spec (\Sym_R(M)))|_U \cong Z_\rho|_U$ is a vector bundle over $U$.
\end{Corollary}
\begin{proof}
Let $\Spec S_f =D(f) \subseteq \Spec S $ be an invariant open affine subset where the action is free. The construction of the invariant module, the invariant algebra, the symmetric algebra and the homomorphism in Lemma \ref{modulealgebra} commute with localization, and whether a given scheme morphism is an isomorphism can be checked locally. Hence, the result follows from Lemma \ref{modulealgebrafree}.	
\end{proof}

\begin{Corollary}
\label{modulealgebralinearisomorphism}
Let	$G$ be a finite group acting linearly and faithfully via $\beta$ on  $K^d$ and let $\rho$ be a linear representation of $G$. Let $R=K[X_1, \ldots, X_d]^\beta$ be the invariant ring and let $M $ be the invariant $R$-module corresponding to $\rho$. Let $V \subseteq {\mathbb A}^d$ be the nonempty open locus where the action $\beta$ is free and let $U \subseteq X$ be its image. Then, the natural $K$-algebra homomorphism from Lemma \ref{modulealgebra} is an isomorphism over $U$. Moreover, $ ( \Spec (\Sym_R (M)))|_U \cong Z_\rho|_U$ is a vector bundle over $U$.
\end{Corollary}
\begin{proof}
This is a special case of Corollary \ref{modulealgebraisomorphism}. The nonemptyness follows from Lemma \ref{genericfree}. 
\end{proof}

\begin{Corollary}
\label{tangentbundlelinearisomorphism}
Let	$G$ be a finite group acting linearly and faithfully via $\beta$ on ${\mathbb A}^d$ with invariant ring
$R=K[X_1, \ldots, X_d]^\beta$. Let $V \subseteq {\mathbb A}^d$ be the nonempty open locus where the action $\beta$ is free and let $U \subseteq X$ be its image. Then, $ ( \Spec (\Sym_R (\Omega_{ R|K})))|_U \cong ((  {\mathbb A}^d \times {\mathbb A}^d)/ \beta \times \beta ))|_U$ is the tangent bundle over $U$.
\end{Corollary}
\begin{proof}
The linear action $\beta$ of the finite group $G$ on ${\mathbb A}^d$ induces the natural action of $G$ on the module of K\"ahler differentials $\Omega_{K[X_1, \ldots, X_d] |K}$ and on the tangent bundle \[ T_{ { \mathbb A}^d} = {\mathbb A}^d \times {\mathbb A}^d =\Spec ( \Sym \Omega_{K[X_1, \ldots, X_d] |K} )\,  \]
which is given by $\beta \times \beta$. The corresponding invariant module is $\Omega_{ R|K} $. The spectrum of its symmetric algebra is related to the quotient 
${\mathbb A}^d \times {\mathbb A}^d/ \beta \times \beta $ as described in Corollary \ref{modulealgebralinearisomorphism}.
\end{proof}

${\mathbb A}^d \times {\mathbb A}^d/ \beta \times \beta $ is a natural candidate for the schematic closure of the tangent bundle of the regular locus of ${\mathbb A}^d/ \beta $. However, in the case of a reflection group, this differs from the (trivial) tangent bundle of the quotient.

For an $R$-module $M$, the Rees algebra ${\mathcal R} (M)$ is defined by
\[ {\mathcal R} (M) =\Sym (M)/(R-\text{torsion}) \, . \]
For an ideal $I$, this matches the usual definition of a Rees algebra, at least when $R$ is a domain, see \cite{vasconcelosarithmetic}.

\begin{Lemma}
\label{modulereesalgebra}
Let	$G$ be a finite group acting faithfully via $K$-algebra automorphisms on the $K$-domain	$S$ with invariant ring $R=S^G$. Let $\rho$ be a $K$-linear representation on the $K$-vector space $V$, let $G$ act naturally on $N=S \otimes_K V$, and let $M=N^\rho $ be the corresponding invariant $R$-module. Then the natural $K$-algebra homomorphism from Lemma \ref{modulealgebra} factors through the Rees algebra ${\mathcal R} (M)$.
\end{Lemma}
\begin{proof}
This is clear, as $S[V]$ and hence $S[V]^G$ is a domain.
\end{proof}

\begin{Example}
We look at Example \ref{toric11.1}. The degree one component of the invariant algebra and hence the invariant module is $M \cong \langle XW,YW \rangle $. This module is isomorphic to the ideal $I=(A,C)$ with the $R$-linear representation
\[ R^2 \longrightarrow  R^2 \longrightarrow I \longrightarrow 0  \]
where the matrix is $ \begin{pmatrix}  C& -A \\ B & - C  \end{pmatrix} \, .$ The symmetric algebra of this module is therefore just
\[  R[U,V]/(  CU-AV, BU-CV) \, .\]
The natural ring homomorphism from Lemma \ref{modulealgebra} is
\[  R[U,V]/(  CU-AV, BU-CV)   \longrightarrow \]
\[ K[A,B,C,D,E,F]/(AB-C^2, AD- E^2, BD-F^2, CD-EF,AF-CE,BE-CF)    \]
with $U \mapsto E$ and $ V \mapsto F$. The fiber over $V(A,B,C)$ of the symmetric algebra is $\Spec K[U,V]$ (two-dimensional), and the fiber of the quotient scheme, which is $ \Spec K[D,E,F]/(E^2,F^2 ,EF)$ (one-dimensional), is contracted under the corresponding scheme morphism to the origin.

In between the two algebras we have, by Lemma \ref{modulereesalgebra}, the Rees algebra of the ideal $R + (A,C)+ (A,C)^2 + \ldots $. In the describing ideal of the symmetric algebra we have the relation
\[ BU(CU-AV) + AV (BU-CV) =C( BU^2-AV^2 ) \, , \]
but the factors on the right are not in the ideal. The element $BU^2-AV^2$ equals in the Rees-Algebra the element $BA^2-A(AB)$, so this goes to $0$. It is also zero in the invariant algebra.
\end{Example}

\section{Fibers}
\label{fiberssection}

We want to describe the fibers of the quotient schemes $Z_\rho \rightarrow X$. For $G$ acting on $Y$ and a point $Q\in Y$, the stabilizer of $Q$ is the subgroup of group elements mapping $Q$ to itself (no requirements on the residue class field).

\begin{Theorem}
\label{fiberdescription}
Let $G$ be a finite nonmodular group acting faithfully via $\beta$ on $Y=  \Spec S$ by $K$-algebra automorphisms with quotient scheme $X=\Spec R$. Let $P \in X$ be a closed point with maximal ideal ${\mathfrak m}_P$, let $Q \in Y$ be a preimage and let $S/ {\mathfrak m}_P S \cong S_1 \times \cdots \times S_r $ be the decomposition of the fiber ring $S/ {\mathfrak m}_P S $ into local rings, with $Q$ corresponding to $S_1$. Let $H$ denote the stabilizer group for the point $Q$. Let $\rho $ be a linear action of $G$ on ${\mathbb A}^m$, let $Z_\rho= ( Y \times {\mathbb A}^m)/ \beta \times \rho$ be the quotient scheme, and let $\tilde{\rho}$ be the induced action of $H$ on ${\mathbb A}^m$. Then, the fiber ring of $Z_\rho$ over $P$ is $ (S_1 [W_1, \ldots, W_m])^{\tilde{\rho} }$.
\end{Theorem}
\begin{proof}
The stabilizers are for all points above $P$ conjugated subgroups, and the rings $S_i$ are isomorphic rings, as the group acts transitively on the orbits. The fiber ring over $P$ is $S \tensor_R \kappa (P) \cong S \tensor_R R/ {\mathfrak m}_P \cong S/{\mathfrak m}_P S$. The tensor products of the diagram
\[   \begin{matrix} S &\longrightarrow &  S[W_1, \ldots, W_m] \\ \uparrow & & \uparrow \\ R& \longrightarrow &  (S[W_1, \ldots, W_m]  )^\rho        \end{matrix} \]
along $R \rightarrow \kappa (P)$ yields
\[   \begin{matrix} S/{\mathfrak m}_PS  &\longrightarrow & ( S/{\mathfrak m}_PS ) [W_1, \ldots, W_m] \\ \uparrow & & \uparrow \\  \kappa (P)    & \longrightarrow & \kappa(P) \tensor_R (S[W_1, \ldots, W_m]  )^\rho  \, .      \end{matrix} \]
Since the quotient of a finite nonmodular group is compatible with base changes on $X$ by Lemma \ref{actionuniversal}, we obtain
\[  \kappa(P) \tensor_R (S[W_1, \ldots, W_m]  )^\rho  \cong ( ( S/{\mathfrak m}_PS ) [W_1, \ldots, W_m]   )/ G \, .    \]
On the right, we have the action of $ G $ on
$ ( S_1 \times \cdots \times S_r) [W_1, \ldots, W_m] $, where the elements of $ G $ act transitively on the points above $P$. Therefore, the result follows from Lemma \ref{actiondisjointunion}.
\end{proof}

\begin{Lemma}
\label{actiondisjointunion}
Let $T$ and $W$ be schemes, and let $Y=T_1 \uplus \ldots \uplus T_r$ be the disjoint union of $r$ copies $T_i \cong T$. Let $G$ be a finite group acting on $Y$ and on $W$, and suppose that $G$ acts transitively on the set $\{ T_1, \ldots, T_r \}$. Let $H_1\subseteq G$ be the subgroup of elements $h$ with $h(T_1)=(T_1)$. Then, there is a natural isomorphism $(T_1 \times W)/H_1 \cong (Y \times W)/G$.	
\end{Lemma}
\begin{proof}
The natural inclusion $T_1 \times W \rightarrow Y \times W $ is $H_1-G$-equivariant and induces a morphism
\[  (T_1 \times W)/H_1 \longrightarrow (Y \times W)/G \, . \]
For each $i$, we fix a $g_i \in G$ with $g_i (T_i) =T_1$, so this is a representing system for $G/H$. We look at the morphism
\[  (Y \times W)  = \biguplus_{i} (  T_i \times W) \stackrel{\biguplus g_i}{ \longrightarrow } T_1 \times W  \stackrel{\pi}{ \longrightarrow} (T_1 \times W)/H_1 \, .  \]
This morphism is $G$-invariant: for $\tilde{g} \in G$ and some $T_i$ let $\tilde{g} (T_i)=T_j$. Then $ g_j \tilde{g} g_i^{-1}  \in H_1$ and hence $\pi g_i$ and $\pi   g_j \tilde{g}$ coincide on each $T_i$. Hence, this induces a morphism
\[  (Y \times W)/G   \longrightarrow   (T_1 \times W)/H_1      \, . \]
These two morphisms are inverse to each other.	
\end{proof}

\begin{Remark}
\label{fiberdimension}
Theorem \ref{fiberdescription} also shows that all fibers of $Z_\rho \rightarrow X$ have the same dimension $m$. This is in contrast to the symmetric algebra of a module, where the dimension of the generic fiber is the rank of the module and where the dimension can increase over special points, as in Example \ref{modulealgebradimension}. See Section \ref{fiberflatsection}.
\end{Remark}

\begin{Remark}
For a linear action $\beta$, the fiber ring over the vertex point $[0] \in X={\mathbb A}^d/\beta$  is $ K[X_1, \ldots, X_d]/ R_+ K[X_1, \ldots, X_d]  $, which is also called the \emph{ring of coinvariants} (and $R_+ K[X_1, \ldots, X_d] $ is called the \emph{Hilbert ideal} or \emph{Hilbert ideal of the null cone}). The only point above $[0]$ is $0$, and the stabilizer group is $G$ itself. For another linear representation $\rho$, the fiber ring is hence
\[  ((  K[X_1, \ldots, X_d]/ R_+ K[X_1, \ldots, X_d])[W_1, \ldots, W_m]  )^\rho \, ,\]
in accordance with Theorem \ref{fiberdescription}. For an interpretation of the length of the ring of coinvariants as a Hilbert-Kunz multiplicity, see \cite{brennermondalsteinberg}.
\end{Remark}

\begin{Example}
For $\Z/(k$) acting on $K[X]$, as in Lemma \ref{cyclicone}, the fiber ring of $ {\mathbb A}^1 \rightarrow {\mathbb A}^1 $ over the origin is $K[X]/(X^k)$, and the fiber ring of $Z_\ell$ for the representation $\rho_\ell$ is the $K$-algebra (see also Example \ref{artinian})
\[  ( K[X]/(X^k)  [W])^{ \beta \times \rho} =  K[ X^iW^j | \, i + \ell j \in \Z k ]/(X^k)  \, . \]
\end{Example}

\begin{Example}
\label{cyclicone5fiber}
We determine the fiber rings over the origin for $B_2$ and $B_3$ in Example \ref{cyclicone5}. We get 
\[ B_2 \tensor_{K[A]} K  = K[B,C,D]/( C^2,CD^2, D^3-BC)  \]
and
\[ B_3 \tensor_{K[A]} K  = K[B,E,F]/( FE^2, E^3, F^2-BE) \, .  \]
These are nonisomorphic $K$-algebras: if we go modulo the third power of the maximal ideal $\mathfrak m$, we get $    K[B,C,D]/( {\mathfrak m}^3 + ( C^2,BC)) $ with $K$-dimension $8$ in the first case and  $K[B,E,F]/( { \mathfrak m}^3 + ( F^2-BE))$ with dimension $9$ in the second case.
\end{Example}

\begin{Example}
\label{kleinactionfibers}
We want to analyze the two cases of Example \ref{kleinaction} along Theorem \ref{fiberdescription}. For the point $(a,b)=(0,0)$, the fiber ring is $K[X,Y]/(X^2,Y^2)$ with no further decomposition, and the stabilizer group is  $\Z/(2 ) \times \Z/(2 )$. For $(a,b)=(0,b)$ with $b \neq 0$, the preimage consists of $(0, \pm \sqrt{b})$, the fiber ring has the decomposition
\[  K[X,Y]/( X^2, Y^2-b) \cong K[X]/(X^2)   \times K[X]/(X^2)     \]
and the stabilizer group $H$ is $\Z/(2 )$ (the first component of $G=\Z/(2 ) \times \Z/(2 )$). For $(a,b)=(a,0)$ with $a \neq 0$, we have a similar behavior, and the stabilizer group is now the second component.
	
In the first case, where $G$ acts via the first component and then through negation, over $(0,b)$, the induced action is negation, hence, the fiber ring is $( K[X]/(X^2)[W])^H = K[W^2,XW]/((XW)^2)$, over $(a,0)$, the induced action is trivial, and hence, the fiber ring is $( K[Y]/(Y^2)[W])^H =K[W]$.
	
In the second case, the invariant algebra has the form $K[A,B] \subseteq K[A,B,C,D]/(D^2-ABC)$, the fiber ring over every point $ (a,b) \in V(AB) $ is just $K[C,D]/(D^2)$. Above $(0,b)$, the induced action of the stabilizer group on $ K[X]/(X^2) [W]$ is negation, the fiber ring is therefore $K[W^2,XW]/((XW)^2)$, and the same holds over $(a,0)$.
\end{Example}

\begin{Remark}
\label{notreducednoreconstruction}
The fiber ring over the origin $(0,0)$ in both cases of Example	\ref{kleinactionfibers} is $K[C,D]/(D^2)$. This shows that one cannot reconstruct the representation $\rho$ from the fiber over the origin of $Z_\rho$ alone, see Theorem \ref{pullback} for our main reconstruction result.
\end{Remark}

\begin{example}
We analyze Example \ref{symmetricdeterminant3} along Theorem \ref{fiberdescription}. For a point $Q=(x,y,z)$ with three different entries, the decomposition in Theorem \ref{fiberdescription} consists in $6$ reduced points, the stabilizer group is trivial and the fiber of $Z_\rho$ above the image point $P$ is a reduced affine line. For a point with two different entries, say $Q=(0,0,1)$, there are three points above the image point $P=(1,0,0)$. The fiber ring is
\[  K[X,Y,Z]/(X+Y+Z-1,  XY+XZ+YZ,XYZ) \, ,\]
the ring in the decomposition corresponding to $Q$ is $K[X]/(X^2)$ (via $Y \mapsto -X$, $Z \mapsto 1$), the stabilizer group is $S_2$, with the induced action given by $X \mapsto-X$. The fiber of $Z_\rho$ above $P$ is given by $ ( K[X](X^2)[W])^{\tilde{\rho} }$, which is $K[XW,W^2]/((XW)^2)$, in accordance with Lemma \ref{symmetricdeterminant}. For a point with only one entry, say $Q=(0,0,0)$, there is no other point above the image point $P=(0,0,0)$. The fiber ring is $  K[X,Y,Z]/(X+Y+Z,  XY+XZ+YZ,XYZ)$, the stabilizer group is $S_3$. The fiber of $Z_\rho$ above $P$ is described by  \[ ( K[X,Y,Z]/(X+Y+Z, XY+XZ+YZ,XYZ)[W])^{ \rho } \, , \]
which is $K[ \triangle W,W^2]/(( \triangle W)^2)$.
\end{example}

We now describe the fibers of a quotient scheme $Z_\rho$ over $X$ as a topological space (the geometric fiber). It turns out that they are homeomorphic to quotients of affine spaces modulo linear actions.

\begin{Lemma}
\label{fiberartiniangeometric}
Let a finite group $G$ act faithfully via $\beta$ by $K$-algebra auto\-morphisms on a local Artinian $K$-algebra $S$ with residue class field $K$ and let $\rho$ be a linear representation of $G$. Then
\[ (S [W_1, \ldots, W_m])^{\beta \times \rho} \longrightarrow K[W_1, \ldots, W_m]^\rho  \]
is the reduction. The corresponding map of spectra is a homeomorphism.
\end{Lemma}
\begin{proof}
The homomorphism is well-defined since $g(s)$ and $s$ have the same value in $S/{\mathfrak m}$, because of $s - \overline{s} \in {\mathfrak m}$ and since the automorphisms preserve the maximal ideal. There is also a subring relation
\[  K[W_1, \ldots, W_m]^\rho   \subseteq    (S [W_1, \ldots, W_m])^{\beta \times \rho}     \]
and so the homomorphism is surjective (and the morphism of the spectra is a closed embedding). The right-hand side above is reduced, and both rings are irreducible and of the same dimension. Hence, the closed embedding is the reduction.
\end{proof}

\begin{Theorem}
\label{geometricfiber}
Let $K$ be an algebraically closed field. Let a finite nonmodular group $G$ act via $\beta$ faithfully on $Y= \Spec S$ of finite type over $K$ by $K$-algebra automorphisms with quotient scheme $X$, let $Q \in Y $ be a closed point with image point $P \in X$. Let $\rho$ be a linear action of $G$ on $ {\mathbb A}^m $ with quotient scheme $Z_\rho= (Y \times {\mathbb A}^m)/(\beta \times \rho)$. Let $H \subseteq G$ denote  the stabilizer group of  $Q$. Then, the geometric fiber of $Z_\rho$ over $P$ is naturally homeomorphic with $ {\mathbb A}^m/ \tilde{\rho}$, where $\tilde{\rho}$ is the restriction of $\rho$ to $H$.
\end{Theorem}
\begin{proof}
By Theorem \ref{fiberdescription}, the scheme theoretic fiber over $P$ has the form of Lemma \ref{fiberartiniangeometric}, so this lemma gives the result.	
\end{proof}

\begin{Corollary}
\label{centralfiber}
Let $K$ be an algebraically closed field. Let the finite nonmodular group act linearly and faithfully on $ {\mathbb A}^d$ with quotient scheme $X$ and vertex point $ \tilde{0} $. Let $\rho$ be a linear action on $ {\mathbb A}^m$ with quotient scheme $Z_\rho \rightarrow X$. Then, there is a natural homeomorphism between the fiber of $Z_\rho$ over $ \tilde{0} $ and $ {\mathbb A}^m /\rho$.
\end{Corollary}
\begin{proof}
This is a special case of Theorem \ref{geometricfiber}, as for $0 \in {\mathbb A}^d$ the stabilizer is $G$, hence, $ \tilde{\rho}  = \rho $.
\end{proof}

\begin{Remark}
Every quotient scheme of a linear group action on $ {\mathbb A}^d $ appears as the geometric fiber over the vertex point $\tilde{0}$ of some quotient scheme, as follows from Corollary \ref{centralfiber}.
\end{Remark}

\begin{Example}
If $G$ acts freely on ${ \mathbb A}^d \setminus \{0\} $, then by Theorem \ref{geometricfiber}, all fibers of $Z_\rho$ are affine spaces ${\mathbb A}^m$ with the exception of the fiber over the vertex, which is homeomorphic to ${\mathbb A}^m/ \rho$ according to Corollary \ref{centralfiber}.
\end{Example}

\begin{Corollary}
The geometric fibers of a quotient scheme of rank one are affine lines.
\end{Corollary}
\begin{proof}
This follows from Theorem \ref{geometricfiber}, as the quotient of ${ \mathbb A}^1$ modulo a finite group acting linearly is, as it factors through a linear action of a cyclic group, always an affine line.
\end{proof}

\section{Ramification}
\label{ramificationsection}

We want to understand which fibers of $Z_\rho \rightarrow X$ are reduced. For the trivial bundle, which stems from the trivial action, all fibers are reduced. By Lemma \ref{modulealgebrafree}, the fibers over the points that lie in the image of the free locus of $Y$ are also reduced since there we even have the structure of a vector bundle.

\begin{Lemma}
\label{nonreducedpoint}
Let $G$ be a finite group and let $S$ be a local Artinian nonreduced $K$-algebra, $Y=\Spec S$. Let $G$ act faithfully on $S$ by $K$-algebra automorphisms with invariant ring $K$. Let $\rho$ be a nontrivial linear representation of $G$. Then, the quotient scheme $(Y \times {\mathbb A}^m)/G$ is not reduced.
\end{Lemma}
\begin{proof}
Let $H \neq G$ be the kernel of the representation $\rho$. We mod out everywhere $H$, which gives the same situation back, with the extra property that $ G/H $ acts faithfully on $ {\mathbb A}^m $. Note that $ Y/H $ is also nonreduced, otherwise, it would already be $\Spec K$. So we may assume that $G$ acts faithfully on ${\mathbb A}^m$. Let $V \subseteq {\mathbb A}^m  $ be a nonempty open subset where the action is free, which exists by Lemma \ref{genericfree}. We consider the action of $ G $ on $ Y \times V $. This action is also free by Lemma \ref{freebasechange}, so $ Y \times V \rightarrow (Y \times V)/G $ is \'etale by Remark \ref{freeprincipal}. Since $Y \times V$ is not reduced,  $ (Y \times V)/G$ is also not reduced by \'etale descent, see \cite[Lemma 10.163.7]{stacksproject}, hence, $(Y \times {\mathbb A}^m)/G$ is not reduced because it contains  $ (Y \times V)/G$ as an open subset.
\end{proof}

\begin{Theorem}
\label{nonreducedfiber}
Let $K$ be algebraically closed. Let $G$ be a finite nonmodular group acting via $\beta$ by $K$-algebra automorphisms on a $K$-algebra $S$ of finite type, and let $Y= \Spec S$ with quotient scheme $X=\Spec R$. Let $\rho$ be a linear representation of $G$. Then, the fiber of the quotient scheme $Z_\rho= ( Y \times {\mathbb A}^m)/{\beta \times \rho}$ over a closed point $P \in X$ is nonreduced if and only if for one (any) point $Q \in Y$ above $P$, the restriction  of $\rho$ to the stabilizer of $Q$ is not trivial.
\end{Theorem}
\begin{proof}
If $P$ belongs to the image of the free locus of the action on $Y $, then the fiber of the vector bundle (by Corollary \ref{modulealgebraisomorphism}) is reduced, and the stabilizer group of $Q$ is trivial. So suppose that $P$ does not lie in the image of the free locus. Then above $P$ there are in $Y$ less than $|G|$ points. The fiber ring $S/{\mathfrak m}_P S = S_1 \times \cdots \times S_r$ has $K$-dimension at least the order of the group, and this is distributed equally on the local rings of it, so they cannot be reduced. By Theorem \ref{fiberdescription}, the fiber of $Z_\rho$ above $P$ is $ S_1 [W_1, \ldots, W_m]^{\tilde{\rho} }$, where $\tilde{\rho}$ is the action restricted to the stabilizer of some point $Q$ above $P$. If $\tilde{\rho}$ is trivial, then the fiber ring is $K[W_1, \ldots, W_m]$, hence reduced. 	If $\tilde{\rho}$ is not trivial, then the result follows from Lemma \ref{nonreducedpoint}.
\end{proof}

\begin{Remark}
Typical examples for nilpotent elements in the fibers described in Theorem \ref{nonreducedfiber} are mixed invariants, as given in Lemma \ref{charactercontribution}.
\end{Remark}

\begin{Corollary}
\label{reflectiongroupramification}
Let $G \subseteq \GLG_d(K)$ be a reflection group with the mirror hyperplanes $H_1, \ldots , H_n$, let $\rho:G \rightarrow \GLG_m(K)$ be a linear representation, and let $Z_\rho \rightarrow X={\mathbb A}^d/G$ be the corresponding quotient scheme. If $P \in X$ lies in the image of $ \bigcup_{i = 1}^n H_i$, then the fiber of $Z_\rho$ above $P$ is nonreduced.
\end{Corollary}
\begin{proof}
Outside the union $ \bigcup_{i = 1}^n H_i$, the action is free (see Corollary \ref{reflectiongroupsubgroup}), therefore, we have a vector bundle by Corollary \ref{modulealgebralinearisomorphism} above the image of this union.
\end{proof}

Therefore, nonreduced fibers can occur only above the zero locus of the discriminant, see Remark \ref{discriminant}.

\begin{Corollary}
\label{reflectiongroupdeterminantramification}
Let $G \subseteq \GLG_d(K)$ be a reflection group with the mirror hyperplanes $H_1, \ldots , H_n$, let $\rho:G \rightarrow K^\times$ be the representation given by the determinant, and let $Z_\rho \rightarrow X={\mathbb A}^d/G$ be the corresponding quotient scheme. Then, the fiber of $Z_\rho$ above a point $P \in X$ is nonreduced if and only if $P$ lies in the image of $ \bigcup_{i = 1}^n H_i$.
\end{Corollary}
\begin{proof}
One direction is a special case of Corollary \ref{reflectiongroupramification}. Let $Q \in \bigcup_{i = 1}^n H_i$ be a point above $P$, say $Q \in H_1$. Let $g \neq \operatorname{Id} $ be a reflection with $H_1$ as fixed space. Then, $g$ belongs to the stabilizer of $Q$, and $\rho(g)= \operatorname{det} (g) \neq 1$, hence, the fiber is not reduced by Theorem \ref{nonreducedfiber}.
\end{proof}

Corollary \ref{reflectiongroupdeterminantramification} can be applied in particular in the situation of Lemma \ref{symmetricdeterminant}, but in this case, the result is also clear from the explicit equation.

\begin{Example}
Theorem \ref{nonreducedfiber} does not hold in this form for a field that is not algebraically closed, as any Galois extension $K \subseteq L$ shows, where (see Lemma \ref{modulealgebrafree} and Example \ref{galoisfree}) the (only) fiber of $Z_\rho$ is always reduced, the stabilizer group is always $G$ and the action $\rho$ need not be trivial.	
\end{Example}

\section{Products}
\label{productssection}

We want to understand how the sum of two linear representations $\rho_1$ and $ \rho_2$ is related to the product of $Z_{\rho_1}$ and $Z_{\rho_2}$ over $X$.

\begin{Lemma}
\label{productcompatibilityintegral}
Let $G$ be a finite group $G$ acting on the affine $k$-schemes $X,Y,Z$ with compatible morphisms $\varphi, \psi:Y,Z \rightarrow X$. Then, there is a natural surjective integral  morphism
\[ (Y \times_X Z)/G \longrightarrow  Y/G \times_{X/G}  Z/G \, .  \]
\end{Lemma}
\begin{proof}
The group $G$ is also acting on $ Y \times_X Z $, therefore, we obtain morphisms $ (  Y \times_X Z   )/G \rightarrow  Y/G $ and  $ (  Y \times_X Z   )/G \rightarrow  Z/G $, and hence, by the universal property of the product, a morphism
\[ (Y \times_X Z)/G \longrightarrow  Y/G \times_{X/G}  Z/G \, .  \]
An $L$-point on the right ($K\subseteq L$ a field extension) is given as $([y], [z])$ over $[x]$. Starting with $y$, we may assume that $\varphi(y) =x$. There exists $g \in G$ such that $g( \psi(z)) = x$. Then, $[ ( y,gz)]$ is a preimage of $ ([y], [z])$. To show finiteness, we look at the ring homomorphism
\[ A^G \tensor_{S^G} B^G \longrightarrow A \tensor_S B \, .\]
The elements of the form $a \tensor 1$ fulfil an integral equation over $A \tensor 1$, and the elements of the form $1 \tensor b$ fulfil an integral equation over $1 \tensor B$, hence, the map is integral.
\end{proof}

This morphism is in general not an isomorphism, as basic examples such as the following show.

\begin{example}
\label{cycliconeproduct}
We consider Example \ref{cycliconespecial} and we look at the tensor product
\[ \begin{aligned} A_{-1} \tensor_R A_{-1} & = K[ X^k,W^k,XW ]   \tensor_{K[X^k]} K[ X^k,\tilde{W}^k,X \tilde{W} ] \\ & \cong K[A] [B,C][ \tilde{B}, \tilde{C}] /(AB-C^k, A \tilde{B} - \tilde{C}^k) \, .   \end{aligned}  \]
This is a subring of $K[X,W, \tilde{W} ]^G$, but in this ring, we also have elements such as $W \tilde{W}^{k-1}$, which do not belong to the tensor product (its $k$th power does).
\end{example}

Note that the invariant algebras in the example above are normal, and the tensor product is not. In the example for $k=2$, we have $AB A \tilde{B} =C^2 \tilde{C}^2$, which implies $B \tilde{B} = \frac{C^2 \tilde{C}^2}{A^2}$, hence $B \tilde{B}$ has a square root in the quotient field, and this square root, which corresponds to $W \tilde{W}$, belongs to the normalization of the tensor product, and in fact, $K[X,W, \tilde{W} ]^G$ is the normalization of $K[X, W ]^G \tensor_{K[X  ]^G } K[X,  \tilde{W} ]^G $. It is not the seminormalization, as the points with coordinates $ (0,b,0, \tilde{b} ,0 ,\pm \sqrt{b \tilde{b} } ) $ are mapped to the same point. See also Example
\ref{quotientschemeadditionfails}.

In the following, we mean by normalization first going to the reduction and then normalize, i.e., taking the integral closure within the total ring of fractions. In our case, we only have one minimal prime, so after reduction, the normalization is inside the quotient field. For an example in a slightly different context where nilpotent elements occur, see Example \ref{pullbacknilpotent}.

\begin{Theorem}
\label{normalproductcompatibility}
Let $G$ be a finite group acting via $\beta$ on a normal $K$-domain $S$ through $K$-algebra automorphisms with invariant ring $R=S^G$. Suppose that there exists a nonempty invariant open subset of $Y$ where the action is free. Let $\rho_1$ and $\rho_2$ be linear representations of $G$ giving rise to the quotient schemes  $Z_{ \rho_1} $ and  $Z_{\rho_2} $ over $X=\Spec R$. Then the morphism
\[  Z_{\rho_1 \times \rho_2 }   \longrightarrow Z_{\rho_1} \times_{\Spec R}  Z_{ \rho_2 }   \]
from Lemma 	\ref{productcompatibilityintegral} is the normalization.
\end{Theorem}
\begin{proof}
Let $Y=\Spec S$ and $\tilde{Z}_i =Y \times {\mathbb A}^{m_i} $, where $G$ is acting via $\beta \times \rho_i$. We look at the commutative diagram
\[ \begin{matrix} & & \tilde{Z}_1 \times_{Y } \tilde{Z}_2 & & \longrightarrow & & Z_1 \times_X Z_2  \\ & \swarrow & & \searrow &   &\swarrow & & \searrow   \\
&	\tilde{Z}_1   & & 	\tilde{Z}_2 &  &  Z_1   & & Z_2 \\ & \searrow & & \swarrow &   & \searrow & & \swarrow   \\ & &Y  & & \longrightarrow & & X &	
\end{matrix}\]
which induces as in Lemma \ref{productcompatibilityintegral}    a morphism
\[   ( \tilde{Z}_1 \times_{Y } \tilde{Z}_2 ) /(( \beta \times \rho_1 ) \times (\beta \times \rho_2))   \longrightarrow    Z_1 \times_X Z_2  \, . \]
This morphism is surjective and integral.
		
We have $\tilde{Z}_1 \times_{Y } \tilde{Z}_2 = (Y \times {\mathbb A}^{m_1})  \times_{Y }(Y \times {\mathbb A}^{m_2})  =Y \times{\mathbb A}^{m_1} \times  {\mathbb A}^{m_2}  $ and	
\[  Z_{\rho_1 \times \rho_2} = (Y \times{\mathbb A}^{m_1} \times  {\mathbb A}^{m_2})(\beta \times \rho_1 \times \rho_2) =    ( \tilde{Z}_1 \times_{Y } \tilde{Z}_2 ) /(( \beta \times \rho_1 ) \times (\beta \times \rho_2))     \, . \]
Therefore, this is an integral scheme, and its quotient scheme is a normal scheme. Because of the surjectivity, $Z_1 \times_X Z_2$ is irreducible. Hence its reduction is an integral scheme.
	
Let $V \subseteq Y$ be a nonempty invariant open subset where the action is free and let $U\subseteq X$ be its image. Then $Z_{\rho_1} |_{U}$,  $Z_{\rho_2} |_{U}$ and  $Z_{\rho_1 \times \rho_2} |_{U}$ are vector bundles over $U$ by	Corollary \ref{modulealgebraisomorphism}. The morphism restricts to an isomorphism
\[  Z_{\rho_1 \times \rho_2 }|_U   \longrightarrow Z_{\rho_1}|_U \times_U  Z_{ \rho_2 } |_U \, ,  \]
which shows that the morphism is birational.	
\end{proof}

\begin{Remark}
The morphism from Theorem \ref{normalproductcompatibility} commutes with the morphism from Lemma \ref{modulealgebra}, i.e.
we have a commutative diagram
\[  \begin{matrix}  Z_{\rho_1 \times \rho_2 }  & \longrightarrow& Z_{\rho_1} \times_{\Spec R}  Z_{ \rho_2 }  \\ \downarrow & & \downarrow \\ \Spec(  \Sym_R(M_1 \oplus M_2)) & \longrightarrow &  \Spec ( \Sym_R(M_1) ) \times_{\Spec R}\Spec ( \Sym_R(M_2) )  \, ,
 \end{matrix} \]
where on the bottom row we have an isomorphism.
\end{Remark}

\begin{Corollary}
\label{productcompatibility}
Let $\beta$ be a faithful linear representation of a finite group $G$ on $K[X_1, \ldots, X_d]$ with invariant ring $R=K[X_1, \ldots, X_d]^G$  and let $\rho_1$ and $\rho_2$ be  linear representations of $G$, giving rise to the quotient  schemes  $Z_{ \rho_1} $ and  $Z_{\rho_2} $. Then, the normalization of $Z_{\rho_1} \times_{\Spec R}  Z_{\rho_2} $ is $ Z_{\rho_1 \times \rho_2 } $.
\end{Corollary}
\begin{proof}
This follows from Theorem \ref{normalproductcompatibility} and Lemma \ref{genericfree}.	
\end{proof}

\begin{Remark}
Without a faithful basic action there is no birational relation between the objects in Corollary \ref{productcompatibility}. If we consider the group $\Z/(2)$ with the trivial basic action on $K$ and the  cyclic action on $K[X]$, then the invariant algebra is $K[X^2]=K[A]$, the tensor product of the two copies is $K[A, B ]$, but the invariant ring of the tensor ring is $K[X^2,Y^2,XY ] \cong K[A, B] [C]/(C^2-AB)  $, which is not birational to $K[A,B]$ with the given embedding.
\end{Remark}

\begin{example}
The relation  between the product $Z_{\rho_1} \times Z_{\rho_2}$ and $Z_{\rho_1 \times \rho_2}$ is generally more complicated than in the normal case described in Theorem \ref{normalproductcompatibility}. We look at Example \ref{artinian} for $k=2$ and $\ell=1$ taken twice. The ring homomorphism is
\[   K[XW,W^2] \tensor_K[X \tilde{W}, \tilde{W}^2] \longrightarrow  K[XW,W^2,X \tilde{W}, \tilde{W}^2  , W \tilde{W} ] \, \]
with the relations $(XW)^2=0$, $(X \tilde{W})^2=0$ on the left and more relations on the right. If we reduce, we get $K[W^2, \tilde{W}^2] \rightarrow K[W^2, \tilde{W}^2,W \tilde{W}]$, which is not the normalization.	 
\end{example}

\section{Regular representation}
\label{regularrepresentationsection}

For a finite group $G$, the \emph{regular representation} is the representation on $K^{|G|}$ given by $g \in G$ acting through the permutation matrix $e_h \mapsto e_{gh}$, where $e_h$ is the standard vector indexed by $h \in G$. In the nonmodular case, this representation has a decomposition
\[ (K^{|G|}, \operatorname{reg} ) \cong \bigoplus_{\rho \in C } V_\rho^{ \oplus  \operatorname{dim} V_\rho } \, , \]
where the index set runs through a system $C$ of all irreducible representations, see \cite[Corollary 1 in I.2.4]{serrerepresentation}. $G$ acts on the polynomial ring $S[K^{|G|} ]= \Sym (S^{|G|})$ with invariant algebra $(S[K^{|G|} ])^G$. Here, $g$ sends the variable $W_h$ to $W_{ h g} $.

\begin{Theorem}
\label{regularrepresentation}
Let $S$ be a normal $K$-domain, and $G$ be a finite nonmodular group acting on $S$ by $K$-algebra automorphisms with invariant ring $R$. Let $C$ be a collection of all irreducible linear representations of $G$, with corresponding invariant algebras $B^{\rho}$ for $\rho \in C$. Let $G$ act regularly on $K^{|G|}$. Then the invariant algebra is (tensor is over $R$)
\[S[K^{|G|}]^G \cong ( \bigotimes_{  \rho \in C} ( B^\rho)^{\tensor \operatorname{dim} (\rho) } )^{\operatorname{norm}} \,. \]
\end{Theorem}
\begin{proof}
This follows from Theorem \ref{normalproductcompatibility} and the decomposition of the regular representation into irreducible representations. 
\end{proof}

\begin{Lemma}
\label{regularfirstdegree}
The first graded component of the invariant algebra $S[K^{|G|}]^G $ of the regular representation is $R$-isomorphic to $ \bigoplus_{\rho \in C} ((  S \tensor   V_\rho)^G)^{ \oplus  \operatorname{dim} V_\rho }  $ and $R$-isomorphic to $S$ via the embedding $s \mapsto \sum_{h\in G}  ( s)h W_h$.
\end{Lemma}
\begin{proof}
From the equivariant decomposition $ K^{|G|}  \cong \bigoplus_{\rho \in C } V_\rho^{ \oplus  \operatorname{dim} V_\rho } $ we get
\[ S \tensor K^{|G|}  \cong \bigoplus_{\rho \in C}   S \tensor   V_\rho^{ \oplus  \operatorname{dim} V_\rho }  \cong \bigoplus_{\rho \in C} (  S \tensor   V_\rho)^{ \oplus  \operatorname{dim} V_\rho }   \, .  \]
Passing to the invariant modules yields
\[ (  S \tensor K^{|G|}   )^G \cong \bigoplus_{\rho \in C} ( ( S \tensor   V_\rho)^G )^{ \oplus  \operatorname{dim} V_\rho }   \, .  \]
This is the degree $1$-part of the invariant algebra by Lemma \ref{modulealgebra}.
	
To prove the second statement, we look at the group isomorphism
\[  S[K^{|G|}]_1 \longrightarrow S[K^{|G|}]_1 \, , s T_h \mapsto (s)h W_h \, .  \]
Here, the action of $g \in G$ on the left is by permuting $T_h \mapsto T_{ hg}$ (no action on $S$) and on the right by $ s W_h \rightarrow (s)g W_{hg}$ (which is the natural action on  $S[K^{|G|}]$ restricted to the first degree component). This homomorphism is $G$-equivariant and $S^G$-linear. Hence, this induces an isomorphism between the invariant submodules. On the left-hand side, the invariant elements are elements of the form $\sum_{h \in G} s T_h$, so they correspond to the elements $s \in S$, and these elements are sent to $\sum_{h\in G} (s)h W_h $ on the right.
\end{proof}

\begin{Corollary}
Let $S$ be a normal $K$-domain, $G = \Z/(k) = \langle \zeta \rangle $, $k \neq 0$ in $K$, acting on $S$ via $\beta$ by $K$-algebra automorphisms with invariant ring $R$. Then, the invariant algebra of the regular representation is  
\[S[T_0,T_1, \ldots , T_{k-1}]^G \cong ( \bigotimes_{  \ell=0, \ldots ,{k-1} }  (  ( S[W_\ell])^G )^{\operatorname{norm}} \, , \]
where the generator acts on the left by $T_i \mapsto T_{i+1}$, and on the right by sending $W_\ell \mapsto \zeta^\ell W_\ell$. The mapping is given by $T_i \mapsto \sum_{\ell =0}^{k-1} \zeta^{i \ell} W_\ell $.
\end{Corollary}
\begin{proof}
The first statement is a special case of Theorem \ref{regularrepresentation}. The second statement follows from the explicit splitting of the regular representation in this case.
\end{proof}

\section{Module schemes}
\label{moduleschemesection}

In the following sections we will take a closer look at the question of what kind of object $Z_\rho \rightarrow X$ is, beside being a scheme over $X$.

Let $X$ denote a scheme. A well-known correspondence is that between (geometric) vector bundles over $X$ and locally free $ {\mathcal O}_X$-modules on $X$. A vector bundle is a scheme $Z \rightarrow X$ that looks locally like $X \times {\mathbb A}^r \rightarrow X $ and has the property that the transition mappings are linear. A locally free $ {\mathcal O}_X$-module looks locally like ${\mathcal O}_X^r$. The correspondence assigns to $Z$ the dual sheaf of the sheaf of sections and to a locally free sheaf $\mathcal F$ the (relative) spectrum of the symmetric algebra $\Sym ({\mathcal F}) = \oplus_{n \in \N} \Sym^n ({\mathcal F}) $, \cite[Exercise II.5.18]{hartshorne}. This last definition can be applied for any quasicoherent $ {\mathcal O}_X$-module $\mathcal F$, see \cite[Definition 1.7.8]{ega2}, and \cite[Section 1.3]{ega2} for the relative spectrum of a quasicoherent algebra.

A vector bundle is in particular a module scheme (schéma de modules) in the sense of the following definition.

\begin{definition}
A \emph{module scheme} $Z$ over a scheme $X$ is a commutative group scheme $Z \rightarrow X$ together with a scheme morphism
\[ \cdot : {\mathbb A }^1_X \times_X Z  \longrightarrow Z \]
over $X$, fulfilling the natural compatible conditions from module theory.
\end{definition}

This formulation requires that we consider $ {\mathbb A }^1_X $ with the natural addition and multiplication as a ring scheme (this term appears also in \cite[Section 26]{mumfordbergman}, the affine line in the only ring scheme we are dealing with). If $\alpha: Z \times_X Z \rightarrow Z$ denotes the addition of the group scheme, then this definition means, to give some examples, that the diagram
\[  \begin{matrix} {\mathbb A }^1_X \times_X Z \times_X Z & \stackrel{ \cdot_{12} \times \cdot_{13} }{ \longrightarrow } & Z \times_X Z  \\  \operatorname{Id}_{  {\mathbb A }^1_X } \times \alpha  \downarrow & & \downarrow \alpha  \\{\mathbb A }^1_X \times_X Z & \stackrel{\cdot}{\longrightarrow} &  Z \end{matrix}\]
commutes (distributivity in the vectors), and that the diagram
\[  \begin{matrix} {\mathbb A }^1_X \times_X {\mathbb A }^1_X \times_X Z & \stackrel{ \cdot_{13} \times \cdot_{23} }{ \longrightarrow } & Z \times_X Z  \\ \alpha_{  {\mathbb A }^1_X } \times \operatorname{Id_Z}  \downarrow & & \downarrow \alpha  \\{\mathbb A }^1_X \times_X Z & \stackrel{\cdot}{\longrightarrow} &  Z \end{matrix}\]
commutes (distributivity in the scalars). 

The concept of a module scheme is referred to briefly in 
\cite[Chapitre 0.8]{ega31} (in particular 0.8.2.3), Grothendieck writes ``L'exemple des groupes a été traité avec assez de détails, mais par la suite nous laisserons généralement au lecteur le soin de développer des considérations analogues dans les exemples de structures algébriques que nous rencontrerons'' (see \cite[8.2.8]{ega31}). The concept is also mentioned in \cite[4.3.3]{sga3}, but a thorough study is missing. Related work, but only over a field, has been done by Greenberg (\cite{greenberglocal}, \cite{greenbergalgebraic}, where module variety over a (noncommutative) ring variety is used) and, in the context of graded Hop algebras, by Milnor and Moore (\cite{milnormoorehopf}, see also Remark \ref{milnormoorehopf} below). See also \cite{raynaudppp} for the notion of schéma en $F$-vectoriels.

The following lemma is stated in \cite[1.7.13]{ega2} (see also \cite[Section 9.4]{ega1}. Grothendieck sets $V({\mathcal E}) =\Spec ( \Sym {\mathcal E} )  $ and calls this, quite inappropriately, fibr\'{e} vectoriel (\cite{sga3} uses fibration vectorielle instead), and writes ``Nous interpré\-terons ces faits en disant que $S[T]$ est un $S$-schéma d'anneaux et que $ V( { \mathcal E})$ est un $S$-schéma de modules sur le S-schéma d'anneaux (cf. chap. 0, § 8)''.

\begin{Lemma}
\label{symmetricmodulescheme} 
For a quasicoherent module $ {\mathcal F}$ on a scheme $X$, $\Spec ( \Sym {\mathcal F} ) $ is a module scheme over $X$.
\end{Lemma}
\begin{proof}
We look at the affine situation with a commutative ring $R$ and an $R$-module $M$. By \cite[Proposition 1.7.11(iii)]{ega2}, there is a canonical isomorphism $\Sym (M) \tensor_R \Sym (M) \cong \Sym (M \oplus M)$. The diagonal $R$-module homomorphism $M \rightarrow M \oplus M$ gives rise to the coaddition (in the sense of Hopf-algebras)
\[\Sym (M) \longrightarrow  \Sym (M) \tensor_R \Sym (M) \text{ induced by }  m \mapsto   m \tensor 1 + 1 \tensor m  \, , \]
the conegation is induced by $m \mapsto -m$, and the zero section comes from $  \Sym (M) \rightarrow R  $, which kills the positive part. The co-scalar-multiplication is given through
\[   \Sym M  \longrightarrow  R[T] \tensor_R  \Sym M  ,\, m_1 \cdots m_r \mapsto T^r \tensor m_1 \cdots m_r \, . \]
The axioms of a module scheme are fulfilled, for example, the first diagram above is commutative because of $T \tensor (m \tensor 1  + 1 \tensor  m )
= T \tensor m \tensor 1 + T\tensor 1 \tensor m$ for $m \in M$.
\end{proof}

\begin{Example}
For the free $R$-module $R^r$, the symmetric algebra is the polynomial algebra $R[T_1, \ldots, T_r]$.The coaddition is given by $T_i \mapsto T_i+S_i$.
\end{Example}

\begin{example}
\label{moduleschemepresentation}
If $R$ is a commutative ring and $M$ a finitely generated $R$-module with a presentation $R^m \stackrel{A}{\rightarrow} R^n \rightarrow M \rightarrow 0$ with an $n \times m$-matrix $A$, then the symmetric algebra has the description
\[ \Sym M = R[T_1, \ldots , T_n]/( a_{11} T_1 \plusdots a_{n1}T_n ,\ldots,  a_{1m} T_1 \plusdots a_{nm}T_n ) \, , \] 
see \cite{vasconcelosarithmetic}. For the use of symmetric algebras and their torsors to understand closure operations, see \cite{brennerberkeley}.
\end{example}

\begin{example}
\label{kaehlerdifferential}
If $R=K[X_1, \dots, X_n]/(g_1, \ldots, g_m)$, then the Jacobian matrix
$J=( \partial g_i / \partial X_j )_{ij} $ gives a presentation $R^m \stackrel{J^{\operatorname{tr}} }{ \rightarrow} R^n \rightarrow \Omega_{R|K} \rightarrow 0$ for the $R$-module of K\"ahler differentials as in Example \ref{moduleschemepresentation}. The module scheme $\Spec (\Sym (\Omega_{R|K} )) \rightarrow \Spec R$ restricts over the smooth locus $U \subseteq X$ to the tangent bundle $T_U$, and provides thus a natural extension for the tangent bundle above $X$.
\end{example}

\begin{example}
If $R$ is a commutative ring with a maximal ideal ${\mathfrak m}=(x_1, \ldots, x_d)$, then $\Sym ( R/{\mathfrak m}) \cong R[T]/(x_1 T, \ldots, x_d T)$. The corresponding module scheme is a sky scraper module scheme over $\Spec R$, which is outside the maximal ideal an isomorphism, but the fiber above the maximal ideal is an affine line.
\end{example}

\begin{example}
\label{idealsymmetricalgebra}
If $R$ is a commutative ring with an ideal $ I= (f_1, \ldots, f_n)$ (minimal generators), then
\[ \Sym ( I)  \cong R[T_1, \ldots, T_n]/(g_1T_1 \plusdots g_nT_n, \sum_{j=1}^n g_jf_j = 0) \,   . \]
Typical equations are the Koszul equations $f_iT_j -f_jT_i$, but there are more in general. A section $s: \Spec R \rightarrow \Spec (\Sym (I))$ in the module scheme given by $T_i \mapsto a_i$ corresponds to the $R$-linear map $I \rightarrow R,\, f_i \mapsto a_i$. The  module scheme is above $D(I) \subseteq \Spec R$ a trivial line bundle. If $I$ is primary to a maximal ideal $\idealm$, then the fiber above $\idealm$ is an $n$-dimensional affine space.
\end{example}

The jump in the dimension of the fibers is a typical behavior of the spectrum of a symmetric algebra. If the module is finitely generated, then the fiber dimension can increase by specializing. 

\begin{example}
Let $R$ be a discrete valuation domain with local parameter $\pi$ and let $Q=R_\pi$ be its quotient field. Then the symmetric algebra is
$\Sym Q =R[T_n, n \in \N]/( \pi T_{n+1} - T_n, n \in \N)$. Here, the generic fiber of the module scheme is a line and the special fiber is a point. 
\end{example}

\section{Module schemes and graded Hopf algebras}
\label{moduleschemehopfgradedsection}

An affine group scheme $\Spec A$ over $\Spec R$ corresponds to the structure of a commutative Hopf algebra structure on $A$. We want to understand what additional structure on the Hopf side is enforced by a module scheme.

\begin{Lemma}
\label{monoidactiongraded}
Let $R \subseteq A$ be an algebra over $R$. Then the following hold.

\begin{enumerate}
\item
An action $\mu$ of the multiplicative group scheme $ G_m = \Spec R[T,T^{-1}] $ on $Z= \Spec A$ is the same as a $\Z$-grading of $ A $ with $R \subseteq A_0$. 

\item 
An action of the multiplicative monoid scheme $ ( \Spec R[T], \cdot )$ on $Z= \Spec A $ is the same as an $\N$-grading of $A$ with $R \subseteq A_0 $. 

\item
In the situation described in (2), $R=A_0$ is equivalent with the existence of a section $0:X \rightarrow Z$ such that the diagram
\[  \begin{matrix} Z  &  \stackrel{ ( 0_{ {\mathbb A}^1} \circ p ) \times \operatorname{Id}_Z }{\longrightarrow}  & {\mathbb A}^1_X \times_X Z \\ p \downarrow & & \downarrow \mu \\ X & \stackrel{0}{ \longrightarrow  }&  Z \,  \end{matrix}  \]
commutes.

\end{enumerate}
	
\end{Lemma}

\begin{proof}
(1) follows from \cite[Proposition I.4.7.3]{sga3}. In this correspondence, for a $\Z$-graded algebra $A$, the cooperation is given as an $R$-algebra homomorphism
\[ A \rightarrow A[\Z]=A[T,T^{-1}] \cong A \tensor_R R[T,T^{-1}], a_n \mapsto a_nT^n \, , \]
for homogeneous elements $a_n$ of degree $n$. The elements $r \in R$ map to $r \tensor 1=1 \tensor r$ and $1$ lives in degree $0$, hence also $R$. An $R$-point of the multiplicative group is given by a ring homomorphism $R[T,T^{-1}] \rightarrow R$, i.e., by a unit $u \in R$, and the corresponding $R$-algebra automorphism is $A \rightarrow A, \, a_n \mapsto u^n a_n$. In particular, the conegation is given as $a_n \mapsto (-1)^n a_n$. Conversely, if we start with an action of the group scheme $G_m $ on $\Spec A$ with coaction $\mu^*$, then the $n$th graded component is given as $\{ f \in A| \mu^*(f) = fT^n \}$.

(2) If there is an action of the multiplicative monoid $ ({\mathbb A}^1, \cdot) $ on $\Spec A$, then this extends the action of the multiplicative group. This means that the cooperation lands in $A[T]$, and therefore, the part of negative degree is $0$, so $A$ is $\N$-graded. On the other hand, if $A$ is $\N$-graded, then the cooperation as before lands in $A[T]$.

(3). If $R=A_0$, then $A_+$ is an ideal in $A$ and $A \rightarrow A/A_+ \cong R$ is a cosection to $R \rightarrow A$. The algebraic version of the diagram 
\[ \begin{matrix}  A &  \stackrel{ }{\longleftarrow}  &  A[T] \\ \iota \uparrow & & \uparrow \mu^* \\ R & \stackrel{ \text{mod} A_+ }{ \longleftarrow  }&  A \,  \end{matrix} \]
commutes, as the upper horizontal map sends $T \mapsto 0$. If such a diagram exists with a ring homomorphism $\theta:A \rightarrow R$, then the homomorphism from $A$ to $A$ via $A[T]$ must be modding out $A_+$ and considering $A_0$ in $A$. If we go via $R$, we see that $A_0 \subseteq R$.
\end{proof}

Note that in (1) and (2), it is allowed that $A$ is concentrated in degree $0$.

\begin{Remark}
	\label{moduleschemegraded}
	Lemma \ref{monoidactiongraded} can be applied in the case of a module scheme or a module scheme up to normalization as well (see Section \ref{modulenormalizationsection}), the existence of a $0$-section is part of the axioms of a group scheme (up to normalization). Therefore, in these cases, we get a multiplicative action of ${\mathbb A}^1$ on $Z \rightarrow X$.
\end{Remark}

\begin{Lemma}
\label{moduleschemestandard}
Let $A$ be an $\N$-graded $R$-algebra with $A_0=R$ and endowed with an $R$-algebra homomorphism $\alpha^*: A \rightarrow A \tensor_R A$. Suppose that $R$ contains a field of characteristic $0$. Then the following are equivalent.

\begin{enumerate}

\item
$A$ is standard-graded and the homogeneous elements of degree $1$ are sent to $\alpha^*(a) = a \tensor 1 +1 \tensor a$.

\item
The diagram 
\[   \begin{matrix}  A [S,T] & \stackrel{   \mu^* \tensor \mu^*     }{\longleftarrow} & A \tensor_R A             \\      \alpha^*_{{\mathbb A}^1}         \uparrow &   & \uparrow  \alpha^* \\ A[T]
 & \stackrel{\mu^*}{\longleftarrow }& A \, \end{matrix} \] commutes.

\end{enumerate}
\end{Lemma}
\begin{proof}
We have a look at the diagram, which expresses the distributivity with respect to scalars. A homogeneous element $a_n \in A_n$ is sent to $a_n (S+T)^n$ in the upper left corner via the lower left path. The upper map sends $a_i \tensor b_j $ to $ a_i b_j S^iT^j$. 

The formula for $\alpha^*(a)$ means that the diagram commutes for elements of degree $1$. In the standard-graded case, the algebra is generated by elements of degree $1$ and then the diagram commutes in general (this holds without the assumption on the characteristic).

Suppose now that the diagram commutes. Then, $\alpha^*(a_n)$ must be homogeneous of degree $n$, because the other homomorphisms are homogeneous and injective. Hence, $\alpha(a_n) = \sum_{i = 0}^n b_i$ where $b_i= \sum_{m \in M_i} c_{i,m} \tensor d_{n-i, m}$ with homogeneous elements $c_{i,m}, d_{n-i,m}$ of the first index as degree. This is sent to $\sum_{i=0}^n ( \sum_{m \in M_i} c_{i,m} d_{n-i,m} ) S^iT^j$ and we must have 
\[\sum_{m \in M_i} c_{i,m} d_{n-i,m}  = \binom{n}{i}a_n \, .\]
In particular, for $n=1$, this means that $a \in A_1$ is sent via $\alpha^*$ to an expression of the form $1 \tensor b +c\tensor 1 $ with $b,c \in A_1$, and we deduce $b=c=a$. Suppose now $n \geq 2$. Applying the displayed formula for $i=1$ shows that $a_n$ is contained in the subalgebra generated by elements of smaller degrees. Hence, standard-graded follows by induction.
\end{proof}

\begin{Remark}
\label{milnormoorehopf}
Positively graded Hopf-algebras $A$ with $A_0=K$ a field (called connected) appear in the work of Milnor and Moore. If $K$ has characteristic $0$, they prove results (the theorem of Milnor-Moore) that imply (under commutative and  cocommutative assumption) that $A$ is a polynomial algebra over $K$, generated by homogeneous elements, see \cite[Corollary 4.18]{milnormoorehopf} and \cite[Theorem 3.8.3]{cartierprimerhopf}. Translated into Grothendiecks terminology, this means via Lemma \ref{monoidactiongraded} and Lemma \ref{moduleschemestandard} that a module scheme over the spectrum of a field of characteristic zero is a vector space (the spectrum of the symmetric algebra of a vector space).
\end{Remark}

\begin{example}
Let $K$ be a field of positive characteristic $p>0$ and let $ A=K[W]/(W^p)$. Then 
\[A \longrightarrow A \tensor_K A \cong K[W,Z]/(W^p,Z^p), \,   W \mapsto W +Z \, ,\]
is a ring homomorphism, as $(W+Z)^p=W^p+Z^p=0$, and this gives a cocommutative coaddition. The zero section is given by $W \mapsto 0$, and the conegation is given by $W \mapsto -W$. A co-scalar multiplication is given by
\[ A \longrightarrow A \tensor_K K[T] \cong A[T],\,  W \longmapsto WT \, .\]
\end{example}

\begin{Theorem}
\label{moduleschememilnormoore}
Let $R$ be a commutative ring containing of field of characteristic $0$ and let $\Spec A \rightarrow \Spec R$ be a module scheme of finite type. Then $A$ is the symmetric algebra of a finitely generated $R$-module (and the module scheme structure is the one from Lemma \ref{symmetricmodulescheme}).	
\end{Theorem}

\begin{proof}
By Lemma \ref{monoidactiongraded} and Lemma \ref{moduleschemestandard}, $A$ is a standard-graded $R$-algebra, and the scalar comultiplication and coaddition are given as described there. Let $M=A_1$, which is finitely generated. By the universal property of the symmetric algebra, there is a surjective graded ring homomorphism $ \Sym M \rightarrow A $, which respects the module scheme structures. We have to show that this homomorphism is injective. We may assume that $R\neq 0$ and $M \neq 0$. Assume that $ U \subseteq \Sym^n M$, $U \neq 0$, is the kernel in degree $n$. This situation holds also after localizing at some prime ideal $\idealp$. Going modulo the maximal ideal ${\idealp}R_\idealp$, we get by the Lemma of Nakayama that $U \tensor \kappa(\idealp) \neq 0$, but this is sent to $0$ under
\[ \Sym (M \tensor_R \kappa(\idealp) ) = ( \Sym M) \tensor_R \kappa (\idealp)  \longrightarrow A \tensor_R  \kappa(\idealp) \, .\]
This contradicts the theorem of Milnor-Moore for the field case, see \cite[Theorem 3.8.3]{cartierprimerhopf}.	
\end{proof}

\section{Module schemes up to normalization}
\label{modulenormalizationsection}

We cannot expect that a quotient scheme $Z_\rho \rightarrow X$ is a module scheme, as the results of Section \ref{productssection} suggest that there is no natural morphism
\[Z_\rho \times_X Z_\rho \longrightarrow Z_\rho \]
which could take the role of the addition (see Example \ref{quotientschemeadditionfails} for a concrete example). However, there is such a morphism when we replace the left-hand side with its normalization, which is isomorphic to $ Z_{\rho \times \rho}$ by Theorem \ref{normalproductcompatibility} (for $Y$ normal).

\begin{Definition}
Let $ Z \rightarrow X$ be an affine scheme over a separated scheme $X$. We say that $Z$ is a \emph{module scheme up to normalization} if the following holds.
\begin{enumerate}
\item 
There exists an open and dense subset $U \subseteq X$ such that $Z|_U \rightarrow U$ is a module scheme.
		
\item
There exists a zero section $X \rightarrow Z $, a negation $-:Z \rightarrow Z$ over $X$, a scalar multiplication $ \cdot : {\mathbb A }^1_X \times_X Z \rightarrow Z $ extending the given structures on $U$.
		
\item
There exists a morphism
\[ \alpha : \widetilde{ (Z \times_X Z) } \longrightarrow Z \, ,\]
extending $\alpha$ on $U$, where $ \widetilde{ (Z \times_X Z) } \rightarrow Z \times_X Z$ is a finite birational morphism.
		
\end{enumerate} 
\end{Definition}

If $Z$ is a module scheme, then it is also a module scheme up to normalization ($U=X$). In the affine situation, if $\Spec A \rightarrow \Spec R$ is a module scheme up to normalization, then we see by Lemma \ref{monoidactiongraded} and Remark \ref{moduleschemegraded} that $A$ is an $\N$-graded $R$-algebra and that the conegation and the zero section are determined as in the case of a module scheme.

The diagram from Section \ref{moduleschemesection} translates to 
\[  \begin{matrix} \widetilde{ Z \times_X Z \times_X Z } & \stackrel{ \widetilde{ \alpha \times   \operatorname{Id}_{ Z } }  }{ \longrightarrow } & \widetilde{Z \times_X Z }  \\ \widetilde{ \operatorname{Id}_{ Z } \times \alpha } \downarrow & & \downarrow \alpha  \\ \widetilde{Z \times_X Z} & \stackrel{\alpha}{\longrightarrow} &  Z \, , \end{matrix}\]  
where
$\widetilde{ \alpha \times \operatorname{Id}_Z }$ is an extension of $ \alpha \times \operatorname{Id} : \widetilde{ (Z \times_X Z) } \times_XZ \longrightarrow Z \times_X Z $ to suitable finite birational extensions.

\begin{Remark}
The data that constitute a module scheme up to normalization are determined by the data on the restriction to $U$, by \cite[Proposition 7.2.2]{ega1}. The only relevant question is whether these data can be extended to some $\widetilde{ Z \times_XZ}$. Since the properties are expressed by the commutativity of certain diagrams, i.e., identities between morphisms, these identities only need to be checked on a dense open subset.
\end{Remark}

\begin{Example}
A module scheme $Z \rightarrow U $ over an open subset $U \subset X$, considered as a scheme over $X$, is not a module scheme up to normalization since the zero section does not extend.	
\end{Example}

\begin{Example}
We consider the module scheme $Y=X \times {\mathbb A}^1 \stackrel{\pi}{\rightarrow} X$, the trivial vector bundle of rank one over $X$. If we remove a point $Q \neq 0$ from $Y$ over $P\in X$ and set $Z=  Y \setminus \{ Q \}$, then $Z|_{X \setminus \{P\} } \rightarrow {X \setminus \{P\} } $ is a module scheme and the zero section extends to $Z \rightarrow X$, but the ${\mathbb A}^1$-scalar multiplication does not extend.
\end{Example}

\begin{Lemma}
\label{moduleschemeuptoconstruction}
\begin{enumerate}
\item
The product $ Z_1 \times_X Z_2 $ of two module schemes up to normalization over a separated scheme $X$ is a module scheme up to normalization.
	
\item
The pull-back of a module scheme up to normalization $Z \rightarrow X $ under a dominant morphism $X' \rightarrow X$ is a module scheme up to normalization.
	
\item 
The restriction of a module scheme up to normalization $Z \rightarrow X $ to a closed irreducible subset $Y \subseteq X$ meeting the locus $U$ where we have the structure of a module scheme is a module scheme up to normalization.

\item
Let $X=\Spec R$ with a normal domain $R$, and suppose that $Z \rightarrow X$ is a module scheme up to normalization, which is generically a vector bundle. Then, the normalization $\tilde{Z} \rightarrow X$ is a module scheme up to normalization.
\end{enumerate}
\end{Lemma}
\begin{proof}
(1) The morphisms $ \alpha_1: \widetilde{Z_1 \times_X Z_1} \rightarrow Z_1$ and $ \alpha_2: \widetilde{Z_2 \times_X Z_2} \rightarrow Z_2$ yield a morphism
\[   ( \widetilde{Z_1 \times_X Z_1} ) \times_X ( \widetilde{Z_2 \times_X Z_2} ) \stackrel{\alpha_1 \times \alpha_2}{ \longrightarrow} Z_1 \times_X Z_2 \, . \]
The left-hand side is finite and birational over \[(Z_1  \times_X Z_1) \times_X (Z_2 \times_X Z_2) =(Z_1  \times_X Z_2) \times_X (Z_1 \times_X Z_2) \, ,\]
which gives the addition up to normalization.
	
(2) We set $Z'=X' \times_X Z$, and we have $Z' \times_{X'} Z'= X' \times_X Z \times_X Z $. We look at the commutative diagram
\[  \begin{matrix}  \widetilde{Z \times_X Z} & \stackrel{\alpha}{\longrightarrow} & Z \\   \theta \downarrow  & & \downarrow \\   Z \times_X Z &  \longrightarrow & X  \end{matrix} \]
and pull it back via $X' \longrightarrow X$. The resulting morphism
\[ \theta': X' \times_X \widetilde{ Z \times_X Z} \longrightarrow X' \times_X Z \times_X Z =Z' \times_{X'} Z' \]
is again finite and, due to the dominance assumption, also birational.

(3) We have $(Z \times_X Z)|_Y = Z|_Y \times_Y Z|_Y$, and the finite birational map $\widetilde{Z \times_X Z} \rightarrow Z \times_X Z$ yields a finite map after the restriction to $Y$. The restriction to $Y \cap U$ shows that this map is also birational. The addition $\alpha: \widetilde{Z \times_X Z} \rightarrow Z$ restricts to an addition $\alpha: \widetilde{Z \times_X Z} |_Y \rightarrow Z|_Y$. The other properties are clear.

(4) Let $R \rightarrow A$ be the algebra. $A$ is an $\N$-graded $R$-algebra with $A_0=R$ by Lemma \ref{monoidactiongraded}. Let $f \in R$, $f \neq 0$, be such that $A_f$ is a polynomial ring over $R_f$. In particular, $A_f$ is normal, $\N$-graded and the normalization of $A$ is in between $ A \subseteq  \tilde{A} \subseteq A_f$. It follows that the normalization is also $\N$-graded, which gives the scalar action of ${\mathbb A}^1$ on $\Spec \tilde{A}$ and the zero section. The coaddition $\alpha^*: A \rightarrow \widetilde{A \tensor A} $ sits in the commutative diagram
\[ \begin{matrix}  A  & & A \tensor_RA \\ \downarrow &  \searrow \alpha^* & \downarrow \\ \tilde{A} && \widetilde{A \tensor_R A} \\ \downarrow & & \downarrow \\ A_f & \longrightarrow & A_f \tensor_{R_f} A_f\, .  \end{matrix} \]
The image of an element $a \in \tilde{A} $ in $A_f \tensor_{R_f} A_f$ fulfils an integral equation over $ \alpha^*(A)$, so it belongs to $ \widetilde{A \tensor_R A} $, which coincides with the normalization of $\tilde{A} \tensor_R \tilde{A}$. The other properties of a module scheme up to normalization are clear.
\end{proof}

\begin{Remark}
As the beginning of the proof of Lemma \ref{moduleschemeuptoconstruction} shows, for every pull-back of a module scheme up to normalization along $X' \rightarrow X$, we get an addition after a finite modification $\widetilde{Z' \times_{X'} Z' } \rightarrow Z' $, which is in general not birational. In particular, for every point $P \in X$ we get an addition in the fiber $Z_P$ up to modification, a finite morphism $ \widetilde{Z_P \times Z_P} \rightarrow Z_P$.
\end{Remark}

\section{Linear actions on module schemes}
\label{linearsection}

We describe the morphisms between module schemes up to normalization which respect the module structure, and the linear actions of a group on them.

\begin{Definition}
Let $ Z_1,Z_2 \rightarrow X$ be module schemes up to normalization. A morphism $\varphi:Z_1 \rightarrow Z_2$ over $X$ is called \emph{linear} if the diagram	
\[  \begin{matrix} {\mathbb A}^1 \times Z_1 & \stackrel{ \operatorname{Id}  \times \varphi }{ \longrightarrow} &  {\mathbb A}^1 \times Z_2  \\ \cdot \downarrow & &  \downarrow \cdot  \\  Z_1 & \stackrel{\varphi}{\longrightarrow} & Z_2      \end{matrix}  \] 
commutes and if  after suitable finite birational extensions, the diagram
\[  \begin{matrix} \widetilde{Z_1 \times Z_1} & \stackrel{\widetilde{\varphi \times \varphi}}{ \longrightarrow} & \widetilde{Z_2 \times Z_2} \\ \alpha_1 \downarrow & & \downarrow  \alpha_2 \\  Z_1 & \stackrel{\varphi}{\longrightarrow} & Z_2      \end{matrix}  \] 
commutes.
\end{Definition}

For a module scheme, the same definition applies. Typical examples are linear morphisms between vector bundles over $X$.

\begin{example}
	\label{matrixmoduleschemelinear}
For an $m \times n$-matrix $B$ over $R$, we get a linear morphism ${\mathbb A}^n_{\Spec R} \rightarrow {\mathbb A}^m_{\Spec R} $ between free module schemes (notice the change in the notation compared to Example \ref{moduleschemepresentation}, which is due to the contravariance).
\end{example}

\begin{Remark}
Homomorphisms of module schemes have kernels, but, like group schemes, in general no cokernels, which is clear by looking at ${\mathbb A}^1_R \rightarrow {\mathbb A}^1_R$ given by $T \mapsto fT$ for a non-unit, non zero divisor $f \in R$. But also the kernel as a module and as a module scheme may differ. For $R=K[X,Y]$ and the matrix $(x,y)$ as in Example \ref{matrixmoduleschemelinear}, we get the linear map 
\[ {\mathbb A}^2 \times {\mathbb A}^2 \longrightarrow {\mathbb A}^2 \times {\mathbb A}^1, (x,y,s,t) \longmapsto (x,y,xs+yt)  \, \]
of module schemes above $ {\mathbb A}^2 $. The kernel is a module scheme, where the fiber dimension is one on the punctured plane but two above the origin (which is also the answer to the linear algebra question: determine the kernel of the linear map $(s,t) \mapsto xs+yt$ in dependence of the parameters $(x,y)$). The module kernel is just $R$, which equals the module of global sections in the module scheme kernel.
\end{Remark}

\begin{example}
For an $R$-module homomorphism $M \rightarrow N$ between $R$-modules $M$ and $N$, we obtain a natural $R$-algebra homomorphism $\Sym (M) \rightarrow \Sym(N) $ and hence a morphism $\Spec \Sym (N) \rightarrow \Spec \Sym(M) $, which is linear as a morphism between module schemes. If we want a morphism in the same direction, we first have to go to the dual homomorphism  $N^\dual \rightarrow M^\dual $.
\end{example}

\begin{Lemma}
\label{moduleschemeuptolinear}
Let $Z_1,Z_2,Z_3,Z_4 \rightarrow X$ be module schemes up to normalization over $X$, and let $\varphi: Z_1 \rightarrow Z_2$ and $\psi: Z_3 \rightarrow Z_4$ be linear morphisms. Then, the following hold.
		
\begin{enumerate}
\item 
The composition of linear morphisms is linear.
	
\item
The morphism $\varphi \times \psi: Z_1 \times_X Z_3 \rightarrow Z_2 \times_X Z_4 $ is linear.
			
\item
The pull-back $ X' \times_X Z_1 \rightarrow X' \times_X Z_2 $ is linear  under a dominant morphism $ X' \rightarrow X $.
		
\item 
If the schemes $Z_1$ and $Z_2$ are of finite type, then $ \varphi $ is determined by the restriction $ \varphi|_U: (Z_1)|_U \rightarrow  (Z_2)|_U $ for any open dense subset $ U \subseteq X $. One can choose $U $ such that $ \varphi|_U $ is a morphism of module schemes.		
\end{enumerate}
\end{Lemma} 

\begin{proof}
(1) is clear. (2) and (3) follow from Lemma \ref{moduleschemeuptoconstruction}. (4) is an application of \cite[Proposition 7.2.2]{ega1}. For $U$ sufficiently small, we get module schemes.
\end{proof}

\begin{Remark}
A symmetric algebra $\Spec \Sym M \rightarrow \Spec R$ for a finitely generated $R$-module $M$ comes naturally together with a linear injective mapping $\Spec \Sym M \rightarrow \Spec R \times {\mathbb A}^\nu $ above $\Spec R$, defined by a surjection $R^\nu \rightarrow M$, where $\nu$ is the number of generators of the module. If $M$ is locally free above $U$, then we get, by \cite[Th\'eorem 9.7.4]{ega1}, on $U$ a mapping $\vartheta: U \rightarrow \operatorname{Gr} (\nu+r,r)$ into the Grassmann variety, where $r$ denotes the rank of $M$, and where the pull-back of the universal bundle on $\operatorname{Gr} (\nu+r,r)$ gives back the embedding of the bundle over $U$. The higherdimensional fibers of  $\Spec \Sym M \rightarrow \Spec R$ outside of $U$ are reflected by the property that the closure of the image of $\vartheta$ degenerates to a higher dimension. For a module scheme $\Spec A \rightarrow \Spec R$ up to normalization, realizing the same module $M$ above $U$, the embedding $\Spec A \subseteq \Spec R \times {\mathbb A}^\mu$ (where $\mu$ denotes the number of $R$-algebra generators of $A$) is not linear in general (not even above $U$), as Example \ref{toric11.1} shows. Hence, there is also no interpretation of $\Spec A$ using the Grassmannian.
\end{Remark}

\begin{Definition}
Let $G$ be a finite group acting on a scheme $Y$ by automorphisms and let $Z \rightarrow Y$ be a module scheme up to normalization over $Y$. An action of $G$ on $Z$ that is compatible with the base action is called \emph{linear} if, for each $g \in G$, the morphism $\varphi_g:Z \rightarrow Z$ is linear.
\end{Definition}

\begin{Remark}
If $Z=Y \times {\mathbb A}^m$ and $G$ is acting in the second component linearly, then this is a linear action on a module scheme. Here no normalization process is needed, and this is usually the situation in which we start.

If $Z \rightarrow X$ is a module scheme up to normalization, and $G$ acts trivially on $X$ and on $Z$, then this is also a linear action. This is usually the situation in which we will end up.
\end{Remark}

\section{Quotient module schemes}
\label{quotientsection}

We discuss for a module scheme up to normalization $Z \rightarrow Y$ with a linear action of a finite group $G$ on $Z$ how the quotient $Z/G$ inherits this structure.

\begin{Theorem}
\label{Moduleschemegeneralquotient}
Let $G$ be a finite group acting on a normal scheme $Y$ over $\Spec K$ with quotient scheme $X$, and let $Z \rightarrow Y$ be a module scheme up to normalization together with a  compatible linear action $\rho$ of $G$ on $Z$. Then, $Z/G \rightarrow X$ is a module scheme up to normalization on $X$.
\end{Theorem}
\begin{proof}
The existence of the morphism $Z/G \rightarrow Y/G=X$ follows from the compatibility. The zero section $ Y \stackrel{  0 }{\rightarrow} Z $ is $G$-equivariant and yields the zero section $X \rightarrow Z/G $ by going to the quotients. The scalar multiplication morphism
\[ \cdot : {\mathbb A}^1 \times Z   \longrightarrow Z \, \]
over $Y$ is also $G$-equivariant  by the linearity of the action and yields by going to the quotients a morphism  
\[  {\mathbb A}^1 \times Z/G \longrightarrow  Z/G  \]
over $X$. Also, the negation $ Z \rightarrow Z  $ is $G$-equivariant and yields the negation $Z/G \rightarrow Z/G $ over $X$. Note that the diagonal action of $G$ on $Z \times_Y Z$ induces an action on the normalization. The addition morphism
\[ \alpha  : \widetilde{ Z \times_Y Z }  \longrightarrow  Z \,   \]
over $Y$ induces by going to the quotients under the action of $G$ a morphism
\[ \widetilde{ Z \times_Y Z }/G \longrightarrow  Z/G \]
over $X$. Since the left-hand side is the normalization of  $Z/G \times_X Z/G $ by a suitable generalization of Theorem \ref{normalproductcompatibility}, we are in the situation of a module scheme up to normalization. The group scheme properties follow from  the corresponding diagrams of the addition by going to the quotients. For example, the associativity law, i.e., the commutative diagram
\[ \begin{matrix}  \widetilde{ Z \times _{Y} Z \times_Y Z   }    & \stackrel{ \widetilde{  \operatorname{Id} \times \alpha_{23}  }  }{ \longrightarrow} & \widetilde{  Z \times_Y Z } \\ \widetilde{ \alpha_{12} \times \operatorname{Id} } \downarrow & & \downarrow \alpha  \\\widetilde{  Z \times_Y Z }  & \longrightarrow & Z   \end{matrix}   \]
gives by going to the quotients the commutative diagram
\[ \begin{matrix}  \widetilde{ Z \times _{Y} Z \times_Y Z   }  /G  & \stackrel{ \widetilde{  \operatorname{Id} \times \alpha_{23}  }  }{ \longrightarrow} & \widetilde{  Z \times_Y Z }/G \\ \widetilde{ \alpha_{12} \times \operatorname{Id} } \downarrow & & \downarrow \alpha  \\\widetilde{  Z \times_Y Z }/G  & \longrightarrow & Z/G \, ,  \end{matrix}   \]
which shows the associativity above
\[ \begin{matrix} Z/G \times_{X} Z/G \times_X Z/G & \cdots \rightarrow  &  Z/G  \times_{X } Z/G \\   \vdots \downarrow & & \vdots \downarrow \\  Z/G  \times_X
 Z/G  & \cdots \rightarrow & Z/G \, .  \end{matrix}   \] 
\end{proof}

\begin{Corollary}
\label{Moduleschemequotient}
Let the finite group $G$ act faithfully on a normal affine scheme $Y$ via $K$-algebra automorphisms with quotient scheme $X$, and let $G$ act linearly via $\rho$ on  ${\mathbb A}^m$. Then the quotient scheme $Z_\rho= ( Y \times {\mathbb A}^m  )/\beta \times \rho \rightarrow X$ is a module scheme up to normalization.
\end{Corollary}
\begin{proof}
This follows from Theorem \ref{Moduleschemegeneralquotient}.	
\end{proof}

We give an example of a quotient scheme in order to show concretely why a true addition does not exist.

\begin{Example}
\label{quotientschemeadditionfails}
We consider Example \ref{cycliconespecial} and Example
\ref{cycliconeproduct} for $k=2$, the invariant ring is $K[A]$, the invariant algebra is $K[A,B,C]/(AB-C^2)$, the tensor product of the invariant algebra with itself over $K[A]$ is $K[A,B,C \tilde{B}, \tilde{C}]/(AB-C^2, A \tilde{B}- \tilde{C}^2)$.
The coaddition (of the affine line) $K[X,V] \rightarrow K[X,V,W]$, $ X \mapsto X$, $V \mapsto V+W$, induces $K[A,B,C]/(AB-C^2) \rightarrow K[X,V,W]^G$, which is the coaddition on the module scheme up to normalization, according to	Corollary \ref{Moduleschemequotient}. This is not a module scheme, the pair $(0,b,0;0, \tilde{b},0) \in Z_\rho \times_X Z_\rho$ can not be added. Above this point, we have in the normalization the two points $ (0, b,0, \tilde{b},0, \pm \sqrt{b \tilde{b} } ) $. These points are sent by the up-to-addition to $(0,b+ \tilde{b} \pm \sqrt{ b \tilde{b} }, 0)$. The problem is that the pair can not decide where to go, there are two equally good candidates for their addition. One needs the extra information of the normalization to remove this indeterminacy. 
\end{Example}

We have a second look at the morphism from Lemma \ref{modulealgebra} with the new terminology.

\begin{Lemma}
\label{modulealgebrascheme}
Let	$G$ be a finite group acting faithfully via $\beta$ on a $K$-algebra $S$ by $K$-algebra automorphisms, and let  $R=S^G$ be its invariant ring. Let $G$ act on $K^m$ linearly via $\rho$, and let $Z_\rho \rightarrow \Spec R$ be the quotient scheme of the corresponding product action of $G$ on $Y \times {\mathbb A}^m$, and let $M $ be the corresponding invariant $R$-module.	Then, there is a canonical linear morphism
\[ Z_\rho \longrightarrow \Spec ( \Sym_R (M) ) \,  \]
of module schemes up to normalization over $\Spec R$.
\end{Lemma}
\begin{proof}
The canonical morphism exists by Lemma \ref{modulealgebra}, the module scheme structure on $ \Spec ( \Sym_R (M) ) $ was noted in Lemma \ref{symmetricmodulescheme}, and the module scheme structure up to normalization on $Z_\rho$ was shown in Theorem \ref{Moduleschemegeneralquotient}. Set $B=S[W_1, \ldots, W_m]$. The linearity follows from the commutative diagram
\[ \begin{matrix} \Sym M & \longrightarrow & B^\rho  & & \\ \downarrow & & \searrow & \\ \Sym M \tensor_R \Sym M    & \longrightarrow B^\rho \tensor_R B^\rho  & \longrightarrow & \widetilde{ B^\rho \tensor_R B^\rho} =(B \tensor_R B)^{\rho \times \rho}    \end{matrix} \]
and since $	\Sym M  \rightarrow  B^\rho $ is homogeneous.
\end{proof}

\section{Normal subgroups}
\label{normalsubgroupsection}

If $G$ acts on $Y$ with quotient scheme $X$ and $H \subseteq G$ is a normal subgroup, then one can often understand the situation better by considering the quotient $Y'=Y/H$, on which the residue class group $G/H$ acts, and which has the quotient $Y'/(G/H) =Y/G$. Typical examples arise when $\rho$ is not a faithful representation on $K^m$, and we use $H= \ker \rho$, or when $G$ acts linearly on $K^d$ and $H=G \cap \SLG_d(K)$, where $G/H$ is a cyclic group, or when $H\subseteq G$ is the subgroup generated by the reflections inside $G$. We want to understand how module schemes up to normalization behave in the presence of a normal subgroup.

\begin{Lemma}
\label{Moduleschemelinearquotient}
Let the finite group $G$ act faithfully on a normal affine scheme $Y$ via $K$-algebra automorphisms with quotient scheme $X$. Let $Z_1,Z_2 \rightarrow Y$ be module schemes up to normalization together with compatible linear  actions of $G$ on them, and let $\varphi:Z_1 \rightarrow Z_2$ be a linear $G$-equivariant morphism over a $G$-equivariant morphism $\psi:Y \rightarrow Y$. Then, this induces a linear morphism  $\overline{\varphi}: Z_1/G \rightarrow Z_2/G$ over $ \overline{\psi}:X \rightarrow X$.
\end{Lemma}
\begin{proof}
The quotient schemes $Z_1/G$ and $Z_2/G$ are module schemes up to normalization by Theorem \ref{Moduleschemegeneralquotient}. The linear morphism $ \varphi:Z_1 \rightarrow Z_2$ over $\psi:Y \rightarrow Y$ induces the morphism  $\overline{\varphi}:Z_1/G \rightarrow Z_2/G$ over $ \overline{\psi} :Y/G \rightarrow Y/G$. The compatibility of $\varphi$ with the scalar multiplication on $Z_1$ and on $Z_2$  induces the compatibility of $\overline{\varphi}$ with the scalar multiplication on $Z_1/G$ and on $Z_2/G$. The commutativity of the addition diagram passes also over to the quotients.
\end{proof}

\begin{Theorem}
\label{quotientschemeinterpolation}
Let $G$ be a finite group with a normal subgroup $H \subseteq G$ and residue class group $G/H$. Let $G$ act faithfully on a normal affine $K$-scheme $Y$ by $K$-automorphisms, and let $Z \rightarrow Y$ be a module scheme up to normalization together with a compatible linear action of $G$. Then, $Z/H \rightarrow Y/H$ is a module scheme up to normalization above $Y/H$, on which the induced action of $G/H$ is linear and compatible with the action of $G/H$ on $Y/H$. There is a natural linear isomorphism $Z/G \rightarrow (Y/H)/(G/H)$.
\end{Theorem}
\begin{proof}
By invariant theory, we have the commutative diagram
\[ \begin{matrix} Z & \longrightarrow & Z/H & \longrightarrow & Z/G \\
	\downarrow && \downarrow && \downarrow \\
	Y & \longrightarrow & Y /H & \longrightarrow & Y/G \, ,\end{matrix} \]
where $G/H$ acts on the schemes in the middle and their quotients coincide with the quotients on the right. Theorem \ref{Moduleschemegeneralquotient} shows that $Z/H \rightarrow Y/H$ and $Z/G \rightarrow Y/G$ are module schemes up to normalization. It remains to show that the action of $G/H$ on $Z/H$ is linear.	
Let $[g] \in G/H$. The result follows by applying 
Lemma \ref{Moduleschemelinearquotient} to the linear morphism $g:Z \rightarrow Z$ 
and the group $H$.
\end{proof}

\begin{Corollary}
\label{inbetween}
Let $G$ be a finite group with a normal subgroup $H \subseteq G$ and residue class group $G/H$. Let $G$ act faithfully on a normal affine $K$-scheme $Y$ by $K$-automorphisms, and let $\rho$ be a linear representation of $G$. Then we have a commutative diagram
\[ 	\begin{matrix} Y \times {\mathbb A}^m  &\longrightarrow & (Y \times {\mathbb A}^m)/H  & \longrightarrow &  (Y \times {\mathbb A}^m)/G       \\ \downarrow  &  & \downarrow  &   & \downarrow \\ Y &\longrightarrow &Y'& \longrightarrow & X \,  \end{matrix} \]
of module schemes up to normalization, 	where $Y'=Y/H$, $X=Y/G=Y'/(G/H)$ and $   (Y \times {\mathbb A}^m)/G =  ( (Y \times {\mathbb A}^m)/H )  /(G/H) $.
\end{Corollary}
\begin{proof}
This is a special case of Theorem \ref{quotientschemeinterpolation}.	
\end{proof}

\section{Similar modifications}
\label{similarsection}

We describe further situations where there is no module scheme structure in the strict sense, because of certain ``bad'' properties of the tensor product, but there is a module scheme structure after (some kind of) a modification of the tensor product.

\begin{Remark}
\label{Reesuptotorsion}
Let $M$ be an $R$-module and consider the Rees algebra 
\[ {\mathcal R} (M) =  \Sym {M} /R- \text{torsion}  \, . \]
We look at the  diagram		
\[ \begin{matrix}  \Sym {M}  & \stackrel{\alpha^*}{ \longrightarrow} & \Sym {M}  \tensor_R \Sym {M} \\ \downarrow   & & \downarrow \\ {\mathcal R} ( {M})  & &   {\mathcal R} ( M)  \tensor_R  {\mathcal R} ( M) \\
& \searrow & \downarrow \\ & & (  {\mathcal R} ( {M})  \tensor_R  {\mathcal R} ( {M}))/R- \text{torsion} \, ,  \end{matrix}   \]
where $\alpha^*$ denotes the coaddition on the symmetric algebra. This ring homomorphism does not induce a corresponding ring homomorphism for the Rees algebras, but when we mod out the $R$-torsion in the tensor product of the Rees algebras, we obtain a homomorphism. This shows that the Rees algebra is a ``module scheme up to torsion''.
\end{Remark}

\begin{example}
Let $R=K[X,Y]/(X^2-Y^3)$, $ {\mathfrak m}=(X,Y)$ and consider the symmetric algebra $\Sym {\mathfrak m}$ and the Rees algebra ${\mathcal R} ( {\mathfrak m}) = R \oplus {\mathfrak m} \oplus {\mathfrak m}^2 \ldots   $. The symmetric algebra has the description
\[ \Sym ( {\mathfrak m}) = R[A,B] /( yA-xB,xA-y^2B)  \]
and the Rees algebra is
\[ {\mathcal R} ( {\mathfrak m})  = R[A,B] /( yA-xB,xA-y^2B, A^2-yB^2) \, . \]
The $R$-torsion of $\Sym ( {\mathfrak m}) $ can be read from
\[  y(A^2-yB^2) = A(yA-xB)  + B(xA-y^2B) = 0  \, .\]
In the diagram of Remark \ref{Reesuptotorsion}, the element $A^2-yB^2 \in \Sym {\mathfrak m} $ is sent to $(A+ \tilde{A})^2 - y(B+ \tilde{B})^2$ and this is then sent to $2 A \tilde{A} - 2 y B \tilde{B} \neq 0$ on the right, hence, there is no homomorphism from $ {\mathcal R} ( {M})$ to $   {\mathcal R} ( M)  \tensor_R  {\mathcal R} ( M)  $. This element is an $R$-torsion element in ${\mathcal R} ( {\mathfrak m}) \tensor_R  {\mathcal R} ( {\mathfrak m})$.
\end{example}

\begin{Remark}
Let $M$ be an $R$-module and let $U \subseteq X = \Spec R$ be open. We consider the $\N$-graded $R$-algebra $A$, where the $n$th component is $\Gamma (U,\Sym^n M )$. If $R$ is normal and local of dimension $\geq 2$ and if $U$ is the punctured spectrum, then this gives the reflexive symmetric algebra, where the $n$th component is the reflexive hull $ (\Sym^n M)^{\dual \dual}$ (see also Section \ref{reflexivesection}). The natural sheaf homomorphism
\[ \alpha:  \widetilde{\Sym M } \longrightarrow  \widetilde{\Sym M } \tensor_{ {\mathcal O}_X}  \widetilde{\Sym M }     \]
induces a natural $R$-algebra homomorphism as shown in the following diagram
\[ \begin{matrix} \bigoplus_{n \in \N}  \Gamma (U,\Sym^n M )  & \stackrel{\alpha}{ \longrightarrow }  & \Gamma (U,  \widetilde{\Sym M } \tensor_{ {\mathcal O}_X} \widetilde{\Sym M } )  & \\ & & \rotatebox[origin=c]{90}{ $\subseteq $} \\ & & ( \bigoplus_{n \in \N} \Gamma (U,\Sym^n M ) )  \tensor_R ( \bigoplus_{n \in \N}  \Gamma (U,\Sym^n M )  ) \, .   \end{matrix}  \]
Again, we cannot expect that $\alpha $ lands in the tensor product below, so we do not have a module scheme, but we have a module scheme after a natural modification (after going to the reflexive hull in the specific example).	
\end{Remark}

\begin{Remark}
Suppose that $X=\Spec R$, normal of finite type over a field $K$, has an isolated singularity in the point $P$, let $U=X \setminus \{P\} $ be the smooth punctured spectrum. The (geometric) tangent bundle ${ T}_U$ is a vector bundle on $U$. Let $ \varphi:\tilde{X} \rightarrow X$ denote a resolution of singularities. Then ${T}_U \cong { T}_{\varphi^{-1}(U)} $ extends on $\tilde{X}$ to the tangent bundle ${ T}_{\tilde{X} }$. We consider $p: T_{\tilde{X} } \rightarrow X$ as a scheme above $X$ extending $T|_U$. In contrast to the other situations considered, this is not an affine scheme over $X$, the exceptional fiber has dimension $2 \dim X -1 $ and contains projective subvarieties of dimension $\dim X -1$. On the resolution, we have the vector bundle addition $\alpha: T_{\tilde{X} } \times_{\tilde{X}} T_{\tilde{X}} \rightarrow T_{\tilde{X}}$, but we do not have an addition $ T_{\tilde{X} } \times_{ X} T_{\tilde{X}} \rightarrow T_{\tilde{X}} $. Here, we have an addition on $ T_{\tilde{X} } \rightarrow X $ after a birational projective modification of $X$.

The same construction can be made for every vector bundle $Z \rightarrow U$ which has a vector bundle extension after some blow-up.
\end{Remark}

\begin{Remark}
In the construction of abelian varieties by A. Weil, one first has only a rational composition $V \times V  \cdots \rightarrow V$, which must then be extended to a proper variety (see  \cite[Th\'{e}or\`{e}me 15]{weilabelian}, \cite[Expos\'{e} XVIII]{sga3}), and a similar construction occurs if one wants to spread an abelian variety defined over a quotient field $Q(R)$  to an abelian scheme over a $R$, see  \cite[Section 4.3]{boschneron}.
\end{Remark}

\section{Fiberflat bundles}

\label{fiberflatsection}

In the case of a finite group acting on $S$ and on a free module $N=S^m$, we have seen in Lemma \ref{modulealgebra} that there are two module schemes up to normalization over $X=\Spec S^G$, namely, $\Spec ( \Sym ( N^\rho)) $ and $ Z_\rho$, which define over a dense open subset  $U \subseteq X$ the same vector bundle, but are quite different when we look at the restriction above $X \setminus U$. A typical behavior of the spectrum of the symmetric algebra is that the dimension of the fibers outside the locally free locus increase, whereas they stay constant for $Z_\rho$.

\begin{Definition}
We call a module scheme up to normalization $p:Z \rightarrow X$ \emph{fiberflat} if the following conditions hold.
	
\begin{enumerate}
\item
$Z$ is of finite type over $X$.
		
\item
There exists an open dense subset $U \subseteq X$ such that $Z|_U \rightarrow U$ is a vector bundle.
		
\item 
All fibers of $p:Z \rightarrow X$ have the same dimension.
\end{enumerate}
\end{Definition}

Every vector bundle is fiberflat. The common dimension of the fibers is called the \emph{rank} of the fiberflat bundle. The term fiberflat points to the fact that for a flat family, the dimension of the nonempty fibers is constant, see \cite[Proposition 9.5]{hartshorne}.

\begin{Lemma}
\label{fiberflatquotient}
Let the finite group $G$ act faithfully on a normal affine scheme $Y$ via $K$-algebra automorphisms with quotient scheme $X$, and let $G$ act linearly via $\rho$ on  ${\mathbb A}^m$. Then, the quotient scheme $Z_\rho= ( Y \times {\mathbb A}^m  )/\beta \times \rho \rightarrow X$ is a fiberflat bundle.
\end{Lemma}
	
\begin{proof}
This follows from Corollary \ref{Moduleschemequotient}, Remark \ref{fiberdimension} and Corollary \ref{modulealgebralinearisomorphism}.	
\end{proof}

\begin{Lemma}
\label{fiberflatconstruction}
Let $X$ denote a separated scheme.

\begin{enumerate}
\item
The product $ Z_1 \times_X Z_2 $ of two fiberflat bundles is fiberflat.
		
\item
The pull-back of a fiberflat bundle under a dominant morphism $X' \rightarrow X$ is fiberflat.
		
\item 
The restriction of a fiberflat bundle $Z \rightarrow X $ to a closed irreducible subset $Y \subseteq X$ meeting the locus $U$ where we have the structure of a vector bundle is fiberflat.
		
\item
Suppose that $X$ is normal integal and excellent. Then the normalization $\tilde{Z} \rightarrow X$ of a fiberflat bundle is fiberflat.
\end{enumerate}
\end{Lemma}
\begin{proof}
This follows from Lemma \ref{moduleschemeuptoconstruction} since the dimensions of the fibers behave nicely for the products of the schemes and for normalization. Excellence is needed to ensure that $ \tilde{Z}$ is of finite type.
\end{proof}

\begin{Lemma}
	Let $R$ be a commutative ring and let $M$ be a finitely generated $R$-module. Then, the symmetric algebra $\Sym M$ is fiberflat if and only if $M$ is locally free.	
\end{Lemma}
\begin{proof}
	We may assume that $R$ is local. If $M$ is free, the symmetric algebra is a polynomial algebra, which is fiberflat. To prove the converse, let $r$ be the rank of $M$ and $\nu$ the minimal number of generators of $M$. The generic fiber of the spectrum of the symmetric algebra has dimension $r$, the special fiber over the maximal ideal has dimension $\nu$. Fiberflatness implies that $r= \nu$, which means that $M$ is free.	
\end{proof}

\begin{Lemma}
\label{fiberflatirreducible}
Let $R$ be a domain essentially of finite type over a field $K$ and let $p:Z \rightarrow \Spec R$ be a fiberflat bundle which is locally free over $U=\Spec R \setminus \{  \idealm \}$, where $\idealm$ denotes a maximal ideal. Suppose that $p^{-1}(\idealm)$ is irreducible. Then, $Z$ is irreducible.	
\end{Lemma}
\begin{proof}
The open subset $p^{-1}(U)$ is irreducible, let $C$ denote its closure in $Z$, which is also irreducible. Assume that there is another component $C'$. This component must be the fiber above $\idealm$ and has dimension $r$, the rank of $Z$. Therefore, the intersection $C \cap C'$ has dimension at most $r-1$. This intersection is not empty because of the zero section.
But then $C \rightarrow \Spec R$ would contradict \cite[Theorem 14.8]{eisenbud}.
\end{proof}

\begin{Remark}
Lemma \ref{fiberflatconstruction} (3) shows that a generic restriction of a fiberflat bundle to a closed irreducible subset is again fiberflat. In combination with Lemma \ref{fiberflatirreducible} this means that such a restriction is again irreducible. This is in stark contrast for module schemes up to normalization which are not fiberflat. There, the jump in the dimensions forces the  scheme to become reducible, when it is restricted to a closed subset of small dimension. E.g., the module scheme
\[ \Spec ( K[X,Y][S,T]/(SX+TY)) \longrightarrow {\mathbb A}^2\]
is integral, but if we restrict it to any line (or irreducible curve) through the origin, we get a scheme with two components. See also Remark \ref{regularfiberflatconjecture}.
\end{Remark}

\begin{Lemma}
	\label{fiberflatopen}
	Let $R$ be a domain of finite type over a field $K$ and let $p:Z \rightarrow \Spec R$ be an affine scheme of finite type  which is locally free over $U=\Spec R \setminus \{  \idealm \}$, where $\idealm$ denotes a maximal ideal of $\idealm$. Suppose that $p^{-1}(\idealm)$ is irreducible. Then, $Z$ is fiberflat if and only if the morphism is open.
\end{Lemma}
\begin{proof}
We use the going down criteria for openness, see \cite[Corollaire 1.10.4]{ega4.1}. If $ Z \rightarrow \Spec R$ is not fiberflat, then for an irreducible curve $C$ through $\idealm $ the restriction $Z|_C$ is not irreducible. A point above $\idealm$ which lives in the component above $\idealm$ but not in the dominant component does not live on a curve dominating $C$.

Suppose now that $ Z \rightarrow \Spec R$ is fiberflat, and let $\idealp \subseteq \idealq$ be a chain of prime ideals in $R$ and $\idealq'$ a prime ideal in $Z$ mapping to $\idealq$. We look at the situation above $V( \idealp )$, the fiberflattness is preserved. By Lemma \ref{fiberflatirreducible}, $Z|_{V(\idealp) }$ is irreducible, $\idealq'$ lies on it and its generic point is mapped to $\idealp$.	
\end{proof}

\section{Reflexive bundles and factorial closure}

\label{reflexivesection}

For a quotient scheme $Z \rightarrow X$ and an open subset $U \subseteq X$, we cannot expect that $Z|_U$ contains substantial information about $Z$. However, this looks different, when $U$ contains all points of codimension one. The results of this section can only be applied in the case of a small representation.

\begin{definition}
A module scheme up to normalization $Z \rightarrow X$ over a normal scheme $X$ is called \emph{reflexive} if there exists an open subset $U \subseteq X$ containing all points of codimension one such that the restriction $Z|_U$ is a vector bundle.	
\end{definition}

Let $\mathcal F$ be the corresponding locally free sheaf on such an open subset  $U \subseteq X$. Then, for every open subset $U' \subseteq U$ also containing all points of codimension one, the global sections of $\mathcal F$ on $U$ and on $U'$ coincide. We call $\Gamma(U, {\mathcal F})$ the corresponding $R$-module of the reflexive bundle, it is a reflexive module. In this setting, we call $Z$ a geometric realization of $\mathcal F$, of $Z_U$ and of $\Gamma(U, {\mathcal F})$.

\begin{Definition}
	An $R$-module $M$ is called \emph{fiberflat} if it has a reflexive fiberflat realization, i.e., if there exists a fiberflat bundle $Z \rightarrow X$ and an open dense subset $U \subseteq X$ containing all points of codimension one such that $Z|U \cong (\Spec (\Sym M))|_U $ is a vector bundle.	
\end{Definition}

\begin{Lemma}
\label{smallreflexive}
Let $\beta $ be a linear faithful small representation of a finite group $G $, and let $X={\mathbb A}^d/G$ be the quotient. Let $\rho$ be a linear representation of $G$ and let $Z_\rho$ be the corresponding fiberflat bundle. Then $Z_\rho$ is reflexive.
\end{Lemma}
\begin{proof}
A small action is fixed point free in codimension one, hence, this follows from Corollary \ref{modulealgebraisomorphism}.
\end{proof}

\begin{Corollary}
\label{invariantringfiberflat}
	Let the finite nonmodular group $G$ act linearly and faithfully on affine space ${\mathbb A}^d$ over $K$  with quotient scheme $X=\Spec R$. Let $U \subseteq X$ be the regular locus. Then the following hold.
	
	\begin{enumerate}
		\item 
	The tangent bundle $T_U$ is fiberflat.

	\item
		Suppose that $K$ is a perfect field of positive characteristic so that $R$ is $F$-finite. Then, the push-forwards $F^{e}_*{ {\mathcal O}_U}$ under iterations of the Frobenius are  fiberflat.
		
	\end{enumerate}
\end{Corollary}

\begin{proof}
We mod out the action of the subgroup generated by the reflections and can then assume that the action $\beta$ is small.
The action is then free on a subset $V$ containing all points of codimension one with image $U \subseteq X$, which is the regular locus of $X$. For the tangent bundle on $U$, we have by Corollary
\ref{tangentbundlelinearisomorphism}, $T_U =T_V/ \beta \times \beta=V \times {\mathbb A}^d/ \beta \times \beta $. The quotient scheme $T_{ {\mathbb A}^d}/   \beta \times \beta =  {\mathbb A}^d \times {\mathbb A}^d / \beta \times \beta$ gives by Lemma \ref{fiberflatquotient} a fiberflat realization of $T_U$.
The proof of th second statement is analogous.
\end{proof}

\begin{Example}
The algebra
	\[ K[X,Y][D,E,F]/(E^2-X^2D, F^2-Y^2D, EF-DXY, YE-XF )  \,  \]
	obtained in Example \ref{pullbacknilpotent} describes a fiberflat reflexive bundle over $ {\mathbb A }^2$ which is not normal.
\end{Example}

\begin{Lemma}
Let $X=\Spec R$ be a normal noetherian affine scheme over $K$ and let $ L \rightarrow U $ be a line bundle on an open subset $U \subseteq X$ containing all points of codimension one and such that $ L$ is a torsion element in the Picard group of $U$ (which is the divisor class group of $X$). Then $ L$ has a geometric realization $Z \rightarrow X$ with a reflexive fiberflat bundle $Z$.
\end{Lemma}
\begin{proof}
Let $n$ be the order of $L$ and let $\mathcal L$ be the corresponding invertible sheaf on $U$. Choosing an isomorphism ${ \mathcal L}^{\tensor n} \cong {\mathcal O}_U$, we can construct an ${\mathcal O}_U$-algebra ${\mathcal A} = \oplus_{i =0}^{n-1} { \mathcal L}^{\tensor i}$ and hence a finite scheme $ V \rightarrow U $ which trivializes $L$. Going to the integral closure $S$ of $R$ in the quotient field of $V$, this map can be extended to $Y=\Spec S \rightarrow X$. The group $G=\Z/(n)$ is acting on $V$ and on $Y$ with quotients $U$ and $X$. The group acts also on ${\mathbb A}^1$ such that the quotient $ ( Y \times {\mathbb A}^1)/G$ gives by Lemma  \ref{smallreflexive} and Lemma \ref{fiberflatquotient} a reflexive fiberflat bundle over $X$, which gives back above $U$ the line bundle $L$.
\end{proof}

The same conclusion does probably also hold for finite bundles, see \cite{norirepresentations}. Example \ref{quadriclinebundle} shows that a
line bundle which does not define a torsion class, there is in general no fiberflat realization.

\begin{Lemma}
\label{reflexivebundle}
Let $X =\Spec R$ be a normal affine scheme of finite type over $K$ and let $p:Z=\Spec A \rightarrow X$ be a normal reflexive fiberflat bundle. Let $U \subseteq X$ be an open subset containing all points of codimension one such that $Z |_U$ is a vector bundle,  let $M$ be the corresponding module and set $V=p^{-1}(U)$. Then
\[ A = \Gamma (V,  {\mathcal O}_Z) = \bigoplus_{n \in \N} \Gamma(U,  \Sym^n M)  = \bigoplus_{n \in \N} (\Sym^n M)^{**}  \, .  \] 
Moreover, this is an $R$-algebra of finite type and $M$ is a finite $R$-module.
\end{Lemma}
\begin{proof}
Let $d = \dim X$ and $r$ be the rank of the bundle, so that $Z$ has dimension $d+r$. The complement of $U$, say $Y=X \setminus U$, has dimension $\leq d-2$, hence the dimension of $p^{-1}(Z)$ is $\leq d+r-2$. Therefore, also $V$ contains all points of codimension one, which implies the first equality, as $A$ is normal. Let $\mathcal F$ denote the corresponding locally free sheaf on $U$. Over $U$, the vector bundle $Z|_U$ is the (relative) spectrum of the graded ${\mathcal O}_U$-algebra $\bigoplus_{n \in \N}  \Sym^n {\mathcal F}$. The ring of global sections of this spectrum is $\bigoplus_{n \in \N} \Gamma(U, \Sym^n {\mathcal F})$. As the coherent module $\tilde{M}$ corresponding to $M$ restricts to $\mathcal F$ on $U$, this equals $\bigoplus_{n \in \N} \Gamma (U, \Sym^n M)$. As $U$ contains all points of codimension one, we have that $ \Gamma (U, \Sym^n M)$ is the reflexive hull of $ \Sym^n M $. The finiteness condition of this algebra is inherent in the fiberflat condition. The finiteness of $M$ follows, as $M$ is the degree one part of this algebra.
\end{proof}

The algebra $ \bigoplus_{n \in \N} (\Sym^n M)^{**} $ is called the \emph{factorial closure} of $M$, see \cite[Chapter 7]{vasconcelosarithmetic}, \cite{samuelanneauxfactoriels}. The factorial closure gives the canonical fiberflat realization of a reflexive fiberflat module.

\begin{Corollary}
	\label{fiberflatfactorialclosure}
Let $X =\Spec R$ be a normal affine scheme of finite type over $K$ and let $q:W=\Spec B \rightarrow X$ be an integral reflexive fiberflat bundle.
Then there is also a normal reflexive fiberflat bundle $p:\Spec A \rightarrow X$ realizing the same reflexive module. In fact, one can take $A$ to be the factorial closure of the module.
\end{Corollary}
\begin{proof}
Let $U \subseteq X$ be an open subset containing all points of codimension one such that $W |_U$ is a vector bundle, let $M$ be the corresponding module and set $V=q^{-1}(U)$. Then $B \rightarrow \Gamma(V, {\mathcal O}_W)=A$ is the normalization, as it is finite, birational and since $V$ is normal. Therefore, $\Spec A \rightarrow X$ gives a normal fiberflat bundle by Lemma \ref{fiberflatconstruction} (4), which restricts to the same vector bundle on $U$. The statement about the factorial closure follows from Lemma \ref{reflexivebundle}.
\end{proof}

The following two results will be important in Section \ref{irreduciblesection} where we characterize irreducible fiberflat bundles.

\begin{Lemma}
\label{reflexivebundleproduct}
Let $X =\Spec R$ be a normal affine scheme of finite type over $K$ and let $M_1,M_2$ be finitely generated reflexive fiberflat $R$-modules with geometric realizations $Z_i =  \Spec A_i$, where $ A_i = \bigoplus_{k \in \N} ( \Sym^k (M_i))^{**}$. Then $M_1 \oplus M_2$ is also reflexive and fiberflat, and the normalization of $Z_1 \times_X Z_2$ equals $ \Spec  A    $, where $ A= \bigoplus_{n \in \N} (\Sym^n (M_1 \oplus M_2))^{**} $.
\end{Lemma}
\begin{proof}
It is part of the definition of fiberflat that $Z_i$ and hence $Z_1 \times_X Z_2$ are of finite type. The fibers of the product scheme $ q:Z_1 \times_X Z_2 \rightarrow X$ have constant dimension. Let $U \subseteq X$ be an open subset containing all points of codimension one and such that $M_1$ and $M_2$ are locally free on $U$. We denote the corresponding locally free sheaves on $U$ by ${\mathcal F}_i$. Let $V= q^{-1}(U) =Z_1 \times_U Z_2$. The complement of $V$ has codimension $\geq 2$, as $U$ has and by the fiberflat property. Hence the restriction map
\[   A_1 \tensor_R A_2 = \Gamma(Z_1 \times_X Z_2, {\mathcal O}_{Z_1 \times_X Z_2} ) \longrightarrow \Gamma(V, {\mathcal O}_{Z_1 \times_X Z_2} )  \]
is finite by \cite[Corollaire 5.11.4]{ega4.2}, it is an isomorphism above $U$, thus birational, and therefore, it is the normalization, as $V$ is normal.
We also have
\[ \Gamma(V, {\mathcal O}_{Z_1 \times_X Z_2} )   = \bigoplus_{n \in \N}  \Gamma(U,  \Sym^n ( {\mathcal F}_1 \oplus {\mathcal F}_2 ) )    = \bigoplus_{n \in \N}   (\Sym^n ( M_1 \oplus M_2 ) )^{**}   \, . \]
\end{proof}

\begin{Lemma}
\label{reflexivebundlesplitting}
Let $X =\Spec R$ be a normal affine scheme of finite type over $K$ and let $p:Z=\Spec A \rightarrow X$ be a normal reflexive fiberflat bundle. Let $U \subseteq X$ be an open subset containing all points of codimension one such that $Z|_U$ is a vector bundle,  let $M$ be the corresponding module and set $V=p^{-1}(U)$. Suppose that $M$ has a decomposition $M=M_1 \oplus M_2$. Then, $M_1$ and $M_2$ are also fiberflat and reflexive.
\end{Lemma}
\begin{proof}
By Lemma \ref{reflexivebundle}, we have
\[ A= \bigoplus_{n \in \N} (\Sym^n (M_1 \oplus M_2))^{**} = \bigoplus_{n \in \N} \bigoplus_{k + \ell =n} (  \Sym^k (M_1) \tensor_R  \Sym^\ell (M_1))^{**} \, .\]
We set $A_1 = \bigoplus_{k \in \N} ( \Sym^k (M_1))^{**}$ and $A_2 = \bigoplus_{\ell \in \N} ( \Sym^\ell (M_2))^{**}$. We have
\[A_1 \tensor_R A_2 = \bigoplus_{n\in \N} \bigoplus_{k + \ell =n} (  \Sym^k (M_1))^{**} \tensor_R ( \Sym^\ell (M_1))^{**} \, ,\]
and we get a natural ring homomorphism $ A_1 \tensor_R A_2 \rightarrow A $, which is an isomorphism over $ U $. We have the subring relation (by letting $\ell=0$)
\[ A_1= \bigoplus_{k \in \N}   ( \Sym^k (M_1))^{**}  \subseteq \bigoplus_{n \in \N} \bigoplus_{k + \ell =n} (  \Sym^k (M_1) \tensor_R  \Sym^\ell (M_1))^{**} =A \, .   \]
The direct sum of all summands of $A$ where $\ell \geq 1$ is an ideal in $A$; this shows that there is also a ring homomorphism $A \rightarrow A_1$ by moding out this ideal. This implies that $A_1$ is of finite type over $R$. It is clear that the $M_i$ are also locally free over $U$ and that they are reflexive.

In geometric terms, we have morphisms
$q_i:Z \rightarrow Z_i = \Spec A_i$ and $\varphi_i : Z_i \rightarrow Z$ such that $   q_i \circ \varphi_i   =  \operatorname{Id} _{Z_i}$. Set $p_i:Z_i \rightarrow X$. We also have $q_1 \circ \varphi_2=  0_1 \circ p_2  $ (as a morphism $Z_2 \rightarrow Z_1$), where $0_1$ denotes the zero section of $Z_1$, which corresponds to the projection from $A_1$ to $R$ modulo $(A_1)_+$.

We want to show that $q_1 \times q_2:Z  \rightarrow Z_1 \times_X Z_2$ is surjective, for this, we may assume that $K$ is algebraically closed. Let $(P_1,P_2) \in  Z_1 \times_X Z_2$ be a $K$-point and consider $(\varphi_1(P_1), \varphi_2(P_2))  \in Z \times_X Z$. We have a commutative diagram
\[ \begin{matrix}   Z \times_X Z & \longleftarrow & \widetilde{Z \times_XZ} & \stackrel{\alpha}{ \longrightarrow} &   Z \\     \!\!\!\!\!   \!\!\!\!\!  \!\!\!\!\! q_1 \times q_1 \downarrow & &   \!\!\!\!\!   \!\!\!\!\!  \!\!\!\!\!  q_1 \times q_1 \downarrow & & \downarrow q_1  \!\!\!\!\!  \\    Z_1 \times_X Z_1   &  \longleftarrow & \widetilde{Z_1 \times_XZ_1} & \stackrel{\alpha_1}{\longrightarrow} &Z_1 \, .   \end{matrix}\]
Let $w \in   \widetilde{Z \times_XZ} $ be a point above  $(\varphi_1(P_1), \varphi_2(P_2)) $. On the left-hand side, this element is sent to $(P_1,0) $, therefore,  $\alpha(w)$ is sent via $q_1$ to $P_1$, since $ \alpha_1  \circ  (\operatorname{Id}_{Z_1} \times 0_1 )= \operatorname{Id}_{Z_1}$. Hence, $(q_1 \times q_2 ) ( \alpha (w) ) = (P_1,P_2) $. We now look at the fibers of the morphism $q_1 \times q_2:Z  \rightarrow Z_1 \times_X Z_2$  above a point $P \in X$. The fibers of $Z_i$ cannot be empty, as the fibers of $Z$ are not empty. Also, by \cite[Theorem 14.8]{eisenbud}, the dimension of the nonempty fibers of $Z_i$ cannot drop. The fibers of $Z$ have dimension $r=r_1+r_2$, where $r_i$ is the generic rank of $M_i$, and the dimensions on the right are $ \geq r_1$ and $\geq r_2$, so, by surjectivity, we have equality. This means that $M_1$ is fiberflat as well.
\end{proof}

\section{Fiberflat bundles on regular rings}

\label{regularfreesection}

For an invariant ring $R=S^G$, there are in general many maximal Cohen-Macaulay modules and we would like to distinguish the ones coming from a representation from the others. The criterion, that the first ones are (in the indecomposable case) the direct summands of $S$, is very unintrinsic and not satisfactory. In Lemma \ref{fiberflatquotient}, we have noted that an invariant algebra defines a fiberflat bundle, which is, in the small case, also reflexive by Lemma \ref{smallreflexive}. Here we will show that fiberflatness can distinguish in many cases between maximal
Cohen-Macaulay modules coming from a representation and the others, but the picture is not complete yet. What is missing is a satisfactory understanding of fiberflat bundles for regular rings.

\begin{Lemma}
\label{notfiberflat}
Let $R$ be a noetherian excellent local domain of dimension $d \geq 3$, $M$ a finitely generated $R$-module of rank $r$. Let $U \subseteq \Spec R$ denote the locus where $M$ is locally free, suppose that $U$ is regular, contains all points of codimension one, and set $Y=X \setminus U$. Let $p:Z \rightarrow \Spec R$ be an irreducible surjective geometric realization of $M$, which is not fiberflat. Suppose that the dimension of $p^{-1}( Y )$ is $ \leq r +d -2$. Then, $M$ is not fiberflat.
\end{Lemma}
\begin{proof}
The dimension of $Z$ is $ \geq d+r$ (look at a chain of prime ideals on the zero section and then continue generically), so the complement of the vector bundle $p^{-1}(U)$ in $Z$ has codimension $\geq 2$. Hence, $A \rightarrow \Gamma(p^{-1}(U), { \mathcal O}_Z) =C$ is a finite homomorphism by \cite[Corollaire 5.11.4]{ega4.2} and $C$ is the normalization of $A$. Then $C$ gives a normal realization which is also not fiberflat. Suppose that $q:\Spec B \rightarrow X$ is an integral fiberflat geometric realization of $M$. Then the dimension of $ q^{-1}(Y) $ is $ \leq r+ d -2$, since $\dim Y \leq d-2$. Therefore, $C$ is also the normalization of $B$, which yields a contradiction.
\end{proof}

For a finitely generated $R$-module $M$, we denote its minimal number of generators by $\nu=\nu(M)$.

\begin{Corollary}
\label{modulegeneratorsnotfiberflat}
Let $(R, \idealm) $ be a noetherian excellent local domain of dimension $d \geq 3$, let $M$ be a finitely generated $R$-module of rank $r$ which is locally free outside $\{ \idealm \}$. Suppose that $r< \nu (M) \leq r+ d-2$ and that the symmetric algebra $\Sym M$ is irreducible. Then, $M$ is not fiberflat.
\end{Corollary}
\begin{proof}
Let $Z=\Spec ( \Sym (M)) \rightarrow \Spec R$. It has dimension $d+r$. The fiber above $\idealm$ has dimension $\nu(M) >r$, hence, $Z$ is not fiberflat. Therefore, all conditions of Lemma \ref{notfiberflat} are fulfilled and the result follows.
\end{proof}

\begin{Example}
Let $R=K[X,Y,Z]$ and consider the module $M=R^3/(  Xe_2,Ye_2,Ze_2, Xe_3,Ye_3,Ze_3 )$, its symmetric algebra is
\[  R[A,B,C]/(XB,YB,ZB, XC,YC,ZC ) \, \]
 which is reducible. $M$ is fiberflat, as on the punctured spectrum, it is just the free module of rank one, the dimension of the fiber above $(X,Y,Z)$ is $3$. This shows that we need the irreducibility condition in Corollary
\ref{modulegeneratorsnotfiberflat}. In the situation of this corollary, one can try to apply Lemma \ref{notfiberflat} to the symmetric algebra modulo torsion, but it may happen, as in the example, that the exceptional dimension does not contradict fiberflatness anymore.
\end{Example}

\begin{Corollary}
\label{fiberflatfree}
Let $(R, \idealm) $ be noetherian excellent local domain of dimension $d \geq 3$, $M$ a finitely generated $R$-module such that $M$ is locally free outside $\idealm$. Suppose that $M$ is fiberflat, that the symmetric algebra $\Sym M$ is irreducible and that $\nu(M) \leq r+d-2$. Then, $M$ is free.
\end{Corollary}
\begin{proof}
This follows from Corollary \ref{modulegeneratorsnotfiberflat} 
\end{proof}

For $\nu(M) = r+d-1$, the fiber over the maximal ideal has codimension one and then there are new (transcendental) functions defined on the open complement $Z \setminus p^{-1}(\idealm)$, which changes the situation completely, as in Example \ref{symmetricnewfunction} below. For $\nu(M) \geq r+d$, the fiber is a component of its own and we have to look at the situation after moding it out (going to the Rees algebra).

\begin{Corollary}
Let $R=K[X_1, \ldots, X_d]^G$ be an invariant ring for a linear action of a finite nonmodular group. Suppose that $R$ has an isolated singularity and that the symmetric algebra $\Spec ( \Sym ( \Omega_{R|K})) $  is irreducible. Then the embedding dimension of $R$ is $\geq 2d-1$.
\end{Corollary}

\begin{proof}
We may assume $d \geq 3$ and want to apply Corollary \ref{modulegeneratorsnotfiberflat} for the $R$-module $\Omega_{R|K}$. By Corollary \ref{invariantringfiberflat}, $\Omega_{R|K}$ is fiberflat. It is not free because of the singularity. Hence, we get $\operatorname{embd} (R) \geq \nu (\Omega_{R|K} ) \geq 2d-1$.
\end{proof}

\begin{Remark}
\label{regularfiberflatconjecture}
We suspect that for a regular ring, Corollary \ref{fiberflatfree} holds without the condition on the number of generators. The intuition behind this is the following: A fiberflat bundle has the property that going modulo a regular element of the base ring, the constance of the dimension is preserved. By Lemma \ref{fiberflatirreducible}, a fiberflat bundle with an irreducible central fiber is irreducible. Hence, a fiberflat bundle behaves in some geometric sense like a maximal Cohen-Macaulay module. And for a regular ring, a maximal Cohen-Macaulay module is free. 
\end{Remark}

\begin{Remark}
By \cite[Theorem 1.1]{galligokwienciski},
the flatness of $R \rightarrow A$ for $R$ regular, everything of finite type over $\C$, is equivalent to the property that $A \tensor_R A \tensor_R \ldots \tensor_R A$ ($\dim R$ factors) is torsionfree (hint of Adrian Langer). Fiberflatness implies at least that this tensor product is irreducible, as the argument in Lemma \ref{fiberflatirreducible} shows.
\end{Remark}

The following example shows that for a line bundle which does not define a torsion class in the divisor class group (even in a toric ring), there is in general no fiberflat realization.

\begin{example}
\label{quadriclinebundle}
Let $R=K[X,Y,U,V]/(UX-VY)$ and consider the prime ideal $\idealp= (X,Y)$, which generates the divisor class group $\Z$ of this ring. According to Example \ref{idealsymmetricalgebra}, its symmetric algebra is $R[S,T]/( YS-XT,US-VT)$, which is a domain, as it is the determinant ring for the matrix $\begin{pmatrix} X & V &S \\ Y &U &T  \end{pmatrix}$.  realization. Due to Corollary \ref{modulegeneratorsnotfiberflat}, $ \idealp$ is not fiberflat.
\end{example}

\begin{Example}
The module $M=\Syz(X,Y,Z)^*$ over $K[X,Y,Z]$ is realized geometrically as 
\[ \Spec K[X,Y,Z][R,S,T]/(RX+SY+TZ) \]
(the syzygies are the sections), which is a domain of dimension $5$. Above $D(X,Y,Z)$, the fibers has dimension two, but above $V(X,Y,Z)$, the fiber has dimension three. By Lemma \ref{notfiberflat}, $M$ is not fiberflat.
\end{Example}

For the following example, I thank Craig Huneke for a hint.

\begin{Example}
\label{symmetricnewfunction}
We consider
\[\Spec K[X,Y,Z][A,B,C,D,E]/(XA+YB+ZC,XC+YD+ZE) \rightarrow \Spec K[X,Y,Z] \, , \]
the spectrum of the symmetric algebra coming from the resolution (see Example \ref{moduleschemepresentation})
\[  0 \longrightarrow S^2 \longrightarrow S^5 \longrightarrow M \longrightarrow 0 \, .\]
$M$ is locally free above $D(X,Y,Z)$ since $X^2,Y^2$ and $Z^2$ appear as minors of the presenting matrix. The rank of the bundle is $3$, and its dimension as a scheme is $6$, the fiber above $V(X,Y,Z$) has dimension $5$, hence, this fiber has codimension $1$ in the total space and its complement must have sections not coming from the ring. In fact,
\[U= \frac{CD-EB}{X} = \frac{C^2-AE}{Y}  = \frac{CB-AD}{Z} \]
is a section. Adjoining this element to the symmetric algebra, we get a new affine realization of $M$ where now above $V(X,Y,Z)$, we have $6$ variables but with the relations $CD- EB,C^2-AE,CB-AD=0$, which gives fiber dimension $4$. So $M$ is not fiberflat by Lemma \ref{notfiberflat}.  
\end{Example}

We describe a typical situation that yields maximal Cohen-Macaulay modules on an invariant ring that are not fiberflat; for the first part, compare \cite[Proposition 1.6]{auslanderreitencohenmacaulay}.

\begin{Lemma}
\label{quotientnotfiberflat}
Let $G$ be a finite nonmodular group acting faithfully and linearly on ${\mathbb A}^d$, let $\chi_i$, $i \in I$, and $\chi_j$, $j \in J$, be finite sets of characters for $G$ and let $f_{ij} \in K[X_1, \ldots, X_d]=S$ be a set of $ \chi_{i}^{-1} \chi_j$-semi-invariant elements such that the cokernel of $S^{|I|} \stackrel{( f_{ij})}{\rightarrow} S^{|J|}$ is Artinian. Suppose further that the cokernel has no $G$-invariants $\neq 0$. Then the invariant module of the kernel is a maximal Cohen-Macaulay $R$-module, $R=S^G$. If $d \geq 3$ and $r < \nu(\ker) \leq r+d-2$, then it is not fiberflat.
\end{Lemma}

\begin{proof}
We look at the exact sequence
\[ 0 \longrightarrow \ker  \longrightarrow  \bigoplus_{i \in I} S(\chi_i) \stackrel{( f_{ij})}{\longrightarrow} \bigoplus_{j \in J} S(\chi_j) \longrightarrow \coker \longrightarrow 0 \, .\]
Here, $S(\chi)$ denotes $S \tensor (K, \chi)$, so $S$, but with the action $\circ$ of $G$ given as $s \circ g= \chi(g) (s) g $. For $s \in S ({\chi_i})$, we have in $S ({\chi_j})$ the equalities
\[ ( s f_{ij}) \circ g =\chi_i(g)  \cdot (s)g \cdot (f_{ij})g =  (s) g \cdot \chi_i(g) \cdot \chi_i^{-1} (g) \cdot \chi_j (g)  \cdot f_{ij}  =    (s)g \cdot \chi_j(g) \cdot f_{ij} = (s \circ g)  f_{ij} \, . \]
Therefore, the matrix describes a $G$-equivariant $S$-module homomorphism. Hence, the cokernel and the kernel carry a natural action of $G$ on them. Passing to the invariants, we get the short exact sequence of $R=S^G$-modules
\[ 0 \longrightarrow \ker^G  \longrightarrow  \bigoplus_{i \in I} S^{\chi_i} \stackrel{( f_{ij})}{\longrightarrow} \bigoplus_{j \in J} S^{\chi_j} \longrightarrow   0 \, ,\]
where the $0$ on the right follows from our assumption. The surjection and the Cohen-Macaulayness of $S^\chi$ ensure that $H^i(D(R_+), \ker^G)=0$ for $i=1, \ldots, d-1$. Therefore, by the cohomological criterion for maximal Cohen-Macaulay, we get the first statement.
	
Let $d \geq 3$ and suppose that $\ker$ fulfills the condition on its rank. Then, by Corollary \ref{modulegeneratorsnotfiberflat}, $\ker$ is not fiberflat, and hence also $\ker^G$ is not fiberflat.
\end{proof}

\begin{Example}
We consider $S=K[X,Y,Z]$ with the action of $\Z/(2)$ by negation on every variable. The invariant ring $R$ is the second Veronese algebra of dimension three, which is the only invariant ring of dimension $\geq 3$ with only finitely many maximal Cohen-Macaulay modules, see \cite[Theorem 4.1]{auslanderreitencohenmacaulay}, \cite[Theorem 16.5]{leuschkewiegand}. We consider, according to Lemma \ref{quotientnotfiberflat}, the  exact sequence
\[ 0 \longrightarrow \Syz(X,Y,Z)  \longrightarrow S^3 \stackrel{X,Y,Z}{\longrightarrow} S(\chi) \longrightarrow S/(X,Y,Z)(\chi) \longrightarrow 0 \, ,\]
where $\chi$ is the nontrivial character. Taking $G$-invariant modules gives the short exact sequence
\[   0 \longrightarrow M=( \Syz(X,Y,Z))^G   \longrightarrow R^3 \stackrel{X,Y,Z}{\longrightarrow} S^\chi \longrightarrow 0\, .\]
$S$ has the decomposition $S=R \oplus S^\chi$ into homogeneous elements of even and of odd degree, $M$ is a maximal Cohen-Macaulay module, which is not a submodule of $S$. By Lemma \ref{quotientnotfiberflat}, $M$ is not fiberflat.
\end{Example}

\section{Reflections}

\label{reflectionssection}

We describe possible reflections for a product representation. The basic trivial observation is the following.

\begin{Lemma}
\label{elementreflection}
Let $G$ be a finite group, and let $\beta$ and $\rho$ be $K$-linear representations of dimensions $d$ and $m$. Then, for $g \in G$, the linear mapping $ (\beta \times \rho)(g) \in  \GLG_d(K)  \times  \GLG_m(K)$ is a reflection if and only if $\beta(g)$ is a reflection and $\rho(g)$ is the identity or the other way around.
\end{Lemma}
\begin{proof}
This is clear from the block matrix decomposition
\[  (\beta \times \rho)( g) = \beta(g) \times \rho(g)= M_1 \times M_2 \in \GLG_d(K) \times \GLG_m(K) \, .\]
\end{proof}

\begin{Corollary}
\label{faithfulsmall}
Let $G $ be a finite group, and let $\beta$ and $\rho$ be faithful $K$-linear representations of $G$ of dimensions $d$ and $m$. Then, the embedded group $ (\beta \times \rho) (G) \subseteq  \GLG_d(K)  \times \GLG_m(K)$ is a small subgroup.	
\end{Corollary}
\begin{proof}
This follows from Lemma \ref{elementreflection}.
\end{proof}

\begin{Corollary}
\label{smallsmall}
Let $G $ be a finite group, and let $\beta$ be a small faithful $K$-linear representation and $\rho$ be any $K$-linear representations of $G$ of dimensions $d$ and $m$. Then, the action $ \beta \times \rho$ is also small.	
\end{Corollary}
\begin{proof}
This follows from Lemma \ref{elementreflection}.
\end{proof}

\begin{Corollary}
\label{subgroupreflection}
Let $G$ be a finite group, $\beta$ be a faithful $K$-linear representation of $G$ on $K^d$ and let $\rho$ be a $K$-linear representation of $G$ on $K^m$ with kernel $H\subseteq G$. Then, the elements $g \in G \setminus H $ are not reflections, considered in $  \GLG_d(K)  \times  \GLG_m(K)$.	
\end{Corollary}
\begin{proof}
For $g \in G \setminus H $, in the block matrix decomposition $(\beta \times \rho) (g) = M_1 \times M_2 \in \GLG_d(K) \times \GLG_m(K) $, both matrices are not the identity, so $g$ cannot be a reflection in the product representation by Lemma \ref{elementreflection}.
\end{proof}

\begin{Corollary}
\label{productnoreflection}
Let $G$ be a finite group with a $d$-dimensional faithful $K$-linear representation $\beta$ and let $\rho$ be a nontrivial $K$-linear representation of $ G $ on $K^m$. Then, the embedded group $(\beta \times \rho)(G) \subseteq  \GLG_d(K)  \times \GLG_m(K)$ is not generated by reflections.	
\end{Corollary}
\begin{proof}
Let $H \subset G$ be the kernel of $ \rho $. Then, the elements $ g \in H $ have in $ \GLG_d(K) \times \GLG_m(K) $  the form $ \beta(g) \times\operatorname{Id}_{K^m} $, and  the elements $ g \in G \setminus H $ are  by Corollary \ref{subgroupreflection} not a reflection. So the subgroup generated by the reflections is contained in the subgroup $ \beta(H) \times \{ \operatorname{Id}_{K^m} \} $, which is not $ ( \beta \times \rho) ( G) $.
\end{proof}

For a finite group $G$ and a representation $\beta:G \rightarrow \GLG_d(K)$, we denote by $G_{\beta \operatorname{refl} }$ the normal subgroup of $G$, which is the preimage of the subgroup generated by reflections.

\begin{Lemma}
\label{reflectionsubgrouprelation}
Let $G $ be a finite group, and let $\beta$ be a faithful linear representation and let $\rho$ be any $K$-linear representation of $G$ of dimension $d$ and $m$. Then, $G_{\beta \times \rho \operatorname{refl} } \subseteq G_{\beta \operatorname{refl} }$, and we have a surjective group homomorpism $G/ G_{\beta \times \rho \operatorname{refl} }  \rightarrow   G/  G_{\beta \operatorname{refl} } $.
\end{Lemma}
\begin{proof}
This is again clear from the block matrix decomposition.	
\end{proof}

\begin{Example}
\label{cycliconereflections}	
We look at the situation described in Lemma \ref{cyclicone}. Via $\beta$, $G$ acts as a reflection group, so $ G_{\beta \operatorname{refl} }=G$. Let $n$ be the order of $\ell$ in $\Z/(k)$. For the representation $\beta \times \rho_\ell$, the power $\begin{pmatrix} \zeta & 0 \\ 0 & \zeta^\ell \end{pmatrix}^m$ is a reflection if and only if $m$ is a multiple of $n$. Hence, $  G_{\beta \times \rho_\ell \operatorname{refl} } = \langle n \rangle \subseteq \Z/(k) $ is the subgroup of reflections, and  $G/ G_{\beta \times \rho \operatorname{refl} }  = \Z/(n)$. 
\end{Example}

\begin{Lemma}
\label{stabilizercompare}
Let $R' =(Q,v) \in  {\mathbb A}^d \times {\mathbb A}^m$ be a point above $Q \in  {\mathbb A}^d $.

\begin{enumerate}
\item 
We have
$ \Stab^{\beta \times \rho} R' =  \Stab^{\beta} (Q) \cap \Stab^\rho (v)  $.

\item
We have $ \Stab^{\beta \times \rho} R' \subseteq \Stab^{\beta} Q $.

\item
If $R'$ lies on the zero section, then we have the equality $ \Stab^{\beta \times \rho} R' = \Stab^{\beta} Q $.
 
\end{enumerate}
\end{Lemma}
\begin{proof}
Clear.
\end{proof}

\section{Singularities}
\label{singularitiessection}

In this section, we restrict to a linear action $\beta$ of a finite group $G$ on ${\mathbb A}^d$ and another linear action $\rho$ of $G$ on ${\mathbb A}^m$, giving rise to the quotient scheme up $Z_\rho= ( {\mathbb A}^d \times  {\mathbb A}^m )/(\beta \times \rho) \rightarrow  {\mathbb A}^d/\beta=X$. We want to understand the singularities on $Z_\rho$, in particular their relation with the singularities of $X$. As they are always quotient singularities, they are normal Cohen-Macaulay (in the nonmodular case), rational singularities, in positive characteristic (again nonmodular) they are $F$-regular. Here, we are rather interested in the existence and the localization (in the sense of where are they?) of singularities. For this question, the theorem of Chevalley, Shephard, Todd, Serre is decisive (see \cite{chevalley}, \cite{shephardtodd}, \cite{serregroupesfinis}). It says that a nonmodular group $G \subseteq \GLG_d(K)$ is a reflection group if and only if the invariant ring is itself a polynomial ring, and this is true if and only if the invariant ring is regular.

\begin{Lemma}
\label{singularityexistence}
A quotient scheme $Z_\rho \rightarrow X$ always has a singularity unless $\beta(G)$ is a reflection group and $\rho$ is the trivial representation. In this case, the quotient scheme is just the trivial bundle ${\mathbb A}^d \times {\mathbb A}^m \rightarrow {\mathbb A}^d$. 
\end{Lemma}
\begin{proof}
This follows from Corollary \ref{productnoreflection} in connection with the theorem of Chevalley, Shephard, Todd and Serre, applied to  ${\mathbb A}^d \times {\mathbb A}^m$. The second statement follows from this theorem applied to  ${\mathbb A}^d $.
\end{proof}

We use the following fact, which is a slight generalization of the theorem of Chevalley, Shephard, Todd and Serre, as a reference, we only found \cite[Satz 1]{gottschling} in the complex case (see also \cite{prill}). I thank  Mitsuyasu Hashimoto
for help with the following proof.

\begin{Lemma}
\label{reflectionsingular}
Let $\beta $ be a linear faithful representation of a nonmodular finite group $G$, let $X={\mathbb A}^d/G$ be the quotient and let $Q \in {\mathbb A}^d$ be a closed point with image point $P \in X$. Let $\Stab (Q) $ be the stabilizer of $Q$. Then, $P$ is a regular point if and only if $\beta( \Stab (Q) ) \subseteq \GLG(K^d)$ is generated by reflections.	
\end{Lemma}

\begin{proof}
We may assume that $K$ is algebraically closed. Let $H= \Stab (Q)$, $Q=(a_1, \dots, a_d)$, $Y_i=X_i-a_i$. Let  $\idealq  \subseteq S=K[ X_1, \ldots, X_d]$ be the maximal ideal corresponding to $Q$, $\idealp= \idealq \cap S^G$. From the ring homomorphisms $ S^G \rightarrow S^H \rightarrow S$ we get natural ring homomorphisms \[ ( S^G)_\idealp \longrightarrow (S^H)_{ \idealq \cap S^H } \cong (S_\idealq)^H  \longrightarrow   S_\idealq \, . \]
The homomorphism $( S^G)_\idealp \rightarrow  (S_\idealq)^H $ is unramified by \cite[Theorem 4.1.2]{nagatalocal}, hence,  $( S^G)_\idealp$ is regular if and only if  $(S_\idealq)^H $ is regular. But we can understand the action of $H$ on ${\mathbb A}^d$ as a linear action with respect to the new variables $Y_i$, and then apply the Theorem of Chevalley, Shephard, Todd.
\end{proof}

The following fact is known as theorem of Steinberg, see \cite[Theorem 1.5]{steinbergdifferential}, we include a short proof based on the previous lemma.

\begin{Corollary}
\label{reflectiongroupsubgroup}
Let $ G \subseteq \GLG_d(K) $ be a nonmodular reflection group and let $Q \in {\mathbb A}^d$ be a point. Then the stabilizer group $\Stab (Q) $ is also a reflection group.
\end{Corollary}
\begin{proof}
By the theorem of Chevalley, Shepherd, Todd, the quotient is ${\mathbb A}^d$, so it is regular everywhere. Hence it follows from Lemma \ref{reflectionsingular} that $S=\Stab (Q)$ is generated by reflections.	
\end{proof}

\begin{Corollary}
\label{smallsingularities}
Let $\beta $ be a linear faithful small representation of a finite group $G $, let $X={\mathbb A}^d/G$ be the quotient. Then the singular locus of $X$ is the image of all fixed spaces of $ g \in G $, $ g \neq \operatorname{Id}$.
\end{Corollary}
\begin{proof}
Let $V\subseteq {\mathbb A}^d$ be the fixed point free locus. The inclusion $V/G \subseteq X_{ \operatorname{reg} } $ always holds: Over the image of the free locus, the map $V \rightarrow V/G$ is a principal fiber bundle, and $V/G$ is smooth, see Remark \ref{freeprincipal}. Now, if $ Q \in {\mathbb A}^d$ lies on the fixed space for some $g \in G$, $g \neq \operatorname{Id}$, then $g$ belongs to the stabilizer of $Q$, and by the smallness assumption, the stabilizer is not a reflection group. Hence, by Lemma \ref{reflectionsingular}, its image point is singular.
\end{proof}

We now want to understand the singularities of $Z_\rho$ in more detail and in particular, how the singularities of $Z_\rho$ are related to the singularities of $X$. Let $p:Z_\rho \rightarrow X$ denote the projection.

\begin{Corollary}
\label{moduleschemesingular}
Let $R' \in {\mathbb A}^d \times {\mathbb A}^m$. Then the point $p(R')=R \in Z_\rho$ is singular if and only if the image of the stabilizer group $(\beta \times \rho)( \Stab ( R') ) \subseteq (\beta \times \rho)(G) \subseteq \GLG_d(K) \times \GLG_m(K) $ is not a reflection group.	
\end{Corollary}
\begin{proof}
This is a special case of Lemma \ref{reflectionsingular}.
\end{proof}

If $\rho$ is the trivial action, then $Z_\rho= X \times {\mathbb A}^m$ and then a point $R \in Z_\rho$ is singular if and only if its base point in $X$ is singular. The examples described in Lemma \ref{cyclicone} and, in fact, by Lemma \ref{singularityexistence}, for every situation where $\beta(G)$ is a reflection group and $\rho$ is not trivial, show that the singular locus of $Z_\rho$ is not always mapped into the singular locus of $X$ (this also shows that Corollary \ref{moduleschemesingularzerosection} below cannot be reversed). But also the regular locus of $Z_\rho$ is not always mapped to the regular locus of $X$, as the following example shows (we will see later in Theorem \ref{moduleschemefundamentalgroup} that the morphism $ Z_{ \rho \operatorname{reg} } \setminus p^{-1} (X_{\operatorname{reg} })  \rightarrow X_{\operatorname{reg} } $ is relevant).

\begin{Example}
In Example \ref{toric11.1}, the quotient scheme $X$ has an isolated singularity in $[0]$. The points above $P \neq [0]$ in $Z_\rho $ are regular. Let $R \in Z_\rho$ be a point above $[0]$. If $R$ lies on the zero section, then $R$ is singular by Corollary \ref{moduleschemesingular}, because the stabilizer of $R'$, a point above $R$ in ${\mathbb A}^2 \times {\mathbb A}$, is $G \cong (\beta \times \rho)(G)$, and this is not a reflection group. If $R$ does not lie on the zero section, then $R$ is regular, because the stabilizer is now trivial (these properties also follow from properties of the Veronese algebra).	
\end{Example}

\begin{Example}
\label{cyclic4reflectionsingularities}
We locate the singularities in Example \ref{cyclic4reflection}. The quotient scheme of the basic action is an $A^1$-singularity. For the one-dimensional representation $\rho$, the singular locus of $Z_\rho$ is the line $V(X_2,W)$ by Corollary \ref{moduleschemesingular}: For a point with $x_2 \neq 0$ or $w \neq 0$, the stabilizer is trivial, and for a point $(x_1,0,0)$ with $x_1 \neq 0$, the stabilizer is the subgroup $\Z/(2)$, whose generator is a reflection via $\beta$, but the corresponding element in $\GLG_3(K)$ is not a reflection.	
\end{Example}

\begin{Corollary}
\label{moduleschemesingularzerosection}
Let $ R \in Z_\rho$ be a point on the zero section, and suppose that $P=p(R) \in X$ is singular. Then, $R$ is also singular.
\end{Corollary}
\begin{proof}
Let $ R'=(Q,0) \in {\mathbb A}^d \times \{0\} \subseteq {\mathbb A}^d \times {\mathbb A}^m $ be a point above $R$. Then $\beta( \Stab ( Q)) $ is not a reflection group within $\beta(G) \subseteq \GLG_d(K) $ by Lemma \ref{reflectionsingular}. We have $ \Stab  (R') =  \Stab (Q)  $ by Lemma \ref{stabilizercompare} (3). Therefore, $R$ is also singular by Lemma \ref{reflectionsingular}.
\end{proof}

\begin{Theorem}
\label{fibersingularonepoint}
Let $ Q \in {\mathbb A}^d $ be a point with image point $P \in X$.	The following properties for the quotient scheme $Z_\rho \rightarrow X$ are equivalent.

\begin{enumerate}

\item 
The fiber of $Z_\rho \rightarrow X$ over $P$ contains a singular point (of $Z_\rho$).	
	
\item 
The point $(P,0)$ (the image of $P$ under the zero section) is singular.

\item
The image $(\beta \times \rho)(\Stab(Q))$ inside $\GLG_d(K) \times \GLG_m(K)$ is not a reflection group. 
\end{enumerate}
	
\end{Theorem}
\begin{proof}
Let $R$ be a singular point of the fiber and let $R'=(Q,v)$ be a point in $ {\mathbb A}^d \times {\mathbb A^m} $ above $R$. We have $\Stab (Q,v) \subseteq \Stab (Q,0)$. By Corollary \ref{moduleschemesingular}, $\Stab (Q,v)$ is not a reflection group inside $\GLG_d(K) \times \GLG_m(K)$. Therefore, the same is true for $ \Stab (Q,0) $ by Corollary \ref{reflectiongroupsubgroup}. Hence, again by Corollary \ref{moduleschemesingular}, $(P,0)$ is also singular. This gives the equivalence between (1) and (2). The equivalence between (2) and (3) follows from this since $ \Stab (Q)  $ equals  $ \Stab (Q,0) $ by Lemma \ref{stabilizercompare} (3).
\end{proof}

\begin{Corollary}
\label{fiberramifiedsingularpoint}
Let $ Q \in  {\mathbb A}^d $ be a point with image point $P \in X$ and let
$Z_\rho \rightarrow X$ be a quotient scheme. If the fiber of $Z_\rho$ above $P$ is nonreduced, then it contains a singular point of $Z_\rho$.
\end{Corollary}
\begin{proof}
By Theorem \ref{nonreducedfiber}, the nonreducedness of the fiber is equivalent with the property that the restriction of $\rho$ to the stabilizer group $\Stab (Q)$ is not trivial. Hence, there exists $ g \in \Stab (Q)$, $\rho(g) \neq \operatorname{Id_{K^m} } $. All reflections of $(\beta \times \rho) (G)$ live in $ \GLG_d(K) \times \{ \operatorname{Id_{K^m} } \} $ by Corollary \ref{subgroupreflection}. Therefore, $(\beta \times \rho)(\Stab(Q))$ is not a reflection group. By Theorem \ref{fibersingularonepoint}, the fiber contains a singularity.	
\end{proof}

\begin{Corollary}
\label{reflectiongroupnotallreflexive}
Let $\beta $ be a linear faithful  representation of a finite group $G $ such that $\beta(G)$ is generated by reflections, and let $X={\mathbb A}^d/G$ be the quotient. Let $\rho$ be a nontrivial linear representation of $G$, and let $Z_\rho$ be the corresponding fiberflat bundle. Then, $Z_\rho \rightarrow X$ has singularities above a subset of codimension one. In particular, $Z_\rho$ is not a reflexive bundle.
\end{Corollary}
\begin{proof}
Since $\beta(G)$ is generated by reflections and $\rho$ is not trivial, there exists a $g\in G$ such that $\beta(g)$ is a reflection and $\rho(g) \neq \operatorname{Id}$. Thus, $\beta(g) \times \rho(g)$ is not a reflection by Lemma \ref{elementreflection}. Let $H$ denote the mirror hyperplane for $g$, we may assume that $g$ generates the group of reflections with this hyperplane. For every $Q\in H$, not in any other mirror hyperplane, the stabilizer of $Q$ is the group generated by $g$; hence, $(\beta \times \rho)(\Stab(Q))$ is not a reflection group and so, by Theorem \ref{fibersingularonepoint}, the fiber of $Z_\rho$ above the image point of $Q$ in $X$ contains singularities. The image of $H$ is a subset of codimension one in $X$. Therefore, $Z_\rho$ is not locally free in codimension one.
\end{proof}

\begin{Theorem}
\label{fibersingular}
Let $ Q \in  {\mathbb A}^d $ be a point with image point $P \in X$, $X= {\mathbb A}^d/\beta$, $\beta$ a linear action. Let $\rho$ be a linear representation of $G$ with the corresponding quotient scheme $Z_\rho$. Then every point of the fiber of $Z_\rho \rightarrow X$ above $P$ is singular if and only if $\Stab (Q) \cap \ker \rho$ is not a reflection group via $\beta$.	
\end{Theorem}
\begin{proof}
Let $R'=(Q,v) \in {\mathbb A}^d \times { \mathbb A}^m$ with image point $R \in Z_\rho$ above $P$. By	Corollary \ref{moduleschemesingular}, the property that every point of the fiber over $P$ is singular means that for all $v \in {\mathbb A}^m$, the image of $ \Stab (Q) \cap \Stab (v)$ (using Lemma \ref{stabilizercompare} (1)) in $ \GLG_d(K) \times \GLG_m(K) $ is not a reflection group. The stabilizer of a sufficiently generic point $v$ is $\ker \rho$. Hence, the image of $ \Stab (Q) \cap \ker \rho $ in $ \GLG_d(K) \times \GLG_m(K) $ is not a reflection group. Since the image in $ \GLG_m(K)$ is the identity, it follows that the image of  $ \Stab (Q) \cap \ker \rho $ in $ \GLG_d(K)$ is not a reflection group. Reading this argument backwards shows that the generic point of the fiber is singular. Since the singular locus is closed, this implies that every point of the fiber is singular.
\end{proof}

\begin{Corollary}
Let the symmetric group $S_d$ act naturally on ${\mathbb A}^d$ and let $\rho$ denote the sign representation with the corresponding product action on ${\mathbb A}^d \times {\mathbb A}^1 $ (as in Lemma \ref{symmetricdeterminant}). Let $Z_\rho \rightarrow {\mathbb A}^d/S_d \cong {\mathbb A}^d$ be the corresponding module scheme. Let $Q \in {\mathbb A}^d$ with image point $P$. Then the following hold.

\begin{enumerate}
	
\item
If $\Stab (Q)$ is trivial, then the fiber of $Z_\rho$ above $P$ contains no singularity (of $Z_\rho$).

\item
If $\Stab (Q)$ is nontrivial, but contains only odd permutations, then the fiber of $Z_\rho$ above $P$ has exactly one singularity (which lies on the zero section).

\item
If $\Stab (Q) $ contains a nontrivial even permutation, then every point of the fiber of $Z_\rho$ above $P$ is a singularity.

\end{enumerate}
\end{Corollary}

\begin{proof} 
(1) The stabilizer group of $Q$ is trivial if and only if all coordinates of $Q$ are different. This condition describes the free locus of the action and so $Z_\rho$ is a vector bundle above the image of this locus by Corollary \ref{modulealgebralinearisomorphism} (this follows also directly from Lemma \ref{symmetricdeterminant}). (2) Under the given condition, it follows from Theorem \ref{fibersingularonepoint} that the point in the fiber on the zero section is singular. It follows from (3) that this is the only singular point. (3) follows from Theorem \ref{fibersingular}, since $\ker \rho$ is the alternating group $A_d$ and because the alternating group is small.
\end{proof}

\section{The fundamental group of quotient schemes}
\label{fundamentalgroupsection}

Let $K=\C$, all schemes are of finite type and we consider them as complex analytic spaces, without changing the notation. We want to understand the fundamental group of the regular locus of the quotient schemes $Z_\rho$ and its relation to the acting group. Note that the fundamental group of $Z_\rho$ is trivial, as the spectrum of a positively graded $\C$-algebra is contractable. For background on algebraic topology, see \cite{hatcher}.

\begin{Lemma}
\label{fundamentalgroup}
Let $K=\C$. Let $\beta $ be a linear faithful representation of a finite group $G$ on ${\mathbb A}^d $, and let $X={\mathbb A}^d/G$ be the quotient. Let $G_{\beta \operatorname{refl} } \subseteq G$ be the subgroup generated by reflections, and let $ \overline{G} = G/G_{\beta \operatorname{refl} } $. Then, the fundamental group of $X_{\operatorname{reg} } $ is $\overline{G}$.	
\end{Lemma}
\begin{proof}
The quotient $={\mathbb A}^d/  G_{\beta \operatorname{refl} }$ is isomorphic to the affine space, $ \overline{G}$ is acting on it and the quotient of this action is also $X$, as mentioned in the beginning of Section \ref{normalsubgroupsection}. So we can assume that the action $\beta$ is small.	Let $V \subseteq {\mathbb A}^d $ be the free locus and let $U=V/G$ be its image. The complement ${\mathbb A}^d \setminus V $ has codimension $\geq 2$ since the group $G$ contains no reflection, and hence, $ X \setminus U $ also has codimension $\geq 2$. The mapping $V \rightarrow U$ is a covering space, $U$ is a complex manifold, $V$ is simply connected, as ${\mathbb A}^d$  is and since the fundamental group does not change by removing closed real submanifolds of real codimension $\geq 3$, by \cite[Satz 4.B.2]{storchwiebe4} (applied to a smooth stratification of ${\mathbb A}^d \setminus V$). Therefore, $G$ is isomorphic to the fundamental group of $U$. By Corollary \ref{smallsingularities}, $U$ is the regular locus of $X$.
\end{proof}

\begin{Remark}
\label{fundamentalidentification}
In the correspondence described in Lemma \ref{fundamentalgroup}, a closed continuous path $\gamma$ starting and ending in a point $P \in X_{ \operatorname{reg}}$ representing an element in the fundamental group $\pi_1(X_{\operatorname{reg} })$ corresponds in the following way to an element in $G$. For a given point $Q \in V$ above $P$, there exists a unique lifting $\tilde{\gamma}$ of the path starting in $Q$ and ending in a point $Q'$ over $P$. Then, there exists a unique $g \in G$ with $g(Q)=Q'$, and $g$ corresponds to $\gamma$.
\end{Remark}

\begin{Remark}
\label{reflectionsfundamental}
If $G$ is generated by reflections, then $\overline{G}=0$ and $X_{\operatorname{reg} }={\mathbb A}^d$ with a trivial fundamental group.
\end{Remark}

\begin{Remark}
Lemma \ref{fundamentalgroup} is related to a result of M.A. Armstrong, which says (see  \cite{armstrong}), adopted to our situation, the following: Let $G$ act on ${\mathbb A}^d$ as before and let $V \subseteq {\mathbb A}^d$ be an invariant open subset such that the complement ${\mathbb A}^d \setminus V$  has codimension $\geq 2$. Then, the fundamental group of $V/G$ is $G/I$, where $I \subseteq G$ is the subgroup generated by elements with at least some fixed point.	
\end{Remark}

We study now the fundamental group of the regular locus of $Z_\rho$ in relation to the fundamental group of the regular locus of $X$.

\begin{Example}
\label{cycliconefundamentalgroup}	
We look at the situation described in Lemma \ref{cyclicone}, see also Example \ref{cycliconereflections}. Via $\beta$, $G$ acts as a reflection group, and for the quotient, we have $\pi_1({\mathbb A}^1)=0$ in accordance with Remark \ref{reflectionsfundamental}. Let $n$ be the order of $\ell$ in $\Z/(k)$. Then, the fundamental group of the regular locus of $Z_\ell$ is $\Z/(n)$ by Lemma \ref{fundamentalgroup}. If $\ell$ is coprime to $k$, then the fundamental group is $\Z/(k)$. In this case, $\rho_\ell$ is faithful, so this follows also from Corollary \ref{faithfulfundamental} below. 
\end{Example}

\begin{Example}
\label{kleinactionfundamental}
For Example \ref{kleinaction}, the quotient scheme is ${\mathbb A}^2$ with trivial fundamental group. For the first representation, the subgroup generated by the reflections is $\Z/(2)$ (the nontrivial reflection is given by the matrix with diagonal entries $1,-1,1$), so the fundamental group of the regular locus of $Z_\rho$ is $\Z/(2)$ by Lemma \ref{fundamentalgroup}. For the second representation, there is no reflection, and the fundamental group of the regular locus is $\Z/(2) \times \Z/(2)$. 
\end{Example}

\begin{Corollary}
\label{faithfulfundamental}
Let $K=\C$. Let $G$ be a finite group, and let $\beta$ be a faithful $\C$-linear representation of $G$ of dimension $d$ with quotient scheme $\Spec \C[X_1, \ldots, X_d]^\beta $. Then, for every faithful $\C$-linear representation $\rho$ of $G$, the fundamental group of the regular locus of the quotient scheme $ Z_\rho $ is isomorphic to $G$.
\end{Corollary}
\begin{proof}	
By Lemma \ref{faithfulsmall}, the group $(\beta \times \rho)(G)$ is small. Hence the claim follows from Lemma \ref{fundamentalgroup}.
\end{proof}

\begin{Corollary}
\label{smallfundamental}
Let $K=\C$. Let $G$ be a finite group, and let $\beta$ be a small faithful $\C$-linear representation of $G$ of dimension $d$ with quotient scheme $\Spec \C[X_1, \ldots, X_d]^\beta $. Then, for every $\C$-linear representation $\rho$ of $G$, the fundamental group of the regular locus of the quotient scheme $ Z_\rho $ is isomorphic to $G$.
\end{Corollary}
\begin{proof}	
By Corollary \ref{smallsmall}, the group $(\beta \times \rho)(G)$ is small. Hence, the claim follows from Lemma \ref{fundamentalgroup}.
\end{proof}

\begin{Lemma}
\label{fundamentalregularremove}
Let $K=\C$. Let $G$ be a finite group, and let $\beta$ be a faithful $\C$-linear representation of $G$ and $\rho$ be another linear representation with quotient scheme $p:Z_\rho \rightarrow {\mathbb A}^d/\beta =X$. Then, the following holds.
\begin{enumerate}

\item 
We have $ \pi_1( Z_{\rho \operatorname{reg} }  \cap p^{-1}( X_{ \operatorname{reg} } )) = \pi_1 ( Z_{\rho \operatorname{reg} }	)$.

\item 
In the case of a small basic action $\beta$, the open subset $  Z_{\rho \operatorname{reg} }  \cap p^{-1}( X_{ \operatorname{reg} } )$ coincides with the vector bundle $Z_\rho|_U $, where $U \subseteq X$ is the image of the fixed point free locus, and its fundamental group is $G$.
\end{enumerate}
\end{Lemma}

\begin{proof}
(1). The singular locus of $X$ has codimension $\geq 2$, as $X$ is normal, and since the fibers of $Z_\rho$ have the same dimension as noted in Remark \ref{fiberdimension}, $p^{-1} ( X_{ \operatorname{reg} } ) $ also has codimension at least $2$. Hence, $ \pi_1  (Z_{ \rho \operatorname{reg} } ) \cong \pi_1 ( Z_{  \rho \operatorname{reg} } \cap p^{-1}( X_{ \operatorname{reg} } ) ) $ by the argument applied in the proof of Lemma \ref{fundamentalgroup}.

(2). In the case of a small basic, $U$ is the regular locus of $X$ by Corollary \ref{smallsingularities}.  $ Z_\rho |_U$ is a vector bundle by Corollary \ref{modulealgebralinearisomorphism}, hence $ Z_\rho)_U = Z_{\rho \operatorname{reg} }  \cap p^{-1}( X_{ \operatorname{reg} } )$. The claim about the fundamental group follows from Corollary \ref{smallfundamental}.
\end{proof}

Recall from Section \ref{reflectionssection} the definition of $G_{\beta \operatorname{refl} }$, and in particular the relation between $G_{\beta \operatorname{refl} }$ and $  G_{\beta \times \rho \operatorname{refl} } $ (Lemma \ref{reflectionsubgrouprelation}).

\begin{Theorem}
\label{moduleschemefundamentalgroup}
Let $K=\C$. Let $G$ be a finite group, let $\beta$ be a faithful $\C$-linear representation of $G$ of dimension $d$ with quotient scheme $X $, and let $\rho$ be another linear representation of $G$ with corresponding quotient scheme $p:Z_\rho \rightarrow X$. Let
$ G_{\beta \times \rho \operatorname{refl} }  \subseteq G_{\beta \operatorname{refl} } \subseteq G$ be the subgroups generated by the reflections for these representations. Then, the natural group homomorphism 
\[  \pi_1  (Z_{  \rho \operatorname{reg} }  ) \cong     \pi_1 ( Z_{  \rho \operatorname{reg} }  \cap p^{-1}( X_{ \operatorname{reg} }  )    )   \longrightarrow \pi_1(X_{ \operatorname{reg} }  ) \]
coincides with the surjection
$G /G_{\beta \times \rho \operatorname{refl} }   \rightarrow G/   G_{\beta   \operatorname{refl} }  $.	
\end{Theorem}
\begin{proof}	
By Lemma \ref{fundamentalgroup}, the fundamental groups of the regular loci coincide with the residue class groups modulo the reflection groups. The isomorphism on the left was noted in Lemma \ref{fundamentalregularremove} (1).
	
We mod out 	$G /G_{\beta \times \rho \operatorname{refl} }$ in the situation, so we may assume that the product representation $\beta \times \rho$ is small. Set $H=   G_{\beta   \operatorname{refl} } $. We have the commutative diagram
\[ \begin{matrix} & & {\mathbb A}^d   & \longleftarrow &   {\mathbb A}^d  \times {\mathbb A}^m & \supseteq & \tilde{V} & \supseteq & \hat{V} \\ & & \downarrow & & \downarrow & &\downarrow & &\downarrow\\ V & \subseteq &  {\mathbb A}^d/H \cong   {\mathbb A}^d    & \longleftarrow &   ( {\mathbb A}^d  \times {\mathbb A}^m) / H   & \supseteq & \tilde{V}/H & \supseteq & \hat{V}/H \\ \downarrow & & \downarrow & & \downarrow & &\downarrow & &\downarrow \\U  & \subseteq &  X & \longleftarrow & Z_\rho & \supseteq &  \tilde{V}/G& \supseteq & \hat{V}/G \, . \end{matrix} \]
On the left, $V$ is the fixed point free locus of the action of the small group $G/H$ on ${\mathbb A}^d$, $V \rightarrow U$ is a covering space, $V$ is simply connected and $U=X_{\operatorname{reg} }$. On the right, $\tilde{V}$ is the fixed point free locus of the product action of $G$ on $  {\mathbb A}^d  \times {\mathbb A}^m $, whose complement has codimension $\geq 2$ since $\beta \times \rho$ is small. Hence, $\tilde{V}$ has trivial fundamental group. On the right, we have covering spaces, and $\tilde{V}/G = Z_{\rho \operatorname{reg} }$. On the very right, $\hat{V}$ is such that $ \hat{V}/G = Z_{ \rho \operatorname{reg} } \cap p^{-1}( X_{ \operatorname{reg} } ) $, which does not change the fundamental groups. The horizontal morphisms map $ \hat{V}/G$ to $U$ and  $ \hat{V}/H$ to $V$.	
	
We fix a point $R \in  Z_{\rho \operatorname{reg} } \cap p^{-1}( X_{ \operatorname{reg} } ) = \hat{V}/G $ with image point $P \in U$ and we fix a point $T \in \hat{V}$ above $R$. An element $\gamma \in \pi_1 ( Z_{\rho \operatorname{reg} }  \cap p^{-1}( X_{ \operatorname{reg} }  ) )$ with starting point $R$ corresponds (see Remark \ref{fundamentalidentification}) to a lifting $\hat{\gamma}$ in $\hat{V}$ starting in $T$ and ending in $T'$, and $\gamma$ corresponds to the unique element $g \in G$ with $g(T) =T'$. Let $S$ and $S'$ be the image points of $T$ and $T'$ in $\hat{V}/H$. Then, we have the relation $[g] (S)=S'$ for $[g] \in G/H$. Let $Q$ and $Q'$ be the image points of $S$ and $S'$ in $V$. There, we also have the relation $[g] (Q)=Q'$. The corresponding relations hold for the various images of the path $\hat{\gamma}$ in $\hat{V}/H$, in $V$ and in $U$, i.e., $S'$ and $Q'$ are the endpoints of the liftings. Therefore, the image of $\gamma$ in $\pi_1(U)$ corresponds to $[g]$.
\end{proof}

\begin{Example}
\label{cyclic4reflectionfundamental}
We continue with Example \ref{cyclic4reflectionsingularities}. The fundamental group of the punctured spectrum of the quotient scheme $X={\mathbb A}^2/G \cong{\mathbb A}^2/ (G/H) $ is $G/H \cong \Z/(2) $ according to Lemma \ref{fundamentalgroup}. The singular locus of $Z_\rho$ is the line $V(X_2,W)$ mapping to a curve in $X$, and its complement $ Z_{  \rho \operatorname{reg} } $ has fundamental group $G$
according to Corollary \ref{faithfulfundamental}. The set $ Z_{  \rho \operatorname{reg} }  \cap p^{-1}( X_{ \operatorname{reg} }  ) $
is the complement of two curves meeting in a point, its fundamental group is also $G$, and the natural mapping between the fundamental groups coincides with $G \rightarrow G/H$ by Theorem \ref{moduleschemefundamentalgroup}.
\end{Example}

\begin{Lemma}
\label{smallbundleetale}
Let $K=\C$. Let $G$ be a finite group, and let $\beta$ be a faithful $\C$-linear representation of $G$ of dimension $d$ with quotient scheme $X $, where $\varphi:{\mathbb A}^d \rightarrow X$ is the quotient morphism. Let $\rho$ be another linear representation of $G$ with the corresponding quotient scheme $p:Z_\rho \rightarrow X$. Suppose that the product representation $\beta \times \rho$ is small. Let $U \subseteq X_{\operatorname{reg} }$ be an open subset
where $Z_\rho$ is a vector bundle. Then $ \varphi^{-1}(U) \rightarrow U$ is \'etale.
\end{Lemma}
\begin{proof}
The mapping $\varphi$ factors as
\[  {\mathbb A}^d \stackrel{\varphi_1}{\longrightarrow} {\mathbb A}^d/G_{\beta \operatorname{refl} } \cong {\mathbb A}^d  \stackrel{\varphi_2}{\longrightarrow} X \, , \]
we show the \'etale property over $ U $ for both morphisms. The map $ \varphi_2 $ is \'etale over the regular locus of $ X $ which contains all points of codimension one according to Corollary \ref{smallsingularities}. Note that $ \varphi_2^{-1}(Z_\rho) $ is a vector bundle over $ \varphi_2^{-1}(U) $. For $\varphi_1$ and the acting group $ G_{\beta \operatorname{refl} } $, due to our assumption, the image group $ (\beta \times \rho)  (G_{\beta \operatorname{refl} }) $ contains no reflection.
Let $Q \in \bigcup_i H_i \subseteq {\mathbb A}^d$ be a point in the union of the mirror hyperplanes, hence the stabilizer group $ \Stab ( Q) $ is not trivial. Then, $ (\beta \times \rho)  ( \Stab (Q) ) $ is not a reflection group, and therefore, by Theorem \ref{fibersingularonepoint}, the fiber of $ \varphi_2^{-1}(Z_\rho) $ over $ \varphi_1(Q) $ contains a singularity. Thus,  $ \varphi_1(Q) \notin \varphi_2^{-1}(U) $ since over $\varphi_2^{-1}(U)$, we have a vector bundle. Hence $ \varphi_2^{-1}(U) $ does not meet the images of the mirror hyperplanes, and so $\varphi_1$ is \'etale on $ \varphi^{-1} (U)$.    
\end{proof}

\section{Pull-back}
\label{pullbacksection}

We want to show that we can reconstruct the linear action of $G$ on ${\mathbb A}^m$ from the quotient module scheme $Z_\rho \rightarrow X$, knowing the action of $G$ on $Y$.

\begin{definition}
Let $G$ be a finite group acting on a scheme $Y$ with quotient scheme $X$. For any scheme $Z $ over $X$, the \emph{pull-back} of $Z$ is the $Y$-scheme $Y \times_X Z$ together with the $G$-action on the first component.
\end{definition}

In particular, for the quotient module scheme $Z_\rho$ up to normalization, the pull-back is $Y \times_X Z_\rho $ together with the action of the group $G$ on the first component.

\begin{Lemma}
\label{pullbacklemma}
Let the finite group $G$ act faithfully on an affine $K$-scheme $Y=\Spec S$ by $K$-automorphisms with quotient scheme $X=Y/G=\Spec R	$, and let $\rho$ be a linear representation of $G$ with the corresponding action on $ Y \times {\mathbb A}^m $ and with quotient scheme $ Z_\rho \rightarrow X $. Then there is a natural $G$-equivariant finite surjective morphism
\[ Y \times {\mathbb A}^m \longrightarrow  Y \times_X  Z_\rho  \, .\]
\end{Lemma}
\begin{proof}
The natural projection $  Y \times {\mathbb A}^m  \rightarrow Y  $ and the quotient morphism
$  Y \times {\mathbb A}^m  \rightarrow Z_\rho  $ are compatible with the morphisms to $X$. By the universal property of the product, this gives the morphism. A point $(y,v)$ is sent to $ (y, [ (y,v) ])$. For $g \in G$, the element $g(y,v)=(g(y),g(v))$ is sent to $ (g(y), [ ( g(y),g(v)) ])= (g(y), [(y,v)])$; hence, the morphism is $G$-equivariant.
	
Surjectivity. The map sends $ (y,v) $ to $ (y, [(y,v)]) $. If $ (y, [ (\tilde{y} ,v)]) $ is given, then $y$ and $ \tilde{y} $ are $\beta$-conjugate, say $ \tilde{y} =gy $ with some $ g \in G $. Then, the element is the same as $( y, [(   y , g^{-1} (v)])  $, and  $(y, g^{-1} ( v )) $ is a preimage.
		
To see finite, we consider the corresponding ring homomorphism
\[ S \tensor_R B^\rho \rightarrow S[W_1, \ldots , W_m] \, . \]	The variables $W_j$  fulfil an integral equation over $K[W_1, \ldots, W_m]^\rho$, hence also over $B^\rho$.
\end{proof}

\begin{Lemma}
\label{pullbackfree}
Let the finite group $G$ act freely on a $K$-scheme $Y$ by $K$-automorphisms with quotient scheme $X=Y/G$, and let $\rho$ be a linear representation of $G$ with the corresponding action on $ Y \times {\mathbb A}^m $ and with quotient scheme $ Z_\rho \rightarrow X $. Then, the natural $G$-equivariant morphism
\[ Y \times {\mathbb A}^m \longrightarrow  Y \times_X  Z_\rho  \, \]
from Lemma \ref{pullbacklemma} is an isomorphism. 
\end{Lemma}
\begin{proof}
This is a special case of faithful flat descent.
\end{proof}

\begin{Theorem}
\label{pullback}
Let the finite group $G$ act faithfully on a normal affine $K$-scheme $Y$ by $K$-automorphisms with quotient scheme $X=Y/G$, and let $\rho$ be a linear representation of $G$ with the corresponding action on $ Y \times {\mathbb A}^m $ and with quotient scheme $ Z_\rho \rightarrow X $. Then, the natural $G$-equivariant finite surjective morphism
\[ Y \times {\mathbb A}^m \longrightarrow  Y \times_X  Z_\rho  \, \]
from Lemma \ref{pullbacklemma} is the normalization. It is an isomorphism over the points $y$ where the action is free. 
\end{Theorem}
\begin{proof}
By Lemma \ref{pullbacklemma}, the morphism is finite and surjective. By the surjectivity of the morphism, it follows that $ Y \times_X  Z_\rho $ is irreducible, hence its reduction is integral. Let $V \subseteq Y$ be the open subset where the action is free with image set $U$. Then, by Lemma \ref{pullbackfree}, the induced morphism $V \times {\mathbb A}^m \longrightarrow ( Y \times_X  Z_\rho )|_V \cong V \times_U (  Z_\rho)|_U $ is an isomorphism.	By Lemma \ref{genericfree}, $V$ is not empty; hence, the morphism is a birational finite morphism between a normal scheme and an integral scheme, so it is the normalization.	
\end{proof}

\begin{Remark}
If the basic action is not faithful, then Theorem \ref{pullback} does not hold. If, looking at the extreme case, $G$ acts trivially on  a point with quotient scheme the point itself, then $Z_\rho ={\mathbb A}^m/\rho $ is the quotient scheme over the base point, but pulling it back does not give ${\mathbb A}^m$ back. The natural morphism ${\mathbb A}^m \rightarrow {\mathbb A}^m/\rho$ will not be birational.
\end{Remark}

\begin{Remark}
On the algebra level, the relevant ring homomorphisms are	
\[ \begin{matrix} & & S[ W_1 ,\ldots, W_m] \\	& & \uparrow \\
S & \rightarrow & S \tensor_{S^\beta}  S[  W_1, \ldots , W_m]^{\beta \times \rho}   \\ 	\uparrow & & \uparrow  \\	
 S^\beta	& \rightarrow &   S[ W_1 ,\ldots , W_m]^{\beta \times \rho} \, .  \end{matrix}  \]
\end{Remark}

\begin{example}
\label{cycliconespecialpullback}
We consider Theorem \ref{pullback} in Example \ref{cycliconespecial}, the relevant rings are $S=  K[X]$, $R=K[X^k]=K[A]$, and $\rho$ acts on $K[X,W]$ by weight $(1,-1)$, yielding the invariant algebra $ B^\rho = K[A,B,C]/(AB-C^k) $. The tensor product is
\[ K[X]   \tensor_{K[A]}  K[A,B,C]/(AB-C^k)   \cong  K[X,B,C]/(X^kB-C^k)    \] with the natural ring homomorphism 
\[  K[X,B,C]/(X^kB-C^k)   \longrightarrow K[X,W], \, B \mapsto W^k,\, C \mapsto XW \, .   \]
This is the normalization, as $W=C/X$ fulfils an integral equation. The points of the form $(0, \zeta w$) for $ \zeta$, a $k$th root of unity are mapped to the same point; hence, the morphism on the spectra is not injective everywhere.
\end{example}

\begin{example}
\label{cycliconespecialpullbacknomatch}
We continue with Example \ref{cycliconespecialpullback}, but now we pull-back the  invariant algebra $K[A] \subseteq K[A,B,C]/(AB-C^n)$ along the  ring homomorphism $K[A] \subseteq K[X]/(X^k - A)$. So we look at the same basic action (determined by $k$) as before, but now we look at a module scheme up to normalization on $\Spec K[A]$, which comes from a different basic action. The tensor product is
\[ K[X]  \tensor_{K[A]}  K[A,B,C]/(AB-C^n)  \cong  K[X,B,C]/(X^k B-C^n
)   \, . \]
This ring is also not normal, but also its normalization is not nice over $K[X]$.	For $k=3$ and $n=2$, we have the integral equation $XB= \frac{C^2}{X^2} =D^2$, and the ring $K[X,B,C,D]/( D^2X^2-C^2 ,XB-D^2)$ is its normalization. As a $K[X]$-algebra, the fiber ring over $X=0$ is still not reduced.
\end{example}

In the situation of Theorem \ref{pullback}, it is indeed necessary to use the reduction, as the following example shows.

\begin{example}
\label{pullbacknilpotent}
We look at Example \ref{toric11.1}, the invariant ring is $K[X^2,Y^2,XY] $, describing the $A_1$-singularity, and the invariant algebra is
\[T= K[X^2,Y^2,XY, W^2,XW,YW] = K[A,B,C,D,E,F]/( \text{relations}) \, , \]
which is the Veronese algebra in $3$ variables of degree $2$. The tensor product of this algebra along $K[X^2,Y^2,XY] \subseteq K[X,Y]$ is
\[ K[X,Y][D,E,F]/( E^2-X^2D, F^2-Y^2D, EF-DXY, XYE-X^2F, XYF- Y^2E )  \, , \]
which describes the pull-back $ {\mathbb A}^2 \times_X  Z_\rho$ in Theorem \ref{pullback}. The element $YE-XF$ is not $0$ in this algebra, but its square is
\[  (YE-XF)^2=  Y^2E^2 +X^2F^2  -2 YEXF= Y^2E^2 +X^2F^2  -Y^2E^2 -X^2F^2 =0 \, .\]
The reduction of the pull-back algebra is
\[ K[X,Y][D,E,F]/(E^2-X^2D, F^2-Y^2D, EF-DXY, YE-XF )  \, , \]
its normalization is $K[X,Y,W]$ with $W= E/X=F/Y $.
\end{example}

\begin{Remark}
If the group $G$ acts on a normal affine scheme $Y$ such that the action is free on an open invariant subset $V$ whose complement has codimension $\geq 2$ (e.g., if $G$ acts in a small way on ${\mathbb A}^d$) and $U \subseteq X$ denotes its image, then one can reconstruct the action on $Y \times {\mathbb A}^m$ from the restriction $(Z_\rho)|_U$, which is a vector bundle by Corollary \ref{modulealgebraisomorphism}. This is a special case of Lemma \ref{pullbackfree} that one gets back $ V \times {\mathbb A}^m $ together with the action. This action can be uniquely extended to an action on $ Y \times {\mathbb A}^m $ due to the codimension condition.
\end{Remark}

Remark \ref{notreducednoreconstruction} has already shown that one cannot reconstruct the original linear action from the quotient without any condition.

\begin{example}
\label{kleinactionhyperplanepullback}
We look at the pull-back in the two cases of Example \ref{kleinactionhyperplane}. The ring homomorphism from Lemma \ref{pullbacklemma} in the first case is
\[ K[X,Y,C,D]/(XY,  D^2-X^2C )  \longrightarrow  K[X,Y]/(XY) [W] \, ,  C \mapsto W^2, D \mapsto XW \, ,\]
and in the second case it is
\[ K[X,Y,C,D]/(XY,  D^2   )  \longrightarrow   K[X,Y]/(XY) [W] \, ,  C \mapsto W^2, D \mapsto 0 \, .\]
It seems to be quite difficult to describe an algebraic construction that yields from the left the right-hand side (even if we remember the group action on the left). The normalization separates the two closed subschemes $V(X)$ and $V(Y)$, so this changes too much.
\end{example}

\section{Correspondence}
\label{correspondencesection}

Let $G$ be a finite group acting on a $K$-scheme $Y$ by $K$-scheme automorphisms and let $\rho$ be a $K$-linear representation of $G$. This gives rise to a $G$-equivariant action of $G$ on the trivial vector bundle $Y \times {\mathbb A}^m \rightarrow Y $, where the action on the second component is linear. This construction may come with a loss of information. On the $Y$-side, it might be possible that two such actions of $G$ on $Y \times {\mathbb A}^m$ are equivalent as $G$-equivariant bundles over $Y$, although the representations are not equivalent, as the following examples show.

\begin{Example}
We consider the action of $ G= \Z/(k)$ on $S=K[X,X^{-1}]$ by $X \mapsto \zeta X$, i.e., the situation of Lemma \ref{cyclicone}, but restricted to the punctured affine line, where the action is free. A representation of $G$ on $K$, where $\zeta$ acts by multiplication with $ \zeta^\ell $, gives the product action on $K[X,X^{-1},W]$. However, this action can be trivialized over $S$, through the $G$-equivariant ring homomorphism $ K[X,X^{-1},W] \rightarrow  K[X,X^{-1},U]$, $W \mapsto X^ \ell U$ (trivial action on $U$).
\end{Example}

\begin{Example}
If $G$ acts as a reflection group on ${\mathbb A}^d$, and $Y = {\mathbb A}^d \setminus \bigcup H_i$, where $H_i$ denotes the mirror hyperplanes, then for any representation $\rho$, the induced action of $G$ on $Y \times  {\mathbb A}^m    $ can be trivialized. The action on $Y$ is free by Corollary \ref{reflectiongroupsubgroup}; hence, we can apply for $Y \rightarrow Y/G$ faithful flat descent. As the quotient schemes $Y \times  {\mathbb A}^m / \beta \times \rho   $ over $Y/G$ are trivial (by Lemma \ref{modulealgebrafree} and since the invariant modules are trivial), it follows that the action is also trivial.
\end{Example}

\begin{Example}
Let $Y=G$ (a discrete scheme) and let $G$ act naturally on $G$. Then, every linear representation on $G \times {\mathbb A}^m$ can be trivialized by the map
\[ G \times {\mathbb A}^m  \longrightarrow  G \times {\mathbb A}^m , (g,v) \mapsto  (g, g^{-1}(v)) \, .\]
Here, $G$ acts on the left in both components, and on the right only on $G$. The morphism is $G$-equivariant, as $(hg,hv)$ is sent to $(hg, (hg)^{-1} (hv))=(hg, g^{-1}(v)) $.
\end{Example}

For a  $K$-point $Q \in Y$ fixed by the action of $G$, there is a natural candidate to reconstruct the linear representation, namely by the induced action of $G$ on the fiber $( Y \times {\mathbb A}^m)_Q \cong {\mathbb A}^m_K = {\mathbb A}^m_Q $.

\begin{Lemma}
\label{manyfixedpoints}
Let $G$ be a finite group, acting faithfully via $\beta$ through $K$-scheme automorphisms on a $K$-scheme $Y$, and let $Q,P \in Y$ be $K$-points that are fixed by the action. Let an action of $G$ on $Y \times  {\mathbb A}^m$ be given, compatible with the basic action and linear in the second component. Suppose that there is a $G$-equivariant morphism of vector bundles  $ \psi: Y \times {\mathbb A}^m \rightarrow  Y \times {\mathbb A}^m  $ over a $G$-equivariant morphism $\varphi:Y \rightarrow Y$ sending $P$ to $Q$. Then, the action of $G$ on the fiber $ {\mathbb A}^m_P $ and the action of $G$ on the fiber $ {\mathbb A}^m_Q $
are equivalent.
\end{Lemma}

\begin{proof}
We have a commutative diagram
\[ \begin{matrix} Y \times {\mathbb A}^m & \stackrel{\psi}{\longrightarrow } &    Y \times {\mathbb A}^m  \\ \downarrow & & \downarrow \\
 Y & \stackrel{\varphi}{\longrightarrow} & Y  \end{matrix}\]
of $G$-equivariant morphisms where $\psi$ is also linear in the second component. The morphism $\psi $ induces a linear isomorphism $ \tilde{\psi}: {\mathbb A}^m_P  \rightarrow {\mathbb A}^m_Q $. We have $ \tilde{\psi} (g(v)) = g( \tilde{\psi} (v))$ for all $g \in G$ and $v \in {\mathbb A}^m_P $ by the equi\-variance of $\psi$. Hence, the representations induced on $  {\mathbb A}^m_P$ and $ {\mathbb A}^m_Q $ are equivalent via $\tilde{\psi}$.
\end{proof}

\begin{Lemma}
\label{correspondencerepresentation}
Let $G$ be a finite group, acting faithfully via $\beta$ through $K$-scheme automorphisms on a $K$-scheme $Y$ with a unique fixed $K-$point. There is a correspondence between
\begin{enumerate}
\item Linear representations of $G$ of dimension $m$.
	
\item Product actions of $G$ on $Y \times {\mathbb A}^m$, linear in the second component.

\end{enumerate}	
\end{Lemma}
\begin{proof}
The linear action of $G$ on $K^m$ extends to a product action of $G$ on $Y \times {\mathbb A}^m$, which is linear in the second component. It is clear that equivalent representations yield equivalent equivariant $G$-actions on $Y \times  {\mathbb A}^m$ (over $Y$). Let $Q$ be the unique fixed point of the action. Then, a linear product action induces an action of $G$ on the fiber $ {\mathbb A}^m_Q $ and on its $K$-points ${\mathbb A}^m_Q (K) \cong K^m $. Suppose that we have two linear product actions and that we have a $G$-equivariant diagram
\[ \begin{matrix}    Y \times {\mathbb A}^m & \stackrel{\psi}{\longrightarrow } &    Y \times {\mathbb A}^m   \\ \downarrow & & \downarrow \\
	Y & \stackrel{\varphi}{\longrightarrow} & Y  \end{matrix}\]
between these actions, where $\psi$ is linear in the second component. Then $\varphi(Q)=Q$, as this is the only fixed point, and the argument of Lemma \ref{manyfixedpoints} shows that the induced representations are the same. The two constructions are inverse to each other.
\end{proof}

By Lemma \ref{manyfixedpoints}, one could also have many fixed points if they are always related by a $G$-automorphism of $Y$.

\begin{Corollary}
\label{correspondence}
Let $G$ be a finite group, acting faithfully via $\beta$ through $K$-algebra automorphisms on a normal affine $K$-scheme $Y$ with a unique fixed $K-$point, and let $X=Y/ \beta $. There is a correspondence between
\begin{enumerate}
\item Linear representations of $G$ of dimension $m$.
	
\item Product actions of $G$ on $Y \times {\mathbb A}^m$, linear in the second component.
	
\item Quotient schemes (module schemes up to normalization) of the form $  (Y \times {\mathbb A}^m) / \beta \times \rho \rightarrow X $.
\end{enumerate}
\end{Corollary}
\begin{proof}
The translation between (1) and (2) is Lemma \ref{correspondencerepresentation}. The correspondence between (2) and (3) comes from Theorem \ref{pullback}.
\end{proof}

\begin{Remark}
Theorem \ref{pullback} and Corollary \ref{correspondence} tell us that we do not lose information by going from $\rho$ to $Z_\rho$. However, it does not say anything about what kind of objects do arise as $Z_\rho$. Note that in the case generated by reflections, when $X$ is itself an affine space, there are many different $Z_\rho$ which arise from different groups $G$ and their representations, as in Example \ref{cycliconespecialpullbacknomatch}.
\end{Remark}

\section{Irreducibility}
\label{irreduciblesection}

An important condition for a linear representation is its irreducibility. In the case of a small action, this property is reflected by the corresponding module of covariants being indecomposable. We want to find a notion of irreducibility for a fiberflat bundle $Z \rightarrow X$. The non-existence of a product representation $Z \neq Z_1 \times_X Z_2$ (or $Z \neq \widetilde{Z_1 \times_X Z_2}$ allowing normalization) within all schemes over $X$ might in the end be the right answer, but does not yield yet a workable definition, as many basic cancellation problems are still open. We work instead with decompositions within reflexive fiberflat bundles. Because of the reflexivity assumptions, the results of this section can only be applied in the case of a small action.

\begin{Definition}
Let $X$ be a normal scheme and let $Z \rightarrow X$ be a normal reflexive fiberflat bundle over $X$. $Z$ is called \emph{irreducible} if it is not possible to write $Z \cong \widetilde{ Z_1 \times_X Z_2}$ with normal reflexive fiberflat bundles $Z_1$ and $Z_2$ of smaller rank (isomorphism of module schemes up to normalization).
\end{Definition}

\begin{Lemma}
\label{irreduciblecharacterize}
Let $X = \Spec R $ be a normal affine scheme of finite type over a field and let $Z \rightarrow X$ be a normal reflexive fiberflat bundle over $X$. Let $U \subseteq X$ denote an open subset containing all points of codimension one such that the restriction $Z|_U$ is a vector bundle. Then, the following are equivalent.
	
\begin{enumerate}
\item
$Z$ is irreducible.

\item
The vector bundle $Z|_U$ is indecomposable.

\item
The reflexive $R-$module corresponding to $Z$ is indecomposable.
\end{enumerate}
\end{Lemma}
\begin{proof}
Suppose that $Z$ is reducible, say $ Z = \widetilde{ Z_1 \times_X Z_2}$ as module schemes up to normalization with normal reflexive fiberflat bundles $Z_1$ and $Z_2$. This induces over $U$ an isomorphism $ Z|_U \cong (Z_1)|_U \times (Z_2)|_U $. Shrinking $U$ to a smaller open subset containing all points of codimension one, we may assume that $(Z_1)|_U$ and $ (Z_2)|_U $ are vector bundles. This implies that we get a decomposition for the corresponding modules, $M=M_1 \oplus M_2$, and so $\tilde{M}|_U \cong \tilde{M_1}|_U \oplus \tilde{M_2}|_U$. This implies that $M_1$ and $M_2$ are already locally free on the original $U$.

Now assume that we have a module decomposition $ M=M_1 \oplus M_2 $. By Lemma \ref{reflexivebundlesplitting}, also $M_1$ and $M_2$ are fiberflat $R$-modules. We set $A_1 = \bigoplus_{k \in \N} ( \Sym^k (M_1))^{**}$ and $A_2 = \bigoplus_{\ell \in \N} ( \Sym^\ell (M_2))^{**}$, and $Z_i = \Spec A_i$ give the fiberflat bundle realization of $M_i$ due to Lemma \ref{reflexivebundle}. By Corollary \ref{fiberflatfactorialclosure}, $Z = \Spec A$, where $ A= \bigoplus_{n \in \N} (\Sym^n (M_1 \oplus M_2))^{**}  $,
 is the normalization of $ Z_1 \times_X Z_2 $. 
\end{proof}

\begin{Lemma}
\label{irreduciblesmall}
Let $\beta $ be a linear faithful small representation of a finite group $G $, and let $X={\mathbb A}^d/G$ be the quotient. Let $\rho$ be a linear representation of $G$ and let $Z_\rho \rightarrow X$ be the corresponding fiberflat bundle. Then the following are equivalent.
	
\begin{enumerate}
\item
The representation $\rho$ is irreducible.
		
\item
The fiberflat bundle $Z_\rho$ is irreducible.
\end{enumerate}
\end{Lemma}
\begin{proof}
If the representation is reducible, say $(V,\rho) =(V_1,\rho_1) \oplus (V_2, \rho_2)$, both representations being nontrivial, then we get by Theorem \ref{normalproductcompatibility} the decomposition $Z_\rho = \widetilde{Z_{\rho_1}   \times_X Z_{\rho_2} }$.

Suppose that the bundle $Z_\rho$ is reducible. Note that, due to the smallness assumption, $Z_\rho$ is reflexive by Lemma \ref{smallreflexive}. Hence, by Lemma \ref{irreduciblecharacterize}, the corresponding module is decomposable. The corresponding result for the invariant modules shows that the representation is reducible.	
\end{proof}

\begin{Corollary}
\label{correspondenceirreducible}
Let $\beta $ be a linear faithful small representation of a finite group $G $, and let $X={\mathbb A}^d/G$ be the quotient. There is a correspondence between
\begin{enumerate}
	
\item
Linear irreducible representations of $G$.

\item
Irreducible quotient schemes (fiberflat bundles) of the form $  (Y \times {\mathbb A}^m) / \beta \times \rho \rightarrow X $.
\end{enumerate}
\end{Corollary}
\begin{proof}
This follows from Corollary \ref{correspondence} and Lemma \ref{irreduciblesmall}.	
\end{proof}

\section{Trivialization}
\label{trivializationsection}

Let $\beta$ be a faithful action of a finite group on a normal scheme $ Y$ with quotient scheme $X$. In Section \ref{pullbacksection}, we have seen that we can reconstruct from a quotient scheme $Z_\rho \rightarrow X$ the action $\rho$ on ${\mathbb A}^m$. Here, we ask the question under what conditions we can reconstruct, given a fiberflat bundle $ Z\rho \rightarrow X$, the group $G$, the scheme $Y$ and the actions of $G$ on $Y$ and on  ${\mathbb A}^m$, which have produced the given bundle. We work over $K=\C$. According to Lemma \ref{fundamentalgroup}, the fundamental group of the regular locus $X_{\operatorname{reg} }$ is $ \overline{G}=G/ G_{\beta \operatorname{refl} } $. Due to the possible presence of reflections, one cannot reconstruct the group $G$ or the action $\beta$ from $X$ alone without any further condition.  

By \cite[Th\'{e}or\`{e}me 5.1]{sga1}, for any scheme $X'$ of finite type over $\C$ such that $X'_\C$ has a finite fundamental group, there exists a finite \'{e}tale morphism $Y' \rightarrow X'$ such that $Y'$ is simply connected and such that $\pi_1(X')$ acts on $Y'$ with quotient $X'$. When $X' \subseteq X$ is open and $X =\Spec R$ is a normal domain, then this can be extended  to a finite morphism $Y \rightarrow X$ by going to the normalization of $X$ inside $Q(Y')$. The group $\pi_1(X')$ also acts on $Y$ with quotient $X$. It seems to be difficult to characterize, starting from $X' =X_{\operatorname{reg} }\subseteq \Spec R$, whether $Y$ is an affine space. However, if we start with a group $G$ acting linearly and faithfully on ${\mathbb A}^d$ with quotient $X$, then this construction gives back ${\mathbb A}^d/ G_{\beta \operatorname{refl} } $ together with the action of $G/ G_{\beta \operatorname{refl} } $ on it, as was shown more generally in Lemma \ref{smallbundleetale}.

If a module scheme $p:Z \rightarrow X$ up to normalization is given, then we can use the fundamental group of the regular locus of $X$ and of the regular locus of $Z$ to perform a similar construction, which, if the data are of the form ${\mathbb A}^{d+m}/ \beta \times \rho \rightarrow {\mathbb A}^d/\beta$, should return these data as well as possible. The  existence of reflections of $\beta \times \rho$ provides a natural obstruction for a complete reconstruction, but according to Section \ref{reflectionssection}, there are many favorable situations, such as Corollary \ref{faithfulsmall} and Corollary \ref{smallsmall}, where the product representation has no reflection.

\begin{Lemma}
\label{fiberflatbundleproperties}
Let $X =\Spec R$ be a normal affine scheme of finite type over $\C$. Let $p:Z \rightarrow X$ be a normal fiberflat bundle. Then the following hold.

\begin{enumerate}
	
\item
For every open subset $U \subseteq X$, there is a natural surjection $\pi_1(Z|_U) \rightarrow \pi_1 (U)$.
	
\item 
There is a natural isomorphism $ \pi_1 ( Z_{\operatorname{reg} } \cap p^{-1} ( X_{\operatorname{reg} } ) ) \rightarrow \pi_1 (Z_{\operatorname{reg} }) $.
	
\item 
There is a natural surjection $ \pi_1 ( Z_{\operatorname{reg} } ) \rightarrow  \pi_1 ( X_{\operatorname{reg} } )$.	

\end{enumerate} 
\end{Lemma}
\begin{proof}
(1) is clear due to the existence of the zero section. (2) The inclusion $Z_{\operatorname{reg} } \cap p^{-1} ( X_{\operatorname{reg} } ) \subseteq Z_{\operatorname{reg} } $
is an isomorphism in codimension one, due to the condition on the fibers and the normality of $X$; hence, these spaces have the same fundamental group by \cite[Satz 4.B.2]{storchwiebe4}. (3). The morphism $p^{-1}( X_{\operatorname{reg} } ) \rightarrow  X_{\operatorname{reg} }$ induces a surjection $\pi_1 ( p^{-1}( X_{\operatorname{reg} } )) \rightarrow  \pi_1 ( X_{\operatorname{reg} })$ by part (1). The inclusion  $	Z_{\operatorname{reg} } \cap p^{-1} ( X_{\operatorname{reg} } ) \subseteq	p^{-1} ( X_{\operatorname{reg} } ) $ defines a surjection $ \pi_1 ( Z_{\operatorname{reg} } \cap p^{-1} ( X_{\operatorname{reg} } ) ) \rightarrow  \pi_1 ( p^{-1} ( X_{\operatorname{reg} } ) ) $ by  \cite[Satz 4.B.2]{storchwiebe4}. Hence, (2) gives the result.	
\end{proof}

In view of Theorem \ref{moduleschemefundamentalgroup}, it is natural to impose the conditions that the fundamental groups of the regular loci are finite. 

\begin{Lemma}
	\label{reconstructionlemma}
Let $X =\Spec R$ be a normal affine scheme of finite type over $\C$. Let $p:Z \rightarrow X$ be a normal fiberflat bundle such that $ \pi_1 ( Z_{\operatorname{reg} } ) $ is finite. Then there exists a finite morphism $X' \rightarrow X$ such that the regular locus of the normalization of $Z'=X' \times_XZ$ has a trivial fundamental group.
\end{Lemma}
\begin{proof}
Let $U \subseteq X_{\operatorname{reg} } $ be nonempty such that $ Z|_U $ is a vector bundle. We look at the commutative diagram
\[ \begin{matrix}  Z|_{U} & \longrightarrow & Z_{\operatorname{reg} } \cap p^{-1} ( X_{\operatorname{reg} } ) &  {\longrightarrow} &    Z_{\operatorname{reg} }  &  {\longrightarrow} & Z \\ \downarrow & & \downarrow & & \downarrow & & \downarrow \\U & \longrightarrow & X_{\operatorname{reg} } & \longrightarrow & X & \longrightarrow & X \, .  \end{matrix}  \]
Then, $\pi_1(U)$ is isomorphic to $\pi_1 ( Z|_U)$ and $ \pi_1( Z_{\operatorname{reg} } \cap p^{-1} ( X_{\operatorname{reg} } ))$ is isomorphic to $ \pi_1 ( Z_{\operatorname{reg} }  )$ by Lemma \ref{fiberflatbundleproperties}. The inclusions $ U \subseteq X_{\operatorname{reg} }$ and $ Z|_U \subseteq Z_{\operatorname{reg} }$ define a surjective homomorphism $ \pi_1( Z|_U) \rightarrow \pi_1( Z_{\operatorname{reg} }) = H $ by \cite[Satz 4.B.2]{storchwiebe4}. Let $I$ be the kernel of the second homomorphism, considered as a subgroup of $\pi_1( U) \cong \pi_1( Z|_U) $.

There exists a finite \'etale morphism $U' \rightarrow U$ with $U'$ connected such that $H$ acts on $U'$ with quotient $U$. One can look at the analytic simply connected universal covering space $\tilde{U} \rightarrow U$, where $\pi_1(U)$ acts on by deck transformations. The normal subgroup $I \subseteq \pi_1(U)$ acts on $\tilde{U}$ as well with quotient space $U'= \tilde{U}/I \rightarrow U$, with $H$ acting on $U'$ and quotient $U$. By the existence theorem of Grothendieck-Riemann, $U'$ has the structure of an algebraic scheme, and $U' \rightarrow U$ is a finite morphism.

Let $L$ be the quotient field of $U'$ and let $X'$ be the spectrum of the integral closure of $R$ in $L$. The action of $H$ on $U'$ extends to an action of $H$ on $X'$ with quotient $X$.

Set $Z'=X' \times_X Z$. This contains the vector bundle $Z'|_{U'} = U' \times_U Z$ as an integral open subscheme. Let $\tilde{Z}$ be the normalization of the integral component of $Z'$ containing $Z'|_{U'}$, which is a fiberflat bundle over $X'$ by Lemma \ref{fiberflatconstruction} (4). $H$ is also acting on $Z'$, on  $\tilde{Z}$ and on $ Z'|_{U'}  $. The quotient of $Z'$ is $Z$, and this is also the quotient of $\tilde{Z}$, since $ \tilde{Z}/G \rightarrow Z$ is finite and birational and $Z$ is normal. We have the inclusions ($\varphi: \tilde{Z} \rightarrow Z' \rightarrow Z$ finite)
\[ Z'|_{U'} \subseteq  \tilde{Z}_{\operatorname{reg} } \cap \varphi^{-1} (Z_{\operatorname{reg} } ) \subseteq \tilde{Z}_{\operatorname{reg} }   \, ,  \] 
where the inclusion on the right is an isomorphism in codimension one, by the finiteness of $\varphi$ and the normality of $Z$. Hence, the upper right arrow in the commutative diagram
\[  \begin{matrix} \pi(U') &=& \pi_1(Z'|_{U'})& \longrightarrow & \pi_1 (\tilde{Z}_{\operatorname{reg} } \cap \varphi^{-1} (Z_{\operatorname{reg} } ) )  & \stackrel{\cong}{\longrightarrow} & \pi_1 ( \tilde{Z}_{\operatorname{reg}} ) \\ \downarrow && \downarrow & & \downarrow & & \\\pi(U) &=& \pi_1(Z|_U) & \longrightarrow & \pi_1 (Z_{\operatorname{reg}} ) & & \,  \end{matrix}  \]
is an isomorphism. The horizontal homomorphisms in the middle are surjective.

Let $W \rightarrow Z_{\operatorname{reg} }$ be the finite universal covering space. Then, we can compare $ Z'|_{U'} $ and $W|_{Z|_U}$ over $Z|_U$. Both are covering spaces of the same degree $|H|$, $H$ acts on both with $Z|_U$ as quotient, and the fibers are in bijection to $H$ with $H$ acting on itself. Hence, there exists a (not unique) isomorphism $ Z'|_{U'}  \rightarrow  W|_{Z|_U}  $ above $Z|_U$. With this identification, $W$ and $ \tilde{Z}$ have the same quotient field, and $W$ and $\varphi^{-1}(Z_{\operatorname{reg}})$ are given by the integral closure of $Z_{\operatorname{reg}}$ in the quotient field. Hence, they coincide, $  \varphi^{-1} ( Z_{\operatorname{reg} }) \rightarrow  Z_{\operatorname{reg} } $ is a covering space and $\pi_1(\varphi^{-1}(Z_{\operatorname{reg}}) = \pi_1 (W)=0$. Also, $ \varphi^{-1}(Z_{\operatorname{reg}})$ is smooth, and so the intersection above is just $ \varphi^{-1} ( Z_{\operatorname{reg}}) $. Therefore, the fundamental group of $\tilde{Z}_{\operatorname{reg} }$ is trivial.
\end{proof}

\begin{Example}
Let $X={\mathbb A}^1$ with $Z \rightarrow X$ a fiberflat bundle that is a vector bundle above $U=X \setminus \{ 0 \} $. The fundamental group of $U$ is $\Z$, and the fundamental group of $Z_{\operatorname{reg} }$ is $\Z/(k)$,  as we have the natural surjection $ \pi_1 (Z|_U) \cong \Z \rightarrow \pi ( Z_{\operatorname{reg} }   ) \cong \Z/(k) $. Then, the construction in Lemma \ref{reconstructionlemma} yields $U'= {\mathbb A}^1 \setminus  \{0\} $ with the morphism $z \mapsto z^k$ and the corresponding extension to ${\mathbb A}^1$. 
	
	If $Z=K[X,W]^{\rho_\ell}$  with $\ell$ coprime to $k$ (else the fundamental group of the regular locus of $Z$ would not be $\Z/(k)$, see Example \ref{cycliconefundamentalgroup}), then the pull-back is $ Z'= \Spec K[X, X^iW^j,\, i + \ell j \in \Z (k)]$ and its normalization is ${\mathbb A}^2$, the preimage of $Z_{\operatorname{reg} }$ is the punctured affine plane, which is simply connected.
\end{Example}

\begin{Theorem}
	\label{reconstruction}
	Let $X$ be a normal affine scheme of finite type over $\C$ and let $p:Z \rightarrow X$ be a normal fiberflat bundle. Suppose that  the fundamental group $ H=\pi_1  ( Z_{\operatorname{reg} } )  $  is finite. Then, there exists a commutative diagram
	\[  \begin{matrix} Y & \longleftarrow & W \\ \downarrow  & & \downarrow \\ 	X &\longleftarrow & Z \, , \end{matrix} \]
	where
	\begin{enumerate}
		\item 
		$Y \rightarrow X$ is a normal scheme finite over $X$, with an action of $H $ on $Y$ with quotient $X$.
		
		\item
		$W \rightarrow Y$ is a normal fiberflat bundle, finite over $Z$ and \'etale in codimension one, with an action of $ H $ on $W$ with quotient $Z$.
		
		\item
	The actions of $H$ on $W$ and $Y$ are compatible.
		
		\item 
		$Y_{\operatorname{reg} }$ and $W_{\operatorname{reg} }$ are simply connected.
		
		\item 
		The corresponding diagram with the zero sections also commutes.
\end{enumerate}
	
	The objects $W \rightarrow Y$ in the diagram are uniquely determined up to an isomorphism over the base $Z \rightarrow X$.
\end{Theorem}

\begin{proof}
Let $Y= X'$ and $W = \tilde{Z}$ as constructed in Lemma \ref{reconstructionlemma}, which gives the commutative diagram and the second part of (4). The proof of Lemma \ref{reconstructionlemma} shows that (1), (2), (3) hold. By Lemma \ref{fiberflatbundleproperties}, $Y_{\operatorname{reg}}$ is also simply connected. The inclusion $Y \rightarrow W$ is compatible with the actions of $H$, and going to the quotients gives back $X \rightarrow Z$. $W$ is uniquely determined by $Z$, as it is normal and as it contains the finite universal covering space of $Z_{\operatorname{reg}} $. The uniqueness of $Y$ follows from the commutativity of the zero sections.
\end{proof}

Note that the morphism $Y \rightarrow X$ is in general not \'etale in codimension one.

\begin{Remark}
If $X$ has an isolated singularity with a trivial local fundamental group and $Z$ is a vector bundle over the regular locus, then Theorem \ref{reconstruction} does not do anything.	
\end{Remark}

\begin{Theorem}
\label{reconstructpullback}
Let $K=\C$. Let $G$ be a finite group, let $\beta$ be a faithful $\C$-linear representation of $G$ of dimension $d$ with quotient scheme $X $ and let $\rho$ be another linear representation of $G$ with corresponding quotient scheme $p:Z_\rho \rightarrow X$. Suppose that the product representation $\beta \times \rho$ is small. Then, the construction of Theorem \ref{reconstruction} yields back $G$ together with the actions $\beta$ and $\rho$ on ${\mathbb A}^d$ and  ${\mathbb A}^m$.
\end{Theorem}
\begin{proof}
The condition means that the subgroup  $ G_{\beta \times \rho \operatorname{refl} }  $  generated by the reflections for the product representation $\beta \times \rho$ is trivial. By Lemma
\ref{fundamentalgroup} and Theorem \ref{moduleschemefundamentalgroup}, 
we have the natural surjective group homomorphism
\[  G =\pi_1  ( Z_{\operatorname{reg} } ) \cong  \pi_1 (    Z_{\operatorname{reg} } \cap p^{-1} (X_{\operatorname{reg} } )  )  \longrightarrow \pi_1 ( X_{\operatorname{reg} } ) = G/ G_{\beta \operatorname{refl} } \, .\]
In particular, all conditions to apply Theorem \ref{reconstruction} are fulfilled. We look at the quotient map $\varphi: Y \rightarrow X=Y/G$. 

Let $U \subseteq X_{\operatorname{reg} }$ be an open nonempty subset
where $Z_\rho$ is a vector bundle. The morphism $ \varphi^{-1}(U) \rightarrow U$ is \'etale by Lemma \ref{smallbundleetale}. This means that $ \varphi^{-1}(U) \rightarrow U $ is a covering space with acting group $G$ and with a surjection $ \pi_1 (U) \rightarrow G \rightarrow 0 $. In the construction of Lemma \ref{reconstructionlemma}, we construct $U'$ starting with these data. Therefore, $U'= \varphi^{-1}(U)$ and then also $X'=Y$, as this is the integral closure of $X$ in the quotient field of $U'$. Therefore, the claim follows from uniqueness in Theorem \ref{pullback}.
\end{proof}

\section{Questions}
\label{questionssection}

We close with some questions (Question \ref{multilinearquestion}, Question \ref{positivequestion} and Question \ref{analyticquestion} are rather projects).

\begin{Question}
What is the most general situation where reconstruction in the sense of Theorem \ref{pullback} and {Corollary} \ref{correspondence} holds? By Remark \ref{notreducednoreconstruction}, one needs to assume that $Y$ is reduced, and for the equivalence between (1) and (2) in {Corollary} \ref{correspondence}, one needs further conditions.
\end{Question}

\begin{Question}
More specifically, suppose that $G$ is a reflection group acting faithfully and linearly on ${\mathbb A}^d$ with mirror hyperplanes $H_1, \ldots, H_n$, $n \geq 2$. Let $Y= \bigcup_{i \in I} H_i$ with the induced (faithful) action. Can one reconstruct a linear representation $\rho$ from the quotient scheme $Z_\rho \rightarrow Y/G$?
\end{Question}

\begin{Question}
What is the exact relation between the general correspondence developed here and the correspondence for (true) reflection groups developed in \cite[Theorem D]{buchweitzfaberingalls}?
\end{Question}

\begin{Question}
Is it true as discussed in Remark \ref{regularfiberflatconjecture} that for a regular ring, a fiberflat bundle is free?
\end{Question}

\begin{Question}
Among the maximal Cohen-Macaulay modules of an invariant ring of a linear action, is it true that the ones coming from a linear representation
are characterized by fiberflatnesss?
\end{Question}

\begin{Question}
What are the rings with only finitely many reflexive fiberflat bundles?
\end{Question}

\begin{Question}
For an isolated singularity $\Spec R$, when has the tangent bundle of the punctured spectrum a fiberflat extension?
\end{Question}

\begin{Question}
When is the Frobenius pushforward $F_*R$ for a normal $F$-finite ring $R$ in positive characteristic reflexive and fiberflat? Are there other examples beyond invariant rings for finite groups?
\end{Question}

\begin{Question}
Do the results of Section \ref{irreduciblesection} also hold without the reflexive and smallness assumption?
\end{Question}

\begin{Question}
	\label{multilinearquestion}
	What does the multilinear theory of module schemes up to normalization look like? It should be such that it respects the correspondence in Corollary \ref{correspondence}.
\end{Question}

\begin{Question}
	\label{positivequestion}
	In Sections \ref{fundamentalgroupsection} and \ref{trivializationsection}, we worked over the complex numbers and with the topological fundamental group. What do the corresponding results in positive characteristic look like? 
\end{Question}

\begin{Question}
Suppose $X=Y/G$ and let $\pi:\tilde{X} \rightarrow X$ be a resolution of singularities. What is the correct pull-back of a quotient scheme $Z_\rho$ over $X$ to $\tilde{X}$? As $X$ might be smooth, any reasonable approach to this question will have to take the singularities of $Z_\rho$ into account.
\end{Question}

\begin{Question}
What does hold when $G$ is not a finite, but a reductive group?
\end{Question}

\begin{Question}
\label{analyticquestion}
In the analytic context, what does the theory look like if we replace the scheme $Y$ with a (normal connected) complex space?
\end{Question}

\begin{Question}
What about mixed characteristic?
\end{Question}

\bibliography{bibfile}

\bibliographystyle{plain}

\end{document}